\documentclass[reqno, 11pt, a4paper]{amsart}
\usepackage{amsmath, amssymb, amsthm}  

\usepackage{appendix}
\usepackage{graphicx}
\usepackage[table]{xcolor}
\usepackage{enumitem}
\usepackage{fancyhdr}
\usepackage{mathrsfs}
\usepackage[truedimen,top=3truecm, bottom=3truecm, left=2.5truecm, right=2.5truecm, includefoot]{geometry}
\usepackage{dsfont} 
\usepackage{pifont}

\usepackage{mathtools}
\usepackage{mathdots}

\usepackage{tikz}
\usetikzlibrary{shapes,arrows,positioning,fit,backgrounds,calc}
\usepackage{pgfplots}

\usepackage{diagbox}
\usepackage{makecell}

\usepackage{aliascnt}
\usepackage{hyperref}
\usepackage[capitalise,noabbrev,nameinlink]{cleveref}

\hypersetup{
  colorlinks=false,
  pdfborder={0 0 1},
  linkbordercolor={1 0 0},
  citebordercolor={0 1 0}
}

\theoremstyle{plain}
\newtheorem{theorem}{Theorem}

\newaliascnt{proposition}{theorem}
\newtheorem{proposition}[proposition]{Proposition}
\aliascntresetthe{proposition}

\newaliascnt{lemma}{theorem}
\newtheorem{lemma}[lemma]{Lemma}
\aliascntresetthe{lemma}

\newaliascnt{corollary}{theorem}
\newtheorem{corollary}[corollary]{Corollary}
\aliascntresetthe{corollary}

\theoremstyle{definition}

\newaliascnt{definition}{theorem}
\newtheorem{definition}[definition]{Definition}
\aliascntresetthe{definition}

\theoremstyle{remark}

\newaliascnt{remark}{theorem}
\newtheorem{remark}[remark]{Remark}
\aliascntresetthe{remark}

\crefname{theorem}{Theorem}{Theorems}
\crefname{proposition}{Proposition}{Propositions}
\crefname{lemma}{Lemma}{Lemmas}
\crefname{corollary}{Corollary}{Corollaries}
\crefname{definition}{Definition}{Definitions}
\crefname{remark}{Remark}{Remarks}
\crefname{appendix}{Appendix}{Appendices}

\newcommand{\N}{\mathbb{N}}
\newcommand{\Z}{\mathbb{Z}}
\newcommand{\R}{\mathbb{R}}

\newcommand{\p}{\mathbb{P}}
\newcommand{\E}{\mathbb{E}}
\newcommand{\ind}{\mathds{1}}
\newcommand{\diff}{\,\mathrm{d}}
\renewcommand{\bar}{\overline}
\newcommand{\schwartz}{\mathscr{S}(\R )}
\newcommand{\schwartzprime}{\mathscr{S}'(\R )}

\newcommand{\modif}[1]{\textcolor{red}{#1}}

\title[Stationary fluctuations for an exclusion process with two conservation laws]{Stationary fluctuations for an exclusion process with mass and energy conservation}

\author[H. Da Cunha]{Hugo Da Cunha}
\address{Hugo Da Cunha -- Université Lyon 1, Centrale Lyon, INSA Lyon, Université Jean Monnet, CNRS, ICJ UMR5208, 69622 Villeurbanne, France and JSPS International Research Fellow at Graduate School of Mathematical Sciences, The University of Tokyo, 3-8-1 Komaba, Meguro-ku, Tokyo 153-8914, Japan}
\email{\href{mailto:dacunha@math.univ-lyon1.fr}{dacunha@math.univ-lyon1.fr}}

\author[M. Sasada]{Makiko Sasada}
\address{Makiko Sasada -- Graduate School of Mathematical Sciences, The University of Tokyo, 3-8-1 Komaba, Meguro-ku, Tokyo 153-8914, Japan}
\email{\href{mailto:sasada@ms.u-tokyo.ac.jp}{sasada@ms.u-tokyo.ac.jp}}


\begin{document}
\maketitle

\begin{abstract}
We introduce a novel exclusion process with two conservation laws, mass and energy, designed to mimic the essential features of continuous systems like interacting oscillators within the framework of interacting particle systems. This distinguishes our model from conventional multi-species processes where only particle numbers are conserved. As a basis for our fluctuation analysis, we first show that applying nonlinear fluctuating hydrodynamics (NFH) to this model reveals a wide variety of universality classes depending on the parameter choices.
    The main objective of this work is to study the stationary fluctuations of these conserved quantities.
    For a suitable choice of parameters, we rigorously show that the fluctuation fields converge to uncoupled stochastic Burgers equations (SBE) in the scaling limit. The proof relies on the second-order Boltzmann-Gibbs principle that we establish for this model, along with the spectral gap estimate and the equivalence of ensembles. Of independent interest is our general proof of the diagonalizability of the Jacobian matrix for the macroscopic current with distinct real eigenvalues. While this property is often taken as given in the physics literature, we establish it rigorously for multi-component systems even when the eigenvectors cannot be explicitly computed, offering a firm mathematical foundation for a broad class of models. 
\end{abstract}

\section{Introduction}

One of the main questions in statistical mechanics is to understand how macroscopic evolution equations arise from microscopic dynamics. In the context of interacting particle systems, this question is addressed through the so-called \emph{hydrodynamic limit}, which aims at deriving deterministic partial differential equations (PDEs) describing the evolution of the density of some conserved quantities in the system, after some space-time rescaling and when the number of particles goes to infinity. The limit being deterministic, a hydrodynamic limit corresponds to a law of large numbers for such systems. A natural next step is to study the fluctuations around this typical behaviour, \textit{i.e.}~to establish a central limit theorem, usually deriving a convergence towards the solution of a stochastic partial differential equation (SPDE). In the case of a single conservation law, this program has been widely studied and is now well-understood for a large class of models. However, when there are several conservation laws, things get more complicated as it is not clear which fluctuation fields one should consider. In particular, the fluctuation fields associated to the different conserved quantities might be considered in diverse timescales, and different limiting behaviours can emerge. 

This is the question we tackle in this paper, by introducing the Exclusion Process with Energy~(EPE), a new model of exclusion process in which particles carry energy and hence displays two conserved quantities: the number of particles and the total energy. This model is inspired by the one introduced in \cite{nagahata_fluctuation_2003}, where particles that have an integer-valued energy between~1 and some~$\kappa\ge 2$ can hop to its neighbouring sites on a discrete lattice, and can also exchange energy with each other when they are located at neighbouring sites. The main difference and novelty with respect to the model of \cite{nagahata_fluctuation_2003} is that we choose here the rates of the dynamics so that it belongs to the class of \emph{gradient models}, which makes the analysis more tractable.
A particular feature of this model is also that, on the contrary to previous multi-component lattice gas models, there is no symmetry between the two conservation laws. The EPE can somehow be seen as a discrete version of the \emph{Bernardin-Stoltz model}, a model of interacting oscillators with two conservation laws (volume and energy) introduced in \cite{bernardin_anomalous_2012} and further studied in \cite{goncalves_derivation_2023,goncalves_stochastic_2025,spohn_nonlinear_2015}. Moreover, setting $\kappa =2$ in the EPE, 
one recovers exactly the ABC model, which is a two-species exclusion process, whose stationary fluctuations have been studied in \cite{cannizzaro_abc_2025}. As we discuss later, the predictions via \emph{nonlinear fluctuating hydrodynamics} (NFH) suggest that this model is richer than the ABC model, in the sense that a broader variety of universality classes can be observed depending on the choice of the parameters. 

NFH is a theoretical framework developed in the physics literature and systematically formulated by Spohn~\cite{spohn2014nonlinear}, which provides predictions for the stochastic differential equations arising as scaling limits of the fluctuations of conserved quantities in multi-component systems. The starting point is the hydrodynamic limit of such system in the strongly asymmetric version of the dynamics, which is usually a system of $n$ conservation laws of the form
\begin{equation*}
    \partial_t\vec{\rho} + \nabla\mathbf{j}(\vec{\rho} ) = 0,
\end{equation*}
where $\vec{\rho}$ is the vector containing the $n$ macroscopic conserved quantities depending on both time and space, and $\mathbf{j}$ is the macroscopic current vector. The next step is to linearize this system around its stationary state $\bar{\rho}$, \textit{i.e.}~to write $\vec{\rho}= \bar{\rho}+ \vec{\mathcal{Y}}$, and to expand the current $\mathbf{j}$ to second order around $\bar{\rho}$, neglecting higher-order terms. Adding a diffusion term $\Delta\vec{\mathcal{Y}}$, and a noise term~$\nabla\dot{\mathcal{W}}$ to account for microscopic fluctuations, one gets the following system of SPDEs for the~$i$-th coordinate of the perturbation~$\vec{\mathcal{Y}}$:
\begin{equation*}
    \partial_t \mathcal{Y}_i +\nabla \left( (\mathbf{J}(\bar{\rho})\vec{\mathcal{Y}})_i + \frac12  \vec{\mathcal{Y}}^\top \mathbf{H}^i(\bar{\rho}) \vec{\mathcal{Y}}\right)  -\Delta \mathcal{Y}_i+ \nabla\dot{\mathcal{W}_i} = 0
\end{equation*}
where $\mathbf{J}$ is the Jacobian matrix of $\mathbf{j}$, and $\mathbf{H}^i$ is the Hessian matrix of the $i$-th component of~$\mathbf{j}$. The idea of NFH is then to diagonalize the linear part of this system, by considering the normal modes $\vec{\phi} = R\vec{\mathcal{Y}}$, where $R$ is the matrix that diagonalizes $\mathbf{J}(\bar{\rho})$, \textit{i.e.}~$R\mathbf{J}(\bar{\rho})R^{-1}=\mathrm{diag}(v_i)$. The~$i$-th normal mode then evolves according to the equation
\begin{equation*}
    \partial_t\phi_i + v_i\nabla\phi_i + \nabla\big(\vec{\phi}^\top\mathbf{G}^i\vec{\phi }\big) - \Delta\phi_i + \nabla(R\dot{\mathcal{W}})_i =0,
\end{equation*}
where $\mathbf{G}^i$ is the \emph{coupling matrix} defined by
\begin{equation*}
    \mathbf{G}^i = \frac12\sum_{j=1}^n R_{ij}(R^{-1})^\top \mathbf{H}^jR^{-1}.
\end{equation*}
Depending on whether some diagonal entries of the coupling matrices $\mathbf{G}^i$ vanish or not, NFH predicts the emergence of different universality classes for the fluctuations of each normal mode. In \cite{spohn_nonlinear_2015}, a complete classification is proposed for the case of two conservation laws and it is predicted that only five possible behaviours can emerge: diffusive, KPZ and $\alpha$-Lévy for $\alpha$ being either~$\frac32$,~$\frac53$ or the golden ratio. We summarise these predictions in pairs in the following \cref{table:classification}. 
\begin{table}
\begin{tabular}{c|c|c|c|c}
    \diagbox{$\mathbf{G}^2$}{$\mathbf{G}^1$} & \makecell{$\begin{pmatrix}
        * & \\
         & * \end{pmatrix}$} & \makecell{$\begin{pmatrix}
        0 & \\
         & * \end{pmatrix}$} & \makecell{$\begin{pmatrix}
        * & \\
         & 0 \end{pmatrix}$} & \makecell{$\begin{pmatrix}
        0 & \\
         & 0 \end{pmatrix}$} \\
         \hline
        \makecell{$\begin{pmatrix}
        * & \\
         & * \end{pmatrix}$} & \cellcolor{red!15} (KPZ,KPZ) & \cellcolor{blue!15}($\frac53$-Lévy,KPZ) & (KPZ,KPZ) & (diff,mod-KPZ) \\
         \hline
         \makecell{$\begin{pmatrix}
        * & \\
         & 0 \end{pmatrix}$} & \cellcolor{blue!15}(KPZ,$\frac53$-Lévy) & (Gold,Gold) & (KPZ,$\frac53$-Lévy) & (diff,$\frac32$-Lévy) \\
         \hline 
         \makecell{$\begin{pmatrix}
        0 & \\
         & * \end{pmatrix}$} & (KPZ,KPZ) & ($\frac53$-Lévy,KPZ) & \cellcolor{green!15}(KPZ,KPZ) & \cellcolor{green!15}\textcolor{orange}{(diff,KPZ)}\\
         \hline
         \makecell{$\begin{pmatrix}
        0 & \\
         & 0 \end{pmatrix}$} & (mod-KPZ,diff) & ($\frac32$-Lévy,diff) & \cellcolor{green!15} (KPZ,diff) & \cellcolor{green!15}(diff,diff) 
\end{tabular}
\vspace{.5cm}
\caption{Classification of the predictions of NLFH for systems with two conservation laws. Here, ``diff'' stands for diffusive behaviour, ``KPZ'' for Kardar-Parisi-Zhang behaviour, ``mod-KPZ'' for modified KPZ behaviour, and ``Gold'' for golden-ratio Lévy behaviour. An entry $*$ means that the corresponding coefficient of the coupling matrix is non-zero. This table is adapted from \cite{spohn_nonlinear_2015}.}
\label{table:classification}
\end{table}
The computation of the coupling matrices for the EPE can be found in \cref{appendix:couplingmatrices}, and we can see at first sight that some cancellations occur when $\kappa=2$, recovering the case of the ABC model for which the computation of the coupling matrices has been carried out in \cite{cannizzaro_abc_2025} (and for which the predictions are depicted in the green shaded cells of \cref{table:classification}). However, when $\kappa >2$ the matrices are more complicated and we can expect wider variety of universality classes. In the most general case where none of the diagonal terms of the coupling matrices vanish, NFH predicts that both modes should display a KPZ behaviour (in red in \cref{table:classification}). We also prove in \cref{prop:vanishingG11} of \cref{appendix:couplingmatrices} that for some choice of parameters, the coefficient $\mathbf{G}_{ii}^i$ can be made to vanish so that we fall in one of the purplish-blue  cells of \cref{table:classification} in which one mode is KPZ whereas the other one displays $\frac53$-Lévy behaviour. For some simpler choice of parameters, one can also fall the green cell with orange text where one mode is diffusive whereas the other one exhibits KPZ behaviour (\textit{cf.} \cref{prop:couplingmatricesalpha_e0}). While a $\frac53$-Lévy behaviour has been rigorously derived in other contexts, such as interacting charged harmonic oscillators under a magnetic field with stochastic noise~\cite{SaitoSasadaSuda2019}, the mechanism in the lattice gas with multiple conservation laws seems different, and a rigorous derivation in this setting remains an open problem. Hence, this case is particularly interesting and will be addressed in future work. The focus of the present paper is rather on the former case, where both modes are expected to display KPZ behaviour. 

Above, we have repeatedly referred to ``KPZ behaviour''; let us elaborate further on this topic. The Kardar-Parisi-Zhang (KPZ) equation was first introduced in \cite{kardar_dynamic_1986} as a phenomenological model for the growth of random interfaces. In a unidimensional setting, it is a nonlinear SDE that reads
\begin{equation*}
    \partial_t\mathfrak{h} = \nu\partial_x^2\mathfrak{h} + \lambda (\partial_x\mathfrak{h})^2 + \sqrt{D}\dot{W},
\end{equation*}
for some parameters $\nu,D >0$ and $\lambda\in\R$, where $t\in\R_+$, $x\in\R$ and $\dot{W}$ is a one-dimensional space-time white noise. This equation is ill-posed due to the nonlinearity and the roughness of the noise, but setting formally $\mathfrak{u}=\partial_x\mathfrak{h}$ leads to the \emph{stochastic Burgers equation} (SBE) for $\mathfrak{u}$, namely
\begin{equation*}
    \partial_t\mathfrak{u}=\nu\partial_x^2\mathfrak{u} + \lambda \partial_x(\mathfrak{u}^2)+ \sqrt{D}\partial_x\dot{W}
\end{equation*}
for which a whole theory of existence and uniqueness of \emph{energy solutions} has been developed in \cite{goncalves_nonlinear_2014,gubinelli_energy_2018,gubinelli_infinitesimal_2020}. These equations are quite universal as they have been shown to arise as scaling limits of a wide variety of weakly asymmetric one-component models, see for instance \cite{BertiniGiacomin1997,goncalves_nonlinear_2014,GoncalvesJaraSethuraman2015}. 

In this paper, we derive two uncoupled SBEs as the scaling limit of the equilibrium fluctuation fields of the EPE (like in \cite{hayashi_nonlinear_2025} for oscillator chains). The main tool to derive the KPZ behaviour is the so-called \emph{second-order Boltzmann-Gibbs principle} for this two-component system, which allows us to replace local functions by a second-order expansion in terms of the locally averaged conserved quantities of the system. Proving this principle is itself a technical challenge, and this requires in particular to establish a \emph{spectral gap estimate} of this Markov process, together with some \emph{equivalence of ensembles} for the stationary measures. All these results are proved in the present paper for this new model, and will undoubtedly be of great interest for future work on this model, in particular to tackle the more challenging case where one mode displays KPZ behaviour while the other exhibits $\frac53$-Lévy fluctuations.

We further provide a quite general proof that, for models sharing similar features, the Jacobian matrix of the macroscopic current vector is always diagonalizable with real eigenvalues, which is a crucial property to define the normal modes, and that moreover both modes satisfy some orthogonality relation that allows us to get independence between them in the macroscopic limit.

\subsection*{Outline of the paper}

The paper is organized as follows. In \cref{sec:model}, we properly introduce the model we are interested in, namely the exclusion process with energy. We define the dynamics in both symmetric and weakly asymmetric versions, and detail about its stationary measures. We also define therein the notion of energy solutions to the SBE. Finally, we state and prove \cref{thm:diagonalizability} about the diagonalizability of the Jacobian, and state the results about the fluctuations of the EPE, namely \cref{thm:main} and \cref{thm:main2}. Then, we prove equivalence of ensembles for the family of aforementioned stationary measures in \cref{sec:equivalenceofensembles}, and establish a spectral gap estimate for the dynamics in \cref{sec:spectralestimates}. These two results are the key ingredients to prove the second-order Boltzmann-Gibbs principle in \cref{sec:2BG-Principle}. With these results at hand, we can proceed with the proof of the main theorem in the remaining of the article, which relies on usual procedures to prove convergence of stochastic processes. In \cref{sec:martingaledecomposition}, we introduce the martingale decomposition of the fluctuation fields, and use the second-order Boltzmann-Gibbs principle to rewrite the different terms appearing therein. Then, we prove tightness of the sequence of fluctuation fields in \cref{sec:tightness}, and finally identify the limit points as energy solutions to the SBE  in \cref{sec:identification}. In \cref{appendix:couplingmatrices}, we collect the explicit computations of macroscopic current, its Jacobian, and the coupling matrices, and discuss the emerging universality classes predicted by NFH for the EPE.

\subsection*{General notation}

\begin{itemize}
    \item For $f,g\in L^2(\R )$, we denote by $\langle f,g\rangle_{L^2}$ their standard inner product, and by $\| f\|_{L^2}$ the associated norm. More generally, for any measure space $(X,\mathcal{A},\mu )$, $\langle \cdot ,\cdot\rangle_\mu$ denotes the inner product in $L^2(\mu )$.
    \item We denote by $\schwartz$ the Schwartz space of smooth functions whose derivatives of all orders are rapidly decreasing on $\R$, and by $\schwartzprime$ its topological dual, \textit{i.e.} the space of tempered distributions on $\R$.
    \item We denote by $\mathcal{D}([0,T],E)$ the Skorokhod space of càdlàg (\textit{i.e.}~right-continuous with left limits) functions from $[0,T]$ to a topological space $E$.  Similarly, we denote by $C([0,T],E)$ the space of continuous functions from $[0,T]$ to $E$.
    \item We denote by $C_c^\infty (\R )$ the space of smooth compactly supported functions on $\R$.
    \item  For a non-negative sequence $(u_k)_{k\in\N}$, we write $v_k=O(u_k)$ if $|v_k|\le Cu_k$ for some constant $C>0$ that is independent of $k$ but may depend on other parameters.
\end{itemize}

\section{Exclusion process with energy}
\label{sec:model}

\subsection{The model}

Let $N\ge 1$ be an integer scaling parameter that will tend to infinity, and let~$\kappa\ge 2$ be an integer. Consider an exclusion process on the one-dimensional lattice $\Z$, where each particle carries a certain amount of energy which has an integer value between $1$ and $\kappa$. In other words, we consider configurations of particles $\eta = (\eta_x)_{x\in\Z}$ belonging to $\Omega\coloneq\{ 0,1,\hdots ,\kappa\}^\Z$, where $\eta_x=0$ means that there is no particle at site $x$, whereas $\eta_x=\ell\ge 1$ means that there is one particle at site $x$ with energy $\ell$. Then, the dynamics is as follows:
\begin{itemize}
    \item If there is a particle at the site $x\in\Z$, and if the site $y\in\Z$ is empty, the particle can jump from $x$ to $y$ with a jump rate denoted by $c_{x\to y}^p (\eta )$, in which the particle's energy remains unchanged.
    \item If the sites $x$ and $y$ both carry particles, then one unit of energy can be transferred from the particle at $x$ to the particle at $y$ provided that the particle at the site $x$ has energy greater than~$1$ (particles are not allowed to disappear), and the one at the site $y$ has energy less than~$\kappa$ (so that it can receive energy). This energy transfer rate is denoted by $c_{x\to y}^e(\eta )$.
\end{itemize}
In what follows, we consider a model where particle jumps and energy transfers occur only between nearest-neighbor sites. Consider first the \emph{symmetric dynamics}, driven by the generator~$\mathcal{L}_S$ that acts on local functions\footnote{We say that a function $f:\Omega\longrightarrow\R$ is \emph{local} if it depends only on a finite number of coordinates.} $f:\Omega\longrightarrow\R$ as
\begin{equation}\label{def:symmetricgenerator}
    \mathcal{L}_Sf(\eta ) = \sum_{\substack {x,y\in\Z\\ |x-y|=1}} c_{x\to y}^p(\eta )\big[ f(\eta^{x,y})-f(\eta )\big] + \sum_{\substack {x,y\in\Z\\ |x-y|=1}} c_{x\to y}^e(\eta )\big[ f(\eta^{x\to y})-f(\eta )\big].
\end{equation}
Above, $\eta^{x,y}$ stands for the configuration obtained from $\eta$ after the exchange of the occupation variables $\eta_x$ and $\eta_y$, and $\eta^{x\to y}$ is the one obtained from $\eta$ after the transfer of one unit of energy from $x$ to $y$ (whenever well-defined), namely
\begin{equation}\label{eq:transformedconfig}
    \eta_z^{x,y}=\begin{cases}
        \eta_y & \mbox{ if }z=x,\\
        \eta_x & \mbox{ if }z=y,\\
        \eta_z & \mbox{ otherwise},
    \end{cases}
    \qquad\mbox{ and }\qquad \eta_z^{x\to y}=\begin{cases}
        \eta_x-1 & \mbox{ if }z=x,\\
        \eta_x+1 & \mbox{ if }z=y,\\
        \eta_z & \mbox{ otherwise}.
    \end{cases}
\end{equation}

To each configuration $\eta$, we associate an exclusion configuration $\xi=\xi (\eta )\in \{0,1\}^{\Z}$ defined by $\xi_x : = \mathbf{1}_{\{\eta_x\ge 1\}}$, which indicates the presence of a particle at each site $x$. Under this dynamics, the total number of particles and the total energy, given respectively by  $\sum_{x}\xi_x$ and $\sum_{x}\eta_x$, are conserved.
Moreover, these quantities satisfy the inequality
\begin{equation*}
    \sum_{x}\xi_x \le\sum_{x}\eta_x \le\kappa\sum_{x}\xi_x.
\end{equation*}
We choose the rates $c^p$ and $c^e$ such that the model belongs to the class of \emph{gradient models}. This means that the instantaneous particle and energy currents across an oriented bond $\{x,x+1\}$ can be expressed as the discrete gradient of some local functions. Specifically, we set the particle jump rate to be a constant $\kappa -1$, namely,
\begin{equation}\label{def:particlerate}
    c_{x\to x+1}^p(\eta ) := (\kappa -1)\xi_x(1-\xi_{x+1}).
\end{equation}
The energy transfer rate is defined by the formula
\begin{equation}\label{def:energyrate}
    c_{x\to x+1}^e(\eta ) := \xi_x\xi_{x+1} (\eta_x-1)(\kappa -\eta_{x+1}) = (\eta_x-\xi_x)(\kappa\xi_{x+1}-\eta_{x+1})
\end{equation}
where the last identity comes from the fact that $\eta_x\xi_x=\eta_x$. Notice that expression \eqref{def:energyrate} reflects the fact that an energy transfer cannot happen if the source particle has the minimal energy or if the target particle has the maximal energy. With this choice of rates, the instantaneous particle current~$j_{x, x+1}^p$ across the bond $\{x,x+1\}$, defined as
\begin{equation}\label{def:jp}
    j_{x,x+1}^p(\eta )\coloneq  c_{x\to x+1}^p(\eta )-c_{x+1\to x}^p(\eta ),
\end{equation}
reduces to the discrete gradient
\begin{equation}\label{eq:gradientjump}
    j_{x,x+1}^p(\eta ) = (\kappa -1)(\xi_x-\xi_{x+1}).
\end{equation}
Similarly, the instantaneous energy current, defined as
\begin{equation}\label{def:je}
    j_{x,x+1}^e(\eta )\coloneq  \eta_xc_{x\to x+1}^p(\eta )-\eta_{x+1}c_{x+1\to x}^p(\eta ) + c_{x\to x+1}^e(\eta )-c_{x+1\to x}^e(\eta ),
\end{equation}
satisfies
\begin{equation}\label{eq:gradientenergy}
    j_{x,x+1}^e(\eta ) = (\kappa -1)(\eta_x-\eta_{x+1}).
\end{equation}

\medskip

We now introduce an asymmetric version of the dynamics by considering the infinitesimal generator $\mathcal{L}\coloneq\mathcal{L}_S+N^{-\gamma}\mathcal{L}_A$, where $\gamma >  0$ and the antisymmetric part $\mathcal{L}_A$ acts on local functions~$f:\Omega\longrightarrow\R$ via 
\begin{multline}
    \mathcal{L}_Af(\eta ) = \frac12\sum_{\substack {x,y\in\Z\\ |x-y|=1}} (y-x) (\alpha_p+\alpha_e\eta_x)c_{x\to y}^p(\eta )\big[ f(\eta^{x,y})-f(\eta )\big]\\ + \frac12\sum_{\substack {x,y\in\Z\\ |x-y|=1}} (y-x)\alpha_ec_{x\to y}^e(\eta )\big[ f(\eta^{x\to y})-f(\eta )\big]
\end{multline}
for some parameters $(\alpha_p,\alpha_e)\in\R^2\setminus\{(0,0)\}$. 
In words, the asymmetry modifies the jump rates: with respect to before, a particle with energy $\ell$ jumps to the right with a factor~$1+\frac{\alpha_p+\alpha_e\ell}{2N^\gamma}$ and to the left with a factor $1-\frac{\alpha_p+\alpha_e\ell}{2N^\gamma}$, while energy transfers are biased by a factor~$1 + \frac{\alpha_e}{2N^\gamma}$ to the right and~$1 - \frac{\alpha_e}{2N^\gamma}$ to the left, respectively. These rules are illustrated in Figure 1. Since~$\ell \le \kappa$ and~$\gamma >0$, the operator $\mathcal{L}$ indeed becomes the generator of a Markov process when $N$ is sufficiently large.

The specific form of $\mathcal{L}_A$
  is chosen to ensure that the dynamics preserves the same invariant measures as in the symmetric case and satisfies the Einstein relation \eqref{def:macroscopiccurrent}. Identifying such an asymmetric dynamics is generally non-trivial. Here, following the approach in \cite{Olla}, we derive this generator by formally considering the local reversal of an inhomogeneous Gibbs state, where small gradients are imposed on the chemical potentials associated with particles and energy, proportional to $\alpha_p$ and $\alpha_e$ respectively.


\begin{figure}
    \centering
    \includegraphics[scale=.75]{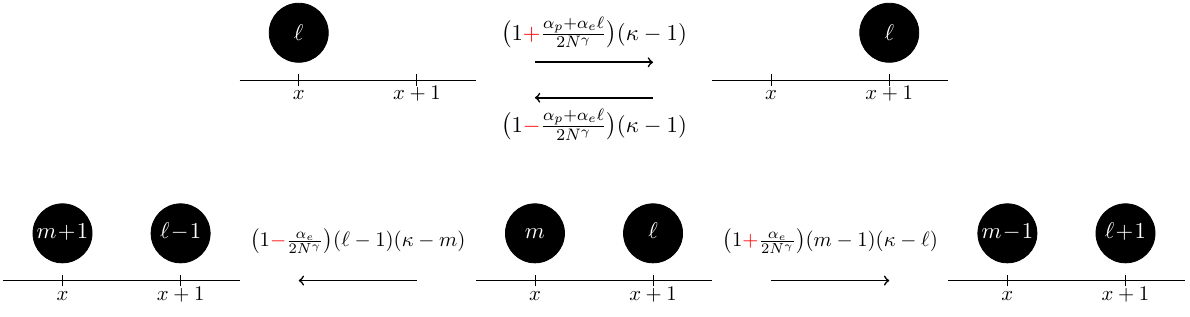}
    \label{fig:model}
    \caption{Schematic illustration of the asymmetric dynamics of the EPE: particle jump rates (top) and energy transfer rates (bottom).}
\end{figure}

\begin{remark}
    Throughout the article, we put a superscript $p$ to refer to quantities related to particle jumps, and a superscript $e$ to refer to quantities related to energy transfers. Even though it is not emphasized in the notations, the generator $\mathcal{L}$ strongly depends on the scaling parameter $N$.
\end{remark}

\subsection{Invariant measures}

\begin{definition}\label{defin:invariantmeasures}
    Take $\rho\in [0,1]$ and $\mathcal{E}\in [\rho ,\kappa\rho]$ corresponding respectively to particle and energy density parameters. Define the marginal $\nu_{\rho ,\mathcal{E}}$ to be the probability measure on $\{0,\hdots ,\kappa\}$ under which:
    \begin{itemize}
        \item the variable $\xi_x =\xi (\eta_x )$ is distributed according to a Bernoulli distribution of parameter~$\rho$,
        \item conditionally on $\{\xi_x=1\}$, the variable $\eta_x-1$ is distributed according to a Binomial distribution of parameters $\kappa -1$ and 
        \begin{equation}\label{def:p}
        \mathtt{p}=\mathtt{p}(\rho ,\mathcal{E}) := \frac{\frac{\mathcal{E}}{\rho} -1}{\kappa -1} \in [0,1].
        \end{equation}
    \end{itemize}
    Namely,
    \begin{equation*}
        \nu_{\rho ,\mathcal{E}}(\eta_x=0)=1-\rho \qquad\mbox{ and }\qquad \nu_{\rho ,\mathcal{E}}(\eta_x=\ell )=\rho {{\kappa -1}\choose{\ell -1}}\mathtt{p}^{\ell -1}(1-\mathtt{p})^{\kappa -\ell} \quad \mbox{ for }1\le \ell \le\kappa .
    \end{equation*}
    Let also $\mu_{\rho ,\mathcal{E}} := \nu_{\rho ,\mathcal{E}}^{\otimes\Z}$ be the corresponding product measure on $\Z$.
\end{definition}

In what follows, we will denote by $\E_{\rho ,\mathcal{E}}$ the expectation, $\mathrm{Var}_{\rho ,\mathcal{E}}$ the variance and $\mathrm{Cov}_{\rho ,\mathcal{E}}$ the covariance with respect to $\mu_{\rho ,\mathcal{E}}$. Then, we have the following result.

\begin{proposition}[Invariant measures]\label{prop:invariantmeasures}
    For any $\rho\in [0,1]$ and any $\mathcal{E}\in [\rho ,\kappa\rho]$, the measure $\mu_{\rho ,\mathcal{E}}$ is reversible for the symmetric generator $\mathcal{L}_S$, and invariant for the generator $\mathcal{L}$. Moreover, for all~$x\in\Z$, it satisfies
    \begin{equation*}
        \E_{\rho ,\mathcal{E}}[\xi_x] = \rho \qquad\mbox{ and }\qquad \E_{\rho ,\mathcal{E}}[\eta_x] =\mathcal{E}.
    \end{equation*}
\end{proposition}

\begin{proof}
    By definition of the measure $\mu_{\rho ,\mathcal{E}}$ and of $\mathtt{p}(\rho ,\mathcal{E})$ given in \eqref{def:p}, a direct computation shows that $\E_{\rho ,\mathcal{E}}[\xi_x]=\rho$ and $\E_{\rho ,\mathcal{E}}[\eta_x]=\mathcal{E}$ for any $x\in\Z$. We start proving the reversibility with respect to $\mathcal{L}_S$, by proving the \emph{detailed balance condition}.

    As the measure $\mu_{\rho ,\mathcal{E}}$ is a homogeneous product measure, it follows that $\eta$ and $\eta^{x,x+1}$ have the same distribution for any $x\in\Z$. Therefore, since the dynamics of $\mathcal{L}_S$ consists of symmetric particle jumps that occur at constant rate, the measure is reversible with respect to the particle jump dynamics. Now, let us check the detailed balance condition for the energy transfer dynamics. Fix any $x\in\Z$ and $m,\ell\in\{1,\hdots ,\kappa -1\}$. As the measure $\mu_{\rho ,\mathcal{E}}$ is a product measure, it is sufficient to compute
    \begin{multline*}
        \mu_{\rho ,\mathcal{E}}(\eta_x=m+1)\mu_{\rho ,\mathcal{E}}(\eta_{x+1}=\ell)c_{x\to x+1}^e(\eta ) \\ = \rho{{\kappa -1}\choose{m}}\mathtt{p}^m(1-\mathtt{p})^{\kappa -1-m} \times \rho {{\kappa -1}\choose{\ell-1}}\mathtt{p}^{\ell-1} (1-\mathtt{p})^{\kappa -\ell} \times m(\kappa -\ell )
    \end{multline*}
    on the one hand, and
    \begin{multline*}
        \mu_{\rho ,\mathcal{E}}(\eta_x=m)\mu_{\rho ,\mathcal{E}}(\eta_{x+1}=\ell +1)c_{x+1\to x}^e(\eta^{x\to x+1} ) \\ = \rho{{\kappa -1}\choose{m-1}}\mathtt{p}^{m-1}(1-\mathtt{p})^{\kappa -m} \times \rho {{\kappa -1}\choose{\ell}}\mathtt{p}^{\ell} (1-\mathtt{p})^{\kappa -1-\ell} \times \ell (\kappa -m )
    \end{multline*}
    on the other hand. Using standard binomial coefficients identities, one easily checks that these quantities are equal, hence, formally
    \begin{equation}\label{eq:reversibility1}
        \mu_{\rho ,\mathcal{E}}(\eta )c_{x\to x+1}^e(\eta )=\mu_{\rho ,\mathcal{E}}(\eta^{x\to x+1})c_{x+1\to x}^e(\eta^{x\to x+1}).
    \end{equation}
    By the same reasoning, one can establish the other formal reversibility relation
    \begin{equation}\label{eq:reversibility2}
        \mu_{\rho ,\mathcal{E}}(\eta )c_{x+1\to x}^e(\eta )=\mu_{\rho ,\mathcal{E}} (\eta^{x+1\to x})c_{x\to x+1}^e(\eta^{x+1\to x}).
    \end{equation}
    This proves the reversibility for the symmetric generator $\mathcal{L}_S$. We now turn to the proof of the invariance with respect to the asymmetric generator $\mathcal{L}$, showing that $\E_{\rho ,\mathcal{E}}\big[\mathcal{L}_Af(\eta)\big]=0$ for any local function $f:\Omega\longrightarrow\R$. The part coming from the particle jump dynamics in $\E_{\rho ,\mathcal{E}}\big[\mathcal{L}_Af(\eta )\big]$ writes 
    \begin{align*}
        \E_{\rho ,\mathcal{E}} & \left[ \frac12 \sum_{x\in\Z} \Big( (\alpha_p +\alpha_e\eta_x)c_{x\to x+1}^p (\eta ) - (\alpha_p+\alpha_e\eta_{x+1})c_{x+1\to x}^p(\eta )\Big) \big[ f(\eta^{x,x+1})-f(\eta )\big]\right]\\ 
        & = \E_{\rho ,\mathcal{E}}\left[ \frac12 \sum_{x\in\Z} \Big( (\alpha_p+\alpha_e\eta_{x+1})c_{x\to x+1}^p(\eta^{x,x+1})-(\alpha_p+\alpha_e\eta_x)c_{x+1\to x}^p(\eta^{x,x+1})\Big) f(\eta )\right] \\ 
        & \qquad - \E_{\rho ,\mathcal{E}}\left[ \frac12 \sum_{x\in\Z} \Big( (\alpha_p+\alpha_e\eta_{x})c_{x\to x+1}^p(\eta )-(\alpha_p+\alpha_e\eta_{x+1})c_{x+1\to x}^p(\eta )\Big) f(\eta )\right] \\ 
        & = -\alpha_p \E_{\rho ,\mathcal{E}}\left[ \sum_{x\in\Z} j_{x,x+1}^p(\eta )f(\eta )\right] + \alpha_e \E_{\rho ,\mathcal{E}} \left[ \sum_{x\in\Z} \big( \eta_{x+1}c_{x+1\to x}^p(\eta )-\eta_xc_{x\to x+1}^p(\eta )\big)f(\eta )\right]
    \end{align*}
    where to obtain the second line, we split the first expectation into two terms, and performed the change of variable $\eta\rightsquigarrow\eta^{x,x+1}$ in the first part and used the fact that~${\mu_{\rho ,\mathcal{E}}(\eta^{x,x+1}) = \mu_{\rho ,\mathcal{E}}(\eta )}$. To obtain the third line, we noticed that~$c_{x\to x+1}^p(\eta^{x,x+1}) = c_{x+1\to x}^p (\eta )$ and recognized the instantaneous current defined in \eqref{def:jp}. Next, the part of $\E_{\rho ,\mathcal{E}}\big[\mathcal{L}_Af(\eta )\big]$ relative to the energy transfer writes
    \begin{align*}
            \E_{\rho ,\mathcal{E}} \bigg[ \frac12\alpha_e\sum_{x\in\Z} \Big( &c_{x\to x+1}^e(\eta )  \big[f(\eta^{x\to x+1})-f(\eta )\big] - c_{x+1\to x}^e(\eta )\big[f(\eta^{x+1\to x})-f(\eta )\big]\Big)\bigg] \\ 
            & = \frac12 \alpha_e \E_{\rho ,\mathcal{E}}\left[ \sum_{x\in\Z}  \left(c_{x\to x+1}^e(\eta^{x+1\to x})\frac{\mu_{\rho ,\mathcal{E}}(\eta^{x+1\to x})}{\mu_{\rho ,\mathcal{E}}(\eta )} - c_{x\to x+1}^e(\eta ) \right)f(\eta )\right] \\
            & \qquad -\frac12\alpha_e\E_{\rho ,\mathcal{E}}\left[\sum_{x\in\Z}\left(  c_{x+1\to x}^e(\eta^{x\to x+1} )\frac{\mu_{\rho ,\mathcal{E}}(\eta^{x\to x+1})}{\mu_{\rho ,\mathcal{E}}(\eta )} - c_{x+1\to x}^e(\eta )\right)f(\eta )\right]\\
            & =  \alpha_e \E_{\rho ,\mathcal{E}}\left[ \sum_{x\in\Z} \big( c_{x+1\to x}^e(\eta )-c_{x\to x+1}^e(\eta )\big)f(\eta )\right]
    \end{align*}
    where to obtain the second line we split the first expectation and performed the changes of variables~$\eta\rightsquigarrow\eta^{x+1\to x}$ and $\eta\rightsquigarrow\eta^{x\to x+1}$ to express everything in terms of $f(\eta )$, and to obtain the third line we used the reversibility relations \eqref{eq:reversibility1} and \eqref{eq:reversibility2}. Summing the obtained expressions and using the definition of the instantaneous current in \eqref{def:je}, we obtain
    \begin{equation*}
        \E_{\rho ,\mathcal{E}}\big[\mathcal{L}_Af(\eta )\big] = - \alpha_p \E_{\rho ,\mathcal{E}}\left[ \sum_{x\in\Z}j_{x,x+1}^p(\eta )f(\eta )\right] - \alpha_e \E_{\rho ,\mathcal{E}}\left[ \sum_{x\in\Z}j_{x,x+1}^e(\eta )f(\eta )\right].
    \end{equation*}
    Combining the gradient decompositions \eqref{eq:gradientjump} and \eqref{eq:gradientenergy} with the translation invariance of the product measure $\mu_{\rho ,\mathcal{E}}$, we conclude that this expectation is zero. This concludes the proof. 
\end{proof} 

\begin{remark}
    When $\rho =0$, the system is empty, and when $\rho =1$ the system reduces to a SEP($\kappa -1$). When $\mathcal{E}=\rho$ or $\mathcal{E}=\kappa\rho$, the system reduces to a SEP. In what follows, we will only consider the non-trivial cases $\rho \in (0,1)$ and $\mathcal{E}\in (\rho ,\kappa\rho )$. 
\end{remark}

\begin{definition}[Compressibility matrix]\label{defin:compressibilitymatrix}
    For any $\rho\in (0,1)$ and $\mathcal{E}\in (\rho ,\kappa\rho)$, we define the \emph{compressibility matrix} to be the covariance matrix
    \begin{equation*}
        \chi (\rho ,\mathcal{E} )\coloneq \begin{pmatrix}
            \mathrm{Var}_{\rho ,\mathcal{E}}(\xi_x) & \mathrm{Cov}_{\rho ,\mathcal{E}}(\xi_x;\eta_x) \\
            \mathrm{Cov}_{\rho ,\mathcal{E}}(\xi_x;\eta_x) & \mathrm{Var}_{\rho ,\mathcal{E}}(\eta_x)
        \end{pmatrix}.
    \end{equation*}
    After calculation, it is given by
    \begin{equation}\label{eq:compressibilitymatrix}
        \chi (\rho ,\mathcal{E}) = \begin{pmatrix}
            \rho (1-\rho) & \mathcal{E}(1-\rho)\\
            \mathcal{E}(1-\rho ) & \frac{(\mathcal{E}-\rho )(\kappa\rho -\mathcal{E})}{(\kappa -1)\rho} + \frac{(1-\rho )\mathcal{E}^2}{\rho }
        \end{pmatrix}.
    \end{equation}
    This matrix is symmetric and positive definite.
\end{definition}

\subsection{Hydrodynamic limit and fluctuation fields}

Let $a>0$, and consider the Markov process $(\eta (t))_{t\ge 0}$ on $\Omega$ driven by the time-rescaled generator $N^a\mathcal{L}$. Let $( \xi (t))_{t\ge 0}$ be the corresponding Markov process on $\{ 0,1\}^{\Z}$. For any $t\ge 0$, we define the following random measures on $\R$, called \emph{empirical measures}, by
\begin{equation*}
    \pi_t^{p,N}(\!\diff u) \coloneq \frac1N\sum_{x\in\Z}\xi_x(t) \delta_\frac {x}{N}(\!\diff u) \qquad \mbox{ and }\qquad \pi_t^{e,N}(\!\diff u) \coloneq \frac1N\sum_{x\in\Z}\eta_x(t) \delta_\frac{x}{N}(\!\diff u),
\end{equation*}
where $\delta_v(\!\diff u)$ denotes the Dirac mass at $v\in\R$. In the diffusive timescale $a=2$, and when~$\gamma =1$, the following convergences in probability, in the sense of weak convergence of measures, is expected to hold for these random measures
\begin{equation*}
    \pi_t^{p,N}(\!\diff u)\xrightarrow[N\to\infty]{\p}\rho (t,u)\diff u\qquad \mbox{ and }\qquad \pi_t^{e,N}(\!\diff u)\xrightarrow[N\to\infty]{\p} \mathcal{E}(t,u)\diff u
\end{equation*}
where the profiles $\rho (t,\cdot )$ and $\mathcal{E}(t,\cdot )$ satisfy the partial differential equation (PDE)
\begin{equation*}
    \partial_t\begin{pmatrix}
        \rho \\ 
        \mathcal{E}
    \end{pmatrix} = (\kappa -1)\Delta\begin{pmatrix}
        \rho \\ 
        \mathcal{E}
    \end{pmatrix} - \nabla \mathbf{j}(\rho ,\mathcal{E})
\end{equation*}
starting from some compatible initial profiles. In this equation, the macroscopic current $\mathbf{j}$ is determined by the expectation of the asymmetric currents under the stationary measures $\mu_{\rho,\mathcal{E}}$ (see \cref{appendix:couplingmatrices} for the explicit form). A direct calculation shows that it satisfies the \emph{Einstein relation}
\begin{equation}\label{def:macroscopiccurrent}
    \mathbf{j}(\rho ,\mathcal{E}) = D \, \chi(\rho,\mathcal{E}) \begin{pmatrix}
    \alpha_p\\
    \alpha_e
    \end{pmatrix} = (\kappa -1) \chi (\rho ,\mathcal{E})\begin{pmatrix}
    \alpha_p\\
    \alpha_e
    \end{pmatrix}
\end{equation}
with the diffusion matrix $D_{ij}=(\kappa-1)\delta_{ij}$.
This convergence result, known as \emph{hydrodynamic limit}, corresponds to a law of large numbers for the process. It is expected to be proved by following Guo, Papanicolaou and Varadhan's \emph{entropy method} \cite{guo1988nonlinear}, together with the equivalence of ensembles proved in the present paper, but we do not provide details here since it is not the focus of the present article. Our aim is rather to establish a central limit theorem around this typical state, in the \emph{stationary} case where the process starts from its invariant measure. For the remainder of the paper, we fix~$\rho\in (0,1)$ and~$\mathcal{E}\in (\rho ,\kappa\rho )$, and assume that the process starts from the measure $\mu_{\rho ,\mathcal{E}}$. Since $\mu_{\rho ,\mathcal{E}}$ is an invariant measure, the law of $\eta(t)$ is~$\mu_{\rho ,\mathcal{E}}$ for all $t \ge 0$. To study the limiting behaviour of~$\sqrt{N}\big( \langle\pi_t^{\alpha ,N},G\rangle - \E_{\rho ,\mathcal{E}}[\langle\pi_t^{\alpha,N},G\rangle]\big)$ for~$\alpha=p,e$, we define the \emph{fluctuation fields} as
\begin{equation}\label{def:fluctuationfields}
\begin{aligned}
    \mathcal{Y}_t^{p,N}(\!\diff u) \coloneq \frac{1}{\sqrt{N}}\sum_{x\in\Z} \bar\xi_x(t)\delta_\frac{x}{N}(\!\diff u), \\
    \mathcal{Y}_t^{e,N}(\!\diff u)\coloneq \frac{1}{\sqrt{N}}\sum_{x\in\Z}\bar\eta_x(t)\delta_\frac{x}{N}(\!\diff u)
\end{aligned}
\end{equation}
where $\bar\xi_x(t)=\xi_x(t)-\rho$ and $\bar\eta_x(t)=\eta_x(t)-\mathcal{E}$ denote the centered variables. These fluctuation fields act on functions in the Schwartz space $\schwartz$, and are thus regarded as elements of the space~$\schwartzprime$ of tempered distributions. The main result of this work characterizes the limit of these fields as~$N\to\infty$ as the solution to a stochastic partial differential equation. Before stating it, we define the notion of solution to this SPDE.

\subsection{Energy solution to the Stochastic Burgers Equation}
\label{subsec:energysolution}

The one-dimensional \emph{Stochastic Burgers Equation} (SBE) is the SPDE of the form
\begin{equation}\label{eq:SBE}
    \partial_t\mathcal{Z}_t = \mathfrak{a}\Delta\mathcal{Z}_t + \mathfrak{g} \nabla\mathcal{Z}_t^2 +\sqrt{2\mathfrak{a}\sigma^2}\nabla\dot{\mathcal{W}}_t
\end{equation}
characterized by parameters $\mathfrak{a} >0$, $\mathfrak{g}\in\R$, $\sigma\neq 0$, where $\mathcal{W}_t$ is an $\schwartzprime$-valued Brownian motion with covariance kernel given by 
\begin{equation*}
    \E\big[ \mathcal{W}_s(\varphi )\mathcal{W}_t(\psi)\big] = (s\wedge t)\langle \varphi ,\psi\rangle_{L^2}
\end{equation*}
for any $s,t\ge 0$ and any $\varphi ,\psi\in\schwartz$. The notion of solution we adopt is that of \emph{energy solutions}, first proposed in \cite{goncalves_nonlinear_2014}, where existence was established, while its uniqueness was subsequently proved in \cite{gubinelli_energy_2018,gubinelli_infinitesimal_2020}. 

Assume that for any $t\in [0,T]$, the random variable $\mathcal{Z}_t$ is distributed as a spatial white noise on $\R$ with variance $\sigma^2$. Then, the quantity defined for any $0\le s<t\le T$, any~$\varphi\in\schwartz$ and~$\varepsilon >0$ by 
\begin{equation}\label{eq:A_energysol}
    \mathcal{A}_{s,t}^{\varepsilon} (\varphi ) \coloneq \int_s^t\int_{\R} \mathcal{Z}_r(\iota_\varepsilon (u;\cdot ))^2\nabla\varphi (u)\diff u\diff r
\end{equation}
exists, even though the function $\iota_\varepsilon (u;\cdot ):= \iota_\varepsilon (u-\cdot )$ where $\iota_\varepsilon (u):=\frac{1}{\varepsilon}\ind_{[0,\varepsilon )}(u)$ does not belong to the Schwartz space. For details we
refer the reader to, for instance \cite[Section 2.2]{GoncalvesJaraSethuraman2015}. To identify the non-linear term of the equation as the limit of this quantity as $\varepsilon\to 0$, we rely on \cite[Theorem 1]{goncalves_nonlinear_2014}, which ensures that the $\schwartzprime$-valued stochastic process with continuous trajectories~$(\mathcal{A}_t)_{t\ge 0}$ can be defined as the $L^2$-limit 
\begin{equation}\label{eq:Atlimit}
    \mathcal{A}_t(\varphi ) = \lim_{\varepsilon\to 0}\ \mathcal{A}_{0,t}^{\varepsilon }(\varphi ) 
\end{equation}
for any $t\in [0,T]$ and $\varphi\in\schwartz$, provided that the following energy estimate holds.
\begin{definition}[Energy estimate]\label{defin:energyestimate}
    We say that the process $(\mathcal{Z}_t)_{t\in [0,T]}$ satisfies the \emph{energy estimate} if there exists a finite constant $C_0>0$ such that for any $0\le s<t\le T$, any $0<\delta\le \varepsilon <1$ and any $\varphi\in\schwartzprime$, we have 
    \begin{equation}\label{eq:energyestimate}
        \E\big[ (\mathcal{A}_{s,t}^{\varepsilon} (\varphi )-\mathcal{A}_{s,t}^{\delta}(\varphi ))^2\big] \le C_0\varepsilon (t-s)\|\nabla\varphi\|_{L^2}^2.
    \end{equation}
\end{definition}

\noindent We are now ready to define the notion of energy solution to the SBE.

\begin{definition}[Energy solution to the SBE]\label{defin:energysolution}
    We say that a process $(\mathcal{Z}_t)_{t\in [0,T]}$ with continuous~$\schwartzprime$-valued trajectories is an \emph{energy solution} to the SBE \eqref{eq:SBE} if the following four conditions are satisfied:
    \begin{enumerate}[label=(\roman*)]
        \item\label{energysol_item1} For any $t\in [0,T]$, the random variable $\mathcal{Z}_t$ is distributed as a spatial white noise on $\R$ with variance $\sigma^2$ ;
        \item\label{energysol_item2} The process $(\mathcal{Z}_t)_{t\in [0,T]}$ satisfies the energy estimate \eqref{eq:energyestimate} ;
        \item\label{energysol_item3} For any $\varphi\in\schwartz$, the process defined for any $t\in [0,T]$ by
        \begin{equation*}
            \mathcal{M}_t(\varphi) = \mathcal{Z}_t(\varphi )-\mathcal{Z}_0(\varphi )-\mathfrak{a}\int_0^t \mathcal{Z}_s(\Delta\varphi )\diff s  + \mathfrak{g}\mathcal{A}_t(\varphi )
        \end{equation*}
        is a continuous mean-zero martingale with respect to the natural filtration of $(\mathcal{Z}_t)_{t\in [0,T]}$, and where $\mathcal{A}_t$ has been defined as the $L^2$-limit \eqref{eq:Atlimit}. Moreover, its quadratic variation is given by $\langle\mathcal{M}(\varphi )\rangle_t = 2\mathfrak{a}\sigma^2t\|\nabla\varphi\|_{L^2}^2$ ;
        \item\label{energysol_item4} The reversed process $(\widehat{\mathcal{Z}}_t=\mathcal{Z}_{T-t})_{t\in [0,T]}$ also satisfies \ref{energysol_item3} with the process $(\mathcal{A}_t)_{t\in [0,T]}$ replaced by the process $(\widehat{\mathcal{A}}_t = \mathcal{A}_{T-t}-\mathcal{A}_T)_{t\in [0,T]}$.
    \end{enumerate}
\end{definition}

Then, we have the following result about existence and uniqueness of energy solutions to the SBE.

\begin{theorem}[\cite{goncalves_nonlinear_2014,gubinelli_energy_2018}]\label{thm:existenceuniquenessenergysolutionsSBE}
       For any initial condition $\mathcal{Z}_0$ distributed as a spatial white noise on $\R$ with variance $\sigma^2$, there exists a unique (up to indistinguishability) energy solution to the SBE \eqref{eq:SBE} in the sense of \cref{defin:energysolution}.
\end{theorem}

We further provide the following definition to cover the case of diffusive fluctuations, which also arises in the context of the EPE.

\begin{definition}[Solution to the Ornstein-Uhlenbeck equation]\label{defin:OUsolution}
    We say that a process $(\mathcal{Z}_t)_{t\in [0,T]}$ with continuous $\schwartzprime$-valued trajectories is a solution to the \emph{Ornstein-Uhlenbeck equation}
    \begin{equation}\label{eq:OU}
    \partial_t\mathcal{Z}_t = \mathfrak{a}\Delta\mathcal{Z}_t +\sqrt{2\mathfrak{a}\sigma^2}\nabla\dot{\mathcal{W}}_t
    \end{equation}
    if it satisfies conditions \ref{energysol_item1} and \ref{energysol_item3} with $\mathfrak{g}=0$ in \cref{defin:energysolution}. The existence and uniqueness of solutions to the Ornstein-Uhlenbeck equation is a classical result, see \cite{holley_generalized_1978}.
\end{definition}

\begin{remark}[About uncoupled SBEs]\label{remark:uncoupledSBEs}
    More generally, if one considers a system of $n$ uncoupled SBEs, that is, the equation that reads for $\vec{\mathcal{Z}}\in C([0,T],\schwartzprime^n)$ as
    \begin{equation*}        
        \partial_t\vec{\mathcal{Z}}_t^i = \mathfrak{a}_i\Delta\vec{\mathcal{Z}}_t^i + \mathfrak{g}_i \nabla\big( \vec{\mathcal{Z}}_t^i\big)^2 + \sqrt{\smash[b]{2\mathfrak{a}_i\sigma_i^2}}\nabla\dot{\mathcal{W}}_t^i,\qquad i = 1,\hdots ,n
    \end{equation*}
    where $(\mathcal{W}^1,\hdots ,\mathcal{W}^n)$ are independent $\schwartzprime$-valued Brownian motions started from spatial white noises with respective variances $\sigma_1^2,\hdots ,\sigma_n^2$ that are mutually independent, and independent of the Brownian motions, then uniqueness of the distribution of~$\mathcal{Z}$ also holds, and it is given by the product of the distributions of the energy solutions to each scalar SBE. Indeed, in \cite[Theorem 2.4]{gubinelli_energy_2018}, the authors proved that there is a unique (up to indistinguishability) strong solution to each scalar SBE, and that it is given by the distributional derivative of the solution to the Stochastic Heat Equation (SHE) with multiplicative noise. Then, each coordinate of $\vec{\mathcal{Z}}$ is a strong solution to a scalar SBE, and as the solution to the SHE with multiplicative noise can be constructed as a mild solution, the assumption that the initial conditions and noises are independent ensures that the solutions of the corresponding multiplicative SHEs are independent. Therefore, the distribution of $\vec{\mathcal{Z}}$ is given by the product of the distributions of the energy solutions to each scalar SBE, and in particular, uniqueness of the distribution of $\vec{\mathcal{Z}}$ holds.
\end{remark}

\subsection{Main results}

Let $\mathbf{J}=\mathbf{J}(\rho ,\mathcal{E})$ be the Jacobian matrix of the macroscopic current $\mathbf{j}$ defined in \eqref{def:macroscopiccurrent}, and of which we give an explicit expression in \cref{appendix:couplingmatrices}. The first notable contribution of this paper is the following result, which establishes the existence of normal modes for the system.

\begin{theorem}[Diagonalizability of the Jacobian]\label{thm:diagonalizability}
    For any $\rho\in (0,1)$ and any $\mathcal{E}\in (\rho ,\kappa\rho)$, the Jacobian matrix $\mathbf{J}(\rho ,\mathcal{E})$ is diagonalizable with real eigenvalues. Under the additional condition~$(\alpha_p,\alpha_e)\neq (0,0)$, these eigenvalues are distinct.

    Moreover, if $u$ and $v$ are two left eigenvectors of $\mathbf{J}(\rho ,\mathcal{E})$ associated with distinct eigenvalues, then they satisfy the orthogonality relation
    \begin{equation}\label{eq:orthogonalityrelation}
        u\chi(\rho ,\mathcal{E})v^\top =0
    \end{equation}
    where $\chi (\rho ,\mathcal{E})$ is the compressibility matrix defined in \eqref{eq:compressibilitymatrix}.
\end{theorem}

\begin{remark}
    Even if the Einstein's relation \eqref{def:macroscopiccurrent} for the asymmetric macroscopic current $\mathbf{j}$ holds, it does not necessarily imply that the Jacobian matrix $\mathbf{J}$ equals $D=\sigma \chi^{-1}$ where $D$ is the diffusion matrix and $\sigma$ is the Onsager matrix satisfying \[\mathbf{j}=\sigma \begin{pmatrix}
    \alpha_p\\
    \alpha_e
    \end{pmatrix}.\] Therefore, the diagonalizability of $\mathbf{J}$ is generally not a priori clear. However, in the proof of \cref{thm:diagonalizability}, we demonstrate that $\mathbf{J}$ can be written in the form $\mathbf{J}=Q\chi^{-1}$ for a certain symmetric matrix $Q$. It should be noted that, in our setting, $Q$ is not necessarily positive semi-definite.
\end{remark}

This result is proved for our particular model, but we believe the same arguments would hold in great generality for gradient models with multiple conservation laws, even when the diffusion matrix is not diagonal. We defer its proof to \cref{sec:proofdiagonalizability} and proceed to state the results about fluctuations of the EPE, starting with the case $\alpha_e\neq 0$. Beforehand, we need to introduce some additional notations. 

For any $v\in\R$, we denote the translation operator $T_v$ that acts on functions $\varphi\in\schwartz$ through
\begin{equation*}
    T_v\varphi (u) := \varphi ( u\,;v) =\varphi (u-v) ,\qquad \forall u\in\R.
\end{equation*}
We focus on linear combinations of the fluctuation fields defined in \eqref{def:fluctuationfields} of the form
\begin{equation}\label{eq:linearcombinationfields}
    \mathcal{Z}_t^N(\varphi )= \mathfrak{c}\mathcal{Y}_t^{p,N}(T_{vN^{a-1}t}\varphi) + \mathcal{Y}_t^{e,N}(T_{vN^{a-1}t}\varphi)
\end{equation}
for appropriate parameters $v, \mathfrak{c}\in\R$. As shown in \cref{appendix:couplingmatrices} and \cref{thm:diagonalizability}, when $\alpha_e\neq 0$, the left eigenvectors of the Jacobian matrix $\mathbf{J}$ can be normalized in the form
\begin{equation*}
    \begin{pmatrix}
         \mathfrak{c}_i & 1
    \end{pmatrix}\mathbf{J} = \lambda_i \begin{pmatrix}
         \mathfrak{c}_i & 1
    \end{pmatrix} \qquad \mbox{ for }i=1,2,
\end{equation*} 
where, without loss of generality, we order the eigenvalues as $\lambda_1>\lambda_2$. For each $i \in \{1,2\}$, let~$\mathcal{Z}^{N,i}$ denote the linear combination \eqref{eq:linearcombinationfields} corresponding to $v=\frac{\lambda_i}{N^\gamma}$ and $\mathfrak{c}=\mathfrak{c}_i$. Our main result then reads as follows.

\begin{theorem}[Stationary fluctuations for the EPE, case $\alpha_e\neq 0$]\label{thm:main}
    Assume that the parameters~$(\alpha_p,\alpha_e,\rho ,\mathcal{E})$ are chosen such that all the diagonal entries of the coupling matrices are non-zero (\textit{cf}.~\cref{prop:noncancellation} in \cref{appendix:couplingmatrices}). In the diffusive timescale $a=2$, and with a weak asymmetry given by $\gamma =\frac12$, the couple of processes $(\mathcal{Z}^{N,1},\mathcal{Z}^{N,2})_{N\ge 1}$ converges in distribution, as $N\to\infty$, to a pair of independent processes $(\mathcal{Z}^1,\mathcal{Z}^2)$, where each $\mathcal{Z}^i$ is an energy solution to the SBE
    \begin{equation}\label{eq:SBE_main}
        \partial_t\mathcal{Z}_t^i = (\kappa -1)\Delta\mathcal{Z}_t^i + \mathfrak{g}_i\nabla (\mathcal{Z}_t^i)^2 + \sqrt{\smash[b]{2(\kappa -1)\sigma_i^2}} \nabla\dot{\mathcal{W}}_t^i
    \end{equation}
    starting from independent spatial white noises with variances $\sigma_1,\sigma_2$ respectively. Here, $(\mathcal{W}^1,\mathcal{W}^2)$ are two independent $\schwartzprime$-valued Brownian motions, and the parameters are given by 
    \begin{itemize}
        \item $\mathfrak{g}_i=\mathbf{G}_{ii}^i$ ;
        \item $\sigma_i^2 = \begin{pmatrix}
            \mathfrak{c}_i & 1
        \end{pmatrix}\chi (\rho ,\mathcal{E}) \begin{pmatrix}
            \mathfrak{c}_i \\ 
            1
        \end{pmatrix}$ where $\chi (\rho ,\mathcal{E})$ is the compressibility matrix defined in \eqref{eq:compressibilitymatrix}.
    \end{itemize}
    Above, the matrix $\mathbf{G}^i$ is the coupling matrix defined in \eqref{def:couplingmatrices}, and discussed in \cref{appendix:couplingmatrices}.
\end{theorem}

This result is the main contribution of this work, and the remaining of the paper will be devoted mainly to its proof. For completeness, we also state the result of fluctuations in the case $\alpha_e=0$, but we do not prove it in detail as it can be adapted from the proof of \cref{thm:main}.

\begin{theorem}[Stationary fluctuations for the EPE, case $\alpha_e=0$]\label{thm:main2}
    Assume that $\alpha_e=0$, and for any $t\in [0,T]$, $\varphi\in\schwartz$ define
    \begin{equation*}
        \mathcal{Z}_t^{N ,1}(\varphi ) = \mathcal{Y}_t^{p,N}(T_{v_1N^{a-1}t}\varphi) - \frac{\rho}{\mathcal{E}}\mathcal{Y}_t^{e,N}(T_{v_1N^{a-1}t}\varphi) \qquad \mbox{ and }\qquad  \mathcal{Z}_t^{N ,2}(\varphi ) = \mathcal{Y}_t^{p,N}(T_{v_2N^{a-1}t}\varphi)
    \end{equation*}
    where $v_1 = \frac{\alpha_p (\kappa -1)(1-\rho )}{N^\gamma}$ and $v_2 = \frac{\alpha_p(\kappa -1)(1-2\rho )}{N^\gamma}$. Then, in the diffusive timescale $a=2$ and with a weak asymmetry given by $\gamma =\frac12$, the couple of processes $(\mathcal{Z}^{N,1},\mathcal{Z}^{N,2})_{N\ge 1}$ converges in distribution, as $N\to\infty$ to a pair of independent processes $(\mathcal{Z}^1,\mathcal{Z}^2)$ starting from independent spatial white noises with respective variances $\sigma_1$ and $\sigma_2$, where $\mathcal{Z}^1$ is a solution to the Ornstein-Uhlenbeck equation 
    \begin{equation*}
        \partial_t\mathcal{Z}_t^1 = (\kappa -1)\Delta \mathcal{Z}_t^1 +\sqrt{\smash[b]{2(\kappa -1)\sigma_1^2}}\nabla\dot{\mathcal{W}}_t^1
    \end{equation*}
    in the sense of \cref{defin:OUsolution}, while $\mathcal{Z}^2$ is an energy solution to the SBE
    \begin{equation*}
        \partial_t\mathcal{Z}_t^2 = (\kappa -1)\Delta\mathcal{Z}_t^2 + \mathfrak{g}_2\nabla (\mathcal{Z}_t^2)^2 + \sqrt{\smash[b]{2(\kappa -1)\sigma_2^2}} \nabla\dot{\mathcal{W}}_t^2
    \end{equation*}
    in the sense of \cref{defin:energysolution}. Above, $(\mathcal{W}^1,\mathcal{W}^2)$ are two independent $\schwartzprime$-valued Brownian motions, and the parameters are given by
    \begin{equation*}
        \mathfrak{g}_2 = -\alpha_p(\kappa -1) ,\qquad \sigma_1^2 = \frac{\rho \big( 1-\frac{\rho}{\mathcal{E}}\big) \big(\kappa\frac{\rho}{\mathcal{E}} -1\big)}{\kappa -1} \qquad \mbox{ and }\qquad \sigma_2^2 = \rho (1-\rho ).
    \end{equation*}
\end{theorem}

\begin{remark}
    The choice of linear combinations of the fluctuation fields in \cref{thm:main2} and the values of the parameters are explained by the computations made in \cref{prop:couplingmatricesalpha_e0} of \cref{appendix:couplingmatrices}.
\end{remark}

\begin{remark}
When $\alpha_e=0$, the process $(\xi(t))_{t \ge 0}$ reduces to the weakly asymmetric exclusion process (WASEP), and our result for its stationary fluctuation $\mathcal{Z}_t^{N ,2}(\varphi ) = \mathcal{Y}_t^{p,N}(T_{v_2N^{a-1}t}\varphi)$ in \cref{thm:main2} is fully consistent with established results for the scaling limit for the WASEP \cite{goncalves_nonlinear_2014}. 
\end{remark}

In the next subsection, we prove \cref{thm:diagonalizability}. Then, the remainder of the article is dedicated to the proof of \cref{thm:main}. 

Even though the choice $a=2$ and $\gamma =\frac12$ is made in these results, we will keep the parameters~$a, \gamma >0$ free as much as possible in the remainder of the article, so that the intermediate results we prove can be useful in other contexts.

\subsection{Diagonalizability of the Jacobian: Proof of \cref{thm:diagonalizability}}
\label{sec:proofdiagonalizability}

We prove a decomposition of the form $\mathbf{J} = Q\chi^{-1}$ for the Jacobian matrix, where $Q$ is a symmetric matrix and $\chi$ is the compressibility matrix defined in \eqref{eq:compressibilitymatrix}. This decomposition implies that $\mathbf{J}$ is diagonalizable with real eigenvalues since it writes as the conjugate of a symmetric matrix. Indeed, as $\chi$ is symmetric and positive definite, it admits a unique symmetric positive definite square root $\chi^\frac12$ and then
\begin{equation*}
    \mathbf{J} = \chi^\frac12 (\chi^{-\frac12}Q\chi^{-\frac12})\chi^{-\frac12}.
\end{equation*}
Moreover, this immediately implies the second part of \cref{thm:diagonalizability} since if $u$ and $v$ are two left eigenvectors of $\mathbf{J}$ associated with distinct eigenvalues, then $u\chi^{\frac12}$ and $v\chi^{\frac12}$ are two left eigenvectors of the symmetric matrix $\chi^{-\frac12} Q\chi^{-\frac12}$ associated with distinct eigenvalues, so that they are orthogonal. 

Thus, we only need to prove the decomposition $\mathbf{J} =Q\chi^{-1}$. Recall the definition of the measure $\nu_{\rho ,\mathcal{E}}$ given in \cref{defin:invariantmeasures}. We reparametrize it by two parameters $(\beta_1,\beta_2)$ depending on $(\rho ,\mathcal{E})$ under the Gibbs form
\begin{equation*}
    \nu_{\rho ,\mathcal{E}} (\eta_x) = \tilde{\nu}_{\beta_1,\beta_2}(\eta_x) = \frac{1}{Z_{\beta_1,\beta_2}}e^{\beta_1\xi_x+\beta_2\eta_x}\bar{\nu}(\eta_x )
\end{equation*}
where $\bar{\nu}$ is a fixed reference\footnote{\textit{i.e.}~a probability measure whose support is the set $\{0,1,\hdots ,\kappa\}$.} measure on $\{ 0,1,\hdots ,\kappa\}$, and $Z_{\beta_1,\beta_2}$ is the associated partition function. The parameters $\beta_1,\beta_2$ are chosen such that the expectations of $\xi_x$ and $\eta_x$ under~$\tilde{\nu}_{\beta_1,\beta_2}$ are equal to $\rho$ and $\mathcal{E}$ respectively. For simplicity, temporarily denote $(\rho ,\mathcal{E}) = (\rho_1,\rho_2)$. Under this parametrization, the compressibility matrix $\chi (\rho ,\mathcal{E})$ writes as
\begin{equation}\label{eq:chi_beta_partitionfunction}
        \chi_{ij}(\beta_1,\beta_2)=\chi_{ij}(\rho_1 ,\rho_2) = \partial_{\beta_i}\partial_{\beta_j} \log Z_{\beta_1,\beta_2} =\frac{\partial\rho_j}{\partial\beta_i}\qquad \mbox{ for }i,j\in\{1,2\}. 
\end{equation}
By the chain rule,
    \begin{equation*}
        \mathbf{J}_{ij} = \frac{\partial\mathbf{j}_i}{\partial\rho_j} = \frac{\partial\beta_1}{\partial\rho_j}\frac{\partial\mathbf{j}_i}{\partial\beta_1} + \frac{\partial\beta_2}{\partial\rho_j}\frac{\partial\mathbf{j}_i}{\partial\beta_2}
    \end{equation*}
for any $i,j\in\{1,2\}$. Hence, setting $Q_{ij} := \frac{\partial\mathbf{j}_i}{\partial\beta_j}$, we have that the matrix $\mathbf{J}$ writes as $\mathbf{J}=Q\chi^{-1}$ thanks to \eqref{eq:chi_beta_partitionfunction}. It remains to show that the matrix $Q$ is symmetric. Recall the expression of the macroscopic current vector $\mathbf{j}(\rho ,\mathcal{E})$ given in \eqref{def:macroscopiccurrent}, we have
\begin{equation*}
    \mathbf{j}_i = (\kappa -1)\big( \alpha_p \chi_{i1}+\alpha_e\chi_{i2}\big)
\end{equation*}
hence
\begin{equation*}
    Q_{ij}= \frac{\partial\mathbf{j}_i}{\partial\beta_j} = (\kappa -1)\left( \alpha_p \frac{\partial\chi_{i1}}{\partial\beta_j}+\alpha_e\frac{\partial\chi_{i2}}{\partial\beta_j}\right).
\end{equation*}
In order to show that $Q$ is symmetric, it is therefore sufficient to prove that for any $i,j,k\in\{1,2\}$, we have 
\begin{equation*}
    \frac{\partial\chi_{ik}}{\partial\beta_j} = \frac{\partial\chi_{jk}}{\partial\beta_i},
\end{equation*}
which is immediate from the relation \eqref{eq:chi_beta_partitionfunction}. This concludes the proof of the diagonalizability with real eigenvalues. 

Notice that if two eigenvalues are equal, then the Jacobian matrix must be a scalar multiple of the identity matrix. In particular, the off-diagonal elements of $\mathbf{J}(\rho ,\mathcal{E})$ must be zero, which immediately implies that $(\alpha_p,\alpha_e)=(0,0)$ according to the explicit expression of $\mathbf{J}$ given in \cref{appendix:couplingmatrices}. \hfill\qed

\section{Equivalence of ensembles}
\label{sec:equivalenceofensembles}

\noindent We start by introducing some notations. For an integer $\ell\ge 1$, and for any $x\in\Z$, define the local averages
\begin{equation}\label{eq:averages}
    \xi_x^\ell = \frac 1\ell\sum_{y=0}^{\ell -1} \xi_{x+y}\qquad \mbox{ and }\qquad \eta_x^\ell = \frac 1\ell\sum_{y=0}^{\ell -1} \eta_{x+y}.
\end{equation}
If $f:\Omega\longrightarrow\R$ is a local function supported in $\{0 ,\hdots ,\ell_0-1\}$ for some $\ell_0\ge 1$, we define the function $\psi_f (\rho ,\mathcal{E}) = \E_{\rho ,\mathcal{E}}[f]$. For $\ell\ge\ell_0$, we also denote by $\E_\ell [f]\coloneq \E_{\rho ,\mathcal{E}}[f|\xi_0^\ell ,\eta_0^\ell]$ the conditional expectation of $f$ with respect to the variables $\xi_0^\ell$ and $\eta_0^\ell$ under the stationary measure $\mu_{\rho ,\mathcal{E}}$. The next result, called \emph{equivalence of ensembles}, gives an approximation of the conditional expectation $\E_\ell [f]$ as an explicit expression involving $\xi_0^\ell$, $\eta_0^\ell$ and the function $\psi_f$.
\begin{proposition}[Equivalence of ensembles]\label{prop:equivalenceofensembles}
    Let $f:\Omega\longrightarrow\R$ be a local function whose support is included in $\{0,\hdots ,\ell_0-1\}$ for some $\ell_0\ge 1$. Then, there exists a constant $C_1=C_1(f)>0$ such that for any $\ell\ge\ell_0$, 
    \begin{equation}\label{eq:equivensembles}
        \sup_{\eta\in\Omega}\left| \E_\ell [f] - \left(\psi_f(\xi_0^\ell ,\eta_0^\ell ) - \frac{1}{2\ell }\mathrm{Tr}\big( \chi \mathrm{Hess}(\psi_f)\big) (\xi_0^\ell ,\eta_0^\ell)\right)\right| \le \frac{C_1}{\ell^2}.
    \end{equation}
    where $\chi$ is the compressibility matrix defined in \cref{defin:compressibilitymatrix}, and $\mathrm{Hess}(\psi_f)$ denotes the Hessian matrix of $\psi_f$. 
\end{proposition}

\begin{proof}
    Any local function $f$ supported in $\{0,\hdots ,\ell_0-1\}$ can be written as a linear combination of the constant function and the indicator functions of the form $\ind_{\{ \eta_{x_1}=k_1,\hdots ,\eta_{x_m}=k_m\}}$ with parameters~$1\le m\le\ell_0$,  $\{x_1,\hdots ,x_m\}\subseteq\{0,\hdots ,\ell_0-1\}$ and $k_1,\hdots ,k_m\in\{1,\hdots ,\kappa\}$. As the case of constant functions is trivial, it is sufficient to check the result for such indicator functions, and since the measure $\mu_{\rho ,\mathcal{E}}$ is invariant under permutation of sites, it is even sufficient to check it for a function of the form $g(\eta )=\ind_{\{ \eta_0=k_0 , \hdots ,\eta_{y-1}=k_{y-1}\}}$ for fixed $1\le y\le\ell_0$, and~$k_0,\hdots ,k_{y-1}\in\{1,\hdots ,\kappa\}$.

    Under the measure $\mu_{\rho ,\mathcal{E}}$, we have the identities in distribution
    \begin{equation*}
        \xi_x \overset{(d)}{=}X_x\qquad \mbox{ and }\qquad \eta_x \overset{(d)}{=} X_x(1+Y_x) \qquad \mbox{ for all }x\in\Z
    \end{equation*}
    where $\{X_x\}_{x \in \Z}$ and $\{Y_x\}_{x \in \Z}$ are two independent sequences of i.i.d. random variables such that
    \begin{equation*}
        X_x \sim\mathrm{Ber}(\rho )\qquad\mbox{ and }\qquad Y_x \sim\mathrm{Bin}(\kappa -1,\mathtt{p})
    \end{equation*}
    and the parameter $\mathtt{p}=\mathtt{p}(\rho ,\mathcal{E})$ is given in \eqref{def:p}. Our goal is to approximate 
    \begin{equation*}
        \E_\ell [g](r,q) = \mu_{\rho ,\mathcal{E}}\left( \eta_0=k_0,\hdots ,\eta_{y-1}=k_{y-1}\bigg| \begin{array}{l}
            \xi_0 + \hdots +\xi_{\ell -1} = r\\ 
            \eta_0 +\hdots +\eta_{\ell -1}=q
        \end{array}\right)
    \end{equation*}
    as an explicit expression of $r\in \{0,\hdots ,\ell\}$ and $q\in \{ r,\hdots ,\kappa r\}$. By definition of the conditional expectation, we have
    \begin{align*}
        \E_\ell [g](r,q) & = \frac{\mu_{\rho ,\mathcal{E}}(\eta_0=k_0,\hdots ,\eta_{y-1}=k_{y-1}, \xi_0+\hdots +\xi_{\ell -1}=r,\eta_0+\hdots +\eta_{\ell -1}=q)}{\mu_{\rho ,\mathcal{E}}(\xi_0+\hdots +\xi_{\ell -1}=r,\eta_0+\hdots +\eta_{\ell -1}=q)} \\
        & = \frac{\mu_{\rho ,\mathcal{E}}(\eta_0=k_0,\hdots , \eta_{y-1}=k_{y-1}, \xi_y+\hdots +\xi_{\ell -1}=r-y,\eta_y+\hdots +\eta_{\ell -1}=q-\bar{k})}{\mu_{\rho ,\mathcal{E}}(\xi_0+\hdots +\xi_{\ell -1}=r,\eta_0+\hdots +\eta_{\ell -1}=q)}\\
        & \!= \mu_{\rho ,\mathcal{E}} (\eta_0=k_0,\hdots ,\eta_{y-1}=k_{y-1})\frac{\mu_{\rho ,\mathcal{E}}(\xi_y+\hdots +\xi_{\ell-1}=r-y,\eta_y+\hdots +\eta_{\ell -1}=q-\bar{k})}{\mu_{\rho ,\mathcal{E}}(\xi_0+\hdots +\xi_{\ell -1}=r,\eta_0+\hdots +\eta_{\ell -1}=q)}
    \end{align*}
    where $\bar{k}\coloneq k_0+\hdots +k_{y-1}$, and we used the fact that the measure $\mu_{\rho ,\mathcal{E}}$ is product to obtain the third line. Let us express everything in terms of the random variables $X_i$ and $Y_i$. The first probability on the right-hand side is equal to
    \begin{align*}
        \mu_{\rho ,\mathcal{E}}(\eta_0=k_0,&\hdots ,\eta_{y-1}=k_{y-1}) \\ & = \mu_{\rho ,\mathcal{E}}(\eta_0=k_0,\hdots ,\eta_{y-1}=k_{y-1}|\xi_0=1,\hdots , \xi_{y-1}=1)\mu_{\rho ,\mathcal{E}}(\xi_0=1,\hdots ,\xi_{y-1}=1) \\
        & = \p (Y_0=k_0-1,\hdots ,Y_{y-1}=k_{y-1}-1)\p (X_0=1,\hdots ,X_{y-1}=1).
    \end{align*}
    The probability in the numerator is equal to
    \begin{align*}
        \mu_{\rho ,\mathcal{E}}(\eta_y+\hdots +\eta_{\ell -1} & =q-\bar{k}|\xi_y+\hdots +\xi_{\ell -1}=r-y)\mu_{\rho ,\mathcal{E}}(\xi_y+\hdots +\xi_{\ell -1}=r-y) \\ 
        & =\p (Y_y+\hdots +Y_{r-1}=q-\bar{k}-r+y)\p (X_y+\hdots + X_{\ell -1}=r-y)
    \end{align*}
    because conditionally on the event $\{\xi_y+\hdots +\xi_{\ell -1}=r-y\}$ only $r-y$ variables among $\eta_y,\hdots ,\eta_{\ell -1}$ are non-zero and distributed like $1+Y_i$, and by symmetry, we can assume without loss of generality that these variables are the variables $Y_y,\hdots ,Y_{r-1}$. Similarly, the probability in the denominator writes 
    \begin{equation*}
        \p (Y_0+\hdots +Y_{r-1}=q-r)\p (X_0+\hdots +X_{\ell -1}=r).
    \end{equation*}
    If we gather everything together, we find that 
    \begin{equation*}
        \E_\ell [g](r,q)=\E\big[g_1(X)\big| X_0+\hdots +X_{\ell -1}=r\big]\E\big[ g_2(Y)\big| Y_0+\hdots +Y_{r-1}=q-r \big]
    \end{equation*}
    where
    \begin{equation*}
        g_1(X)=\ind\{ X_0=1,\hdots ,X_{y-1}=1\}\qquad\mbox{ and }\qquad g_2(Y)=\ind\{Y_0=k_0-1,\hdots ,Y_{r-1}=k_{y-1}-1\}.
    \end{equation*}
    The equivalence of ensembles for the Bernoulli product measure is known, we state it in \cref{lemma:equivensembles_bernoulli} below, and we refer to \cite{goncalves_scaling_2013} for a proof.  Applying it to the function $g_1$, we get that 
    \begin{equation*}
        \E\big[ g_1(X)\big| X_0+\hdots +X_{\ell -1}=r\big] = \varphi_1\big(\tfrac{r}{\ell}\big) - \frac{\tfrac{r}{\ell} (1-\frac{r}{\ell})}{2\ell }\varphi_1''\big(\tfrac{r}{\ell}\big)+O(\ell^{-2})
    \end{equation*}
    and, as a Binomial random variable is a sum of independent Bernoulli random variables, we can also apply the same lemma to the function $g_2$ to get
    \begin{equation*}
        \E\big[ g_2(Y)\big| Y_0+\hdots +Y_{r-1}=q-r\big] = \varphi_2\big(\tfrac{q-r}{(\kappa -1)r}\big) - \frac{\tfrac{q-r}{(\kappa -1)r}\big( 1-\frac{q-r}{(\kappa -1)r}\big)}{2(\kappa -1)r}\varphi_2''\big(\tfrac{q-r}{(\kappa -1)r}\big) + O(r^{-2}).
    \end{equation*}
    Above, $\varphi_1(\rho)$ and $\varphi_2(\rho )$ stand respectively for the expectation of $g_1$ under the Bernoulli product measure of parameter $\rho$ and the expectation of $g_2$ under the Binomial product measure of parameters $\kappa -1$ and $\rho$. As a consequence, we find that
    \begin{multline*}
        \E_\ell [g](r,q) = \varphi_1\big(\tfrac{r}{\ell}\big)\varphi_2\big(\tfrac{q-r}{(\kappa -1)r}\big) - \frac{\tfrac{r}{\ell} (1-\frac{r}{\ell})}{2\ell } \varphi_1''\big(\tfrac{r}{\ell}\big)\varphi_2\big(\tfrac{q-r}{(\kappa -1)r}\big) \\ - \frac{\tfrac{\ell}{(\kappa -1)r}\frac{q-r}{(\kappa -1)r}\big( 1-\tfrac{q-r}{(\kappa -1)r}\big)}{2\ell }\varphi_1\big(\tfrac{r}{\ell}\big)\varphi_2''\big(\tfrac{q-r}{(\kappa -1)r}\big) + O(\ell^{-2}).
    \end{multline*}
    For this particular function $g$, one finds that $\psi_g(\rho ,\mathcal{E}) = \varphi_1(\rho )\varphi_2\big(\tfrac{\mathcal{E}-\rho}{(\kappa -1)\rho})$. Moreover, using the expression \eqref{eq:compressibilitymatrix}, a tedious but easy computation shows that 
    \begin{align*}
        \mathrm{Tr}(\chi\mathrm{Hess}(\psi_g))(\rho ,\mathcal{E}) & = \chi_{11}(\rho ,\mathcal{E})\frac{\partial^2\psi_g}{\partial\rho^2}(\rho ,\mathcal{E}) + 2\chi_{12}(\rho ,\mathcal{E})\frac{\partial^2\psi_g}{\partial\rho\partial \mathcal{E}}(\rho ,\mathcal{E}) + \chi_{22}(\rho ,\mathcal{E})\frac{\partial^2\psi_g}{\partial \mathcal{E}^2}(\rho ,\mathcal{E})\\ 
        & = \rho (1-\rho )\varphi_1''(\rho )\varphi_2\big( \tfrac{\mathcal{E}-\rho}{(\kappa -1)\rho}\big) + \frac{\tfrac{\mathcal{E}-\rho}{(\kappa -1)\rho}\big( 1-\tfrac{\mathcal{E}-\rho}{(\kappa -1)\rho})}{(\kappa -1)\rho} \varphi_1(\rho )\varphi_2''\big(\tfrac{\mathcal{E}-\rho}{(\kappa -1)\rho}\big)
    \end{align*}
    so it concludes the proof for this particular function $g$. As explained at the beginning of the proof, it allows to conclude for any local function $f$.
\end{proof}

During the proof, we used the equivalence of ensembles for the Bernoulli product measure, we state it here for completeness.

\begin{lemma}[Equivalence of ensembles for the Bernoulli product measure, \cite{goncalves_scaling_2013}]\label{lemma:equivensembles_bernoulli}
    Consider the Bernoulli product measure $\pi_\rho = \mathrm{Ber}(\rho )^{\otimes\Z}$ of parameter $\rho\in [0,1]$, and let $\xi$ be a random variable distributed according to $\pi_\rho$. Let $g:\{0,1\}^{\Z}\longrightarrow\R$ be a local function whose support is included in $\{0,\hdots ,\ell_0-1\}$ for some $\ell_0\ge 1$, and define the function $\varphi_g(\rho ) = \pi_\rho(g)$. Then, there exists a constant~$C_2=C_2(g)>0$ such that for any $\ell\ge\ell_0$, we have 
    \begin{equation*}
        \sup_{\xi\in\{0,1\}^{\Z}} \left| \pi_\rho [ g|\xi_0^\ell ] - \left(\varphi_g(\xi_0^\ell ) - \frac{1}{2\ell}\xi_0^\ell (1-\xi_0^\ell)\varphi_g''(\xi_0^\ell)\right)\right| \le \frac{C_2}{\ell^2}.
    \end{equation*}
\end{lemma}

\noindent The following result is a consequence of \cref{prop:equivalenceofensembles}. 

\begin{corollary}\label{cor:equivalenceofensembles}
    Let $f:\Omega\longrightarrow\R$ be a local function whose support is included in $\{0,\hdots ,\ell_0-1\}$ for some $\ell_0\ge 1$. Then, there exists a constant $C_3=C_3(f,\rho ,\mathcal{E})>0$ such that for any $\ell\ge\ell_0$, we have
    \begin{multline}
        \int_{\Omega}\bigg( \E_\ell[f] - \psi_f(\rho ,\mathcal{E}) - \nabla\psi_f(\rho ,\mathcal{E})\cdot \bar{\mathbf{u}}^\ell \\ - \frac12 \left( (\bar{\mathbf{u}}^\ell)^\top \mathrm{Hess}(\psi_f)(\rho ,\mathcal{E})\bar{\mathbf{u}}^\ell -\frac{1}{\ell}\mathrm{Tr}\big(\chi\mathrm{Hess}(\psi_f)\big)(\rho ,\mathcal{E})\right)\bigg)^2\diff\mu_{\rho ,\mathcal{E}}(\eta )\le \frac{C_3}{\ell^3}
    \end{multline}
    where $\bar{\mathbf{u}}^\ell = \bar{\mathbf{u}}^\ell (\eta )$ is the column vector whose coordinates are $\xi_0^\ell -\rho$ and $\eta_0^\ell -\mathcal{E}$. In particular, 
    \begin{itemize}
        \item if $\psi_f(\rho ,\mathcal{E}) = 0$, we can choose the constant $C_3$ so that
        \begin{equation*}
            \mathrm{Var}_{\rho ,\mathcal{E}}\big( \E_\ell [f]\big) \le \frac{C_3}{\ell};
        \end{equation*}
        \item if $\psi_f(\rho ,\mathcal{E})=0$ and $\nabla\psi_f(\rho ,\mathcal{E})=0$, we can choose the constant $C_3$ so that
        \begin{equation*}
            \mathrm{Var}_{\rho, \mathcal{E}}\big( \E_\ell [f]\big) \le \frac{C_3}{\ell^2};
        \end{equation*}
        \item if $\psi_f(\rho ,\mathcal{E})=0$, $\nabla\psi_f(\rho ,\mathcal{E})=0$ and $\mathrm{Hess}(\psi_f)(\rho ,\mathcal{E})=0$, we can choose the constant $C_3$ so that
        \begin{equation*}
            \mathrm{Var}_{\rho ,\mathcal{E}}\big( \E_\ell [f]\big) \le \frac{C_3}{\ell^3}.
        \end{equation*}
    \end{itemize}
\end{corollary}

\begin{proof}
    By a second-order Taylor expansion of the function $\psi_f$ around the point $(\rho ,\mathcal{E})$, we have
    \begin{equation*}
        \psi_f(\xi_0^\ell ,\eta_0^\ell ) = \psi_f(\rho ,\mathcal{E})+\nabla\psi_f(\rho ,\mathcal{E})\cdot \bar{\mathbf{u}}^\ell + \frac12 (\bar{\mathbf{u}}^\ell)^\top \mathrm{Hess}(\psi_f)(\rho ,\mathcal{E})\bar{\mathbf{u}}^\ell + O \big( \|\bar{\mathbf{u}}^\ell\|^3\big).
    \end{equation*}
    where the error is uniform in $\rho$, $\mathcal{E}$, $\xi_0^\ell$ and $\eta_0^\ell$. Similarly, by a zero-order Taylor expansion of the function $\mathrm{Tr}(\chi\mathrm{Hess}(\psi_f))$, we have
    \begin{equation*}
        \mathrm{Tr}(\chi\mathrm{Hess}(\psi_f))(\xi_0^\ell ,\eta_0^\ell ) = \mathrm{Tr}(\chi\mathrm{Hess}(\psi_f))(\rho ,\mathcal{E}) + O\big( \|\bar{\mathbf{u}}^\ell\|\big).
    \end{equation*}
    where the error is once again uniform in $\rho$, $\mathcal{E}$, $\xi_0^\ell$ and $\eta_0^\ell$. As the $L^2(\mu_{\rho ,\mathcal{E}})$-norms of $\|\bar{\mathbf{u}}^\ell\|^3$ and~$\frac1\ell \|\bar{\mathbf{u}}^\ell\|$ are of order $\ell^{-3/2}$, we deduce the result from \cref{prop:equivalenceofensembles}. The other estimates follow easily.
\end{proof}

\section{Spectral estimates}
\label{sec:spectralestimates}

\subsection{Spectral tools and Kipnis-Varadhan inequality}

In this section, we introduce some spectral tools that will be useful in the sequel. We start with \emph{Kipnis-Varadhan inequality}, which gives an estimation of the variance of an additive functional of a Markov process under the invariant state.

\begin{proposition}[Kipnis-Varadhan inequality, \cite{chang_equilibrium_2001,kipnis_central_1986}]
    Let $(X_t)_{t\ge 0}$ be a Markov process evolving on a state-space $\mathscr{E}$, with generator $\mathscr{L}$ and invariant measure $\pi$. Then, there exists a universal constant $C_4>0$ such that for any smooth mean-zero function $f:[0,T]\times\mathscr{E}\longrightarrow\R$, we have 
    \begin{equation*}\label{eq:kipnisvaradhan}
        \E_\pi\left[ \sup_{0\le t\le T}\left(\int_0^t f(s,X_s)\diff s\right)^2\right] \le C_4\int_0^T \|f(t,\cdot )\|_{-1,\pi}^2\diff t
    \end{equation*}
    where the variational norm $\|\cdot\|_{-1,\pi}$ is defined by
    \begin{equation*}
        \| f\|_{-1,\pi}^2 = \sup_g\left\lbrace 2 \langle f,g\rangle_\pi -\langle g,(-\mathscr{L})g\rangle_\pi \right\rbrace.
    \end{equation*}
    Above, the supremum is taken over all local functions $g:\mathscr{E}\longrightarrow\R$, and $\langle\cdot ,\cdot\rangle_\pi$ stands for the scalar product in $L^2(\pi )$.
\end{proposition}

For our Markov process $(\eta(t))_{t\ge 0}$ with generator $N^a\mathcal{L}$ and invariant measure $\mu_{\rho ,\mathcal{E}}$, this inequality reads as follows. There exists a universal constant $C_4>0$ such that for any smooth mean-zero function $f:[0,T]\times\Omega\longrightarrow\R$, we have
\begin{equation}\label{eq:kipnisvaradhan_ourcontext}
    \E_{\rho ,\mathcal{E}} \left[ \sup_{0\le t\le T}\left(\int_0^t f(s,\eta (s))\diff s\right)^2\right] \le \frac{C_4}{N^a} \int_0^T\|f(t,\cdot )\|_{-1}^2\diff t,
\end{equation}
with the norm $\|\cdot\|_{-1}$ defined by 
\begin{equation*}
    \| f\|_{-1}^2 = \sup_g\big\lbrace 2 \langle f,g\rangle_{\rho ,\mathcal{E}} - \langle g,(-\mathcal{L}_S)g\rangle_{\rho ,\mathcal{E}}\big\rbrace
\end{equation*}
where the supremum is taken over all local functions, and $\langle\cdot ,\cdot\rangle_{\rho ,\mathcal{E}}$ denotes the usual scalar product in $L^2(\mu_{\rho ,\mathcal{E}})$. Indeed, by reversibility of $\mu_{\rho ,\mathcal{E}}$ with respect to the generator $\mathcal{L}_S$ (\textit{cf.}~\cref{prop:invariantmeasures}), the generator $\mathcal{L}_S$ is self-adjoint in $L^2(\mu_{\rho ,\mathcal{E}})$, so it is the only part of the generator $\mathcal{L}$ that contributes to the variational norm.

Inequality \eqref{eq:kipnisvaradhan_ourcontext} reduces the problem of estimating the variance of a time-integrated functional of the process to that of controlling the variational norm of the functional itself. One possible approach to bound this norm is to have an estimate on the spectral gap of the generator $\mathcal{L}_S$. We state in \cref{prop:spectralgap} the \emph{spectral gap estimate} for our process, that will be proved in \cref{subsec:proof_spectral_gap}. But first, we have to introduce the \emph{Dirichlet forms} of the process.

\begin{definition}[Dirichlet forms]\label{defin:dirichletforms}
    Let $\mu$ be any probability measure on $\Omega$. For any $x,y\in\Z$ and any local function $f:\Omega\longrightarrow\R$ we define the Dirichlet forms with respect to the measure~$\mu$ by
    \begin{align}
        & \mathfrak{D}_{x\to y}^p (f;\mu ) \coloneq \frac12\int_{\Omega} c_{x\to y}^p(\eta )\big[ f(\eta^{x,y})-f(\eta )\big]^2\diff\mu (\eta ), \label{def:dirichletformjump}\\
        & \mathfrak{D}_{x\to y}^e(f;\mu )\coloneq \frac12\int_{\Omega}c_{x\to y}^e(\eta )\big[ f(\eta^{x\to y})-f(\eta )\big]^2\diff\mu (\eta )
    \end{align}
    where the jump rates $c^p$ and $c^e$ are given respectively in \eqref{def:particlerate} and \eqref{def:energyrate}, and the transformations of configurations are given by \eqref{eq:transformedconfig}. Finally, we define the total Dirichlet form by
    \begin{equation*}
        \mathfrak{D}(f;\mu ) = \sum_{\substack {x,y\in\Z\\ |x-y|=1}} \big( \mathfrak{D}_{x\to y}^p(f;\mu ) + \mathfrak{D}_{x\to y}^e(f;\mu )\big).
    \end{equation*}
    By reversibility of the measure $\mu_{\rho ,\mathcal{E}}$ with respect to the generator $\mathcal{L}_S$ (\textit{cf.}~\cref{prop:invariantmeasures}), we have the identity $\langle f,(-\mathcal{L}_S)f\rangle_{\rho ,\mathcal{E}} = \mathfrak{D}(f;\mu_{\rho ,\mathcal{E}})$ for any local function $f:\Omega\longrightarrow\R$.
\end{definition}

\noindent Then, the spectral gap inequality reads as follows.

\begin{proposition}[Spectral gap estimate]\label{prop:spectralgap}
    There exists a universal constant $C_5=C_5(\kappa )>0$ such that for any local function $f$ supported in $\{0,\hdots ,\ell -1\}$ for some $\ell\ge 1$, with $\psi_f(\rho ,\mathcal{E})=0$ for any $\rho \in [0,1]$ and $\mathcal{E}\in [\rho ,\kappa\rho]$, we have 
    \begin{equation}\label{eq:spectralgapineq}
        \mathrm{Var}_{\rho ,\mathcal{E}}(f) \le C_5\ell^2 \sum_{\substack {x,y\in\{0,\hdots ,\ell -1\}\\ |x-y|=1}} \Big( \mathfrak{D}_{x\to y}^p (f;\mu_{\rho ,\mathcal{E}}) + \mathfrak{D}_{x\to y}^e(f;\mu_{\rho ,\mathcal{E}})\Big).
    \end{equation}
\end{proposition}

The proof of this result requires some work, so we postpone it to the upcoming \cref{subsec:proof_spectral_gap}, we state first how we will use it to estimate the variational norm. Following for instance the proof of \cite[Proposition 7]{goncalves_nonlinear_2014}, we have the following result that gives an upper bound on the norm of a sum of local functions with disjoint supports in terms of their individual variances.

\begin{proposition}\label{prop:orthovariationalnorm}
    Let $f_1,\hdots ,f_m$ be a sequence of local functions from $\Omega$ to $\R$ that have disjoint supports, and assume that the support of the function $f_i$ has cardinality $\ell_i$ for any~$1\le i\le m$. Assume that $\psi_{f_i}(\rho ,\mathcal{E})=0$ for any $\rho\in [0,1]$, any $\mathcal{E}\in [\rho ,\kappa\rho]$ and any $1\le i\le m$. Then, there exists a generic constant $C_6=C_6(\kappa )>0$ such that
    \begin{equation}
        \| f_1+\hdots +f_m\|_{-1}^2 \le C_6 \sum_{i=1}^m \ell_i^2\mathrm{Var}_{\rho ,\mathcal{E}}(f_i).
    \end{equation}
\end{proposition}

This result can be easily extended as follows to a sequence of infinitely many local functions,  For the proof of its generalization, we refer to \cite[Proposition 4.5]{goncalves_characterization_2025}.

\begin{corollary}\label{cor:spectral0}
    Let $(f_i)_{i\ge 1}$ be a sequence of local functions from $\Omega$ to $\R$ that have disjoint supports, and assume that the support of the function $f_i$ has cardinality $\ell_i$ for any $i\ge 1$. Assume that $\psi_{f_i}(\rho ,\mathcal{E})=0$ for any $\rho\in [0,1]$, any $\mathcal{E}\in [\rho ,\kappa\rho ]$ and any $i\ge 1$. Finally, assume that the series $\sum_{i=1}^\infty \mathrm{Var}_{\rho ,\mathcal{E}}(f_i)$ converges. Then, we have
    \begin{equation}
        \left\|\sum_{i=1}^\infty f_i\right\|_{-1}^2 \le C_6\sum_{i=1}^\infty \ell_i^2 \mathrm{Var}_{\rho ,\mathcal{E}}(f_i).
    \end{equation}
\end{corollary}

If we combine this with inequality \eqref{eq:kipnisvaradhan_ourcontext}, we get the following corollary that will be the main spectral estimate that we use in the sequel.

\begin{corollary}
    \label{cor:spectral}
    There exists a generic constant $C_7=C_7(\kappa )>0$ such that for any sequence of local functions $(f_i)_{i\ge 1}$ from $[0,T]\times \Omega \longrightarrow\R$ such that $f_i(t,\cdot )$ satisfies the hypotheses of \cref{cor:spectral0} for any $t\in [0,T]$ and any $i\ge 1$, we have
    \begin{equation}
        \E_{\rho ,\mathcal{E}}\left[\sup_{0\le t\le T}\left( \int_0^t \sum_{i=1}^\infty f_i(s,\eta (s))\diff s\right)^2\right] \le \frac{C_7}{N^a}\sum_{i=1}^\infty \ell_i^2 \int_0^T\mathrm{Var}_{\rho ,\mathcal{E}}(f_i(t,\cdot ))\diff t.
    \end{equation}
\end{corollary}

\subsection{Proof of the spectral gap estimate (Proposition~\ref{prop:spectralgap})}
\label{subsec:proof_spectral_gap}

The proof relies on the spectral gap estimate for the simple exclusion process that we state just below, together with a mean-field version of it for the generalized exclusion process that we state and prove in \cref{subsec:meanfield_spectral_gap}, and two ``moving particle'' and ``moving energy'' lemmas that allow to compare Dirichlet forms.

\begin{lemma}[Spectral gap estimate for the SSEP, \cite{diaconis_comparison_1993,quastel_diffusion_1992}]
    \label{lemma:spectralgap_SSEP}
    Recall the notations introduced in \cref{lemma:equivensembles_bernoulli}. There exists a universal constant $C_8>0$ such that 
    \begin{equation}\label{eq:spectralgap}
        \mathrm{Var}_{\pi_\rho}(g)\le C_8\ell^2 \sum_{x=0}^{\ell -2} \int_{\{0,1\}^{\Z}} \big[ g(\xi^{x,x+1}) - g(\xi )\big]^2\diff\pi_{\rho}(\xi )
    \end{equation}
    for any local function $g:\{0,1\}^{\Z}\longrightarrow\R$ supported in $\{0,\hdots ,\ell -1\}$ for some~$\ell\ge 1$, and such that~$\varphi_g(\rho )=0$ for any $\rho\in [0,1]$.
\end{lemma}

Consider a local function $f:\Omega\longrightarrow\R$ that depends only on coordinates in $\{0,\hdots ,\ell -1\}$, and such that $\psi_f(\rho ,\mathcal{E})=0$ for any $\rho\in [0,1]$ and $\mathcal{E}\in [\rho ,\kappa\rho]$. Let $g\coloneq \E_{\rho ,\mathcal{E}}[f|\xi ]$ be the conditional expectation of $f$ with respect to the random variable $\xi$ under $\mu_{\rho ,\mathcal{E}}$, and define also $h=f-g$.  Immediately,
\begin{equation*}
    \mathrm{Var}_{\rho ,\mathcal{E}}(f) = \mathrm{Var}_{\rho ,\mathcal{E}}(g) + \mathrm{Var}_{\rho ,\mathcal{E}}(h).
\end{equation*}
As under $\mu_{\rho ,\mathcal{E}}$, the random variable $\xi$ is distributed according to the Bernoulli product measure~$\pi_\rho$, and as $\varphi_g(\rho )=\psi_f(\rho ,\mathcal{E})=0$ for any $\rho \in [0,1]$, we can apply the spectral gap estimate in \cref{lemma:spectralgap_SSEP} to get that 
\begin{align*}
    \mathrm{Var}_{\rho ,\mathcal{E}}(g) & \le \frac{C_8\ell^2}{\kappa -1} \sum_{x=0}^{\ell -2} \int_{\Omega}\big( c_{x\to x+1}^p(\eta )+c_{x+1\to x}^p (\eta ) \big) \big[ g(\xi^{x,x+1})-g(\xi )\big]^2\diff\mu_{\rho ,\mathcal{E}}(\eta )\\ 
    & = \frac{C_8\ell^2}{\kappa -1} \sum_{\substack {x,y\in\{0,\hdots ,\ell -1\}\\ |x-y|=1}} \mathfrak{D}_{x\to y}^p (g;\mu_{\rho ,\mathcal{E}})\\
    & \le  \frac{C_8\ell^2}{\kappa -1} \sum_{\substack {x,y\in\{0,\hdots ,\ell -1\}\\ |x-y|=1}} \mathfrak{D}_{x\to y}^p (f;\mu_{\rho ,\mathcal{E}})
\end{align*}
where the last estimate comes from Jensen's inequality as the Dirichlet form is convex and $g$ is a conditional expectation of $f$. We then only have to estimate the variance of $h$. Notice that~$h(0)=0$, so we have 
\begin{equation*}
    \mathrm{Var}_{\rho ,\mathcal{E}}(h)= \sum_{r=1}^\ell \E_{\rho ,\mathcal{E}}[h^2|r]\mu_{\rho ,\mathcal{E}} (\xi_0+\hdots +\xi_{\ell -1}=r)
\end{equation*}
where $\E_{\rho ,\mathcal{E}}[\cdot |r]$ is the expectation associated with $\mu_{\rho ,\mathcal{E}}(\cdot |r)$, the measure conditioned to the event~$\{\xi_0+\hdots +\xi_{\ell -1}=r\}$. Therefore, we only need to estimate $\E_{\rho ,\mathcal{E}}[h^2|r]$ for any $1\le r\le \ell$. For a configuration $\eta\in\Omega$, we denote by $\mathrm{Part}(\eta )\coloneq \big\{ 0\le x\le \ell -1 : \xi_x(\eta )=1\big\}$ the set that contains the coordinates of the particles in the box $\{0,\hdots ,\ell -1\}$ in $\eta$. Recall \cref{defin:invariantmeasures}. As under the measure $\mu_{\rho ,\mathcal{E}}$, the random variables $(\eta_x-1)_{x\in\mathrm{Part}(\eta )}$ are distributed according to a Binomial product measure, we can use the mean-field spectral gap estimate stated in \cref{lemma:meanfieldspectralgap} for SEP($\kappa -1$) to bound
\begin{equation}\label{eq:aftermeanfield}
    \E_{\rho ,\mathcal{E}}[h^2|r]  \le \int_{\Omega} \frac{C_{11}}{r}\sum_{x,y\in\mathrm{Part}(\eta )}(\eta_x -1)(\kappa -\eta_y)\big[ f(\eta^{x\to y})- f(\eta )\big]^2\diff\mu_{\rho ,\mathcal{E}}(\eta |r)
\end{equation}
for some universal constant $C_{11}>0$. Here, we also used that $h(\eta^{x\to y})-h(\eta )=f(\eta^{x\to y})-f(\eta )$ since the transformation~$\eta^{x\to y}$ does not change the value of $\xi$. The term on the right-hand side of \eqref{eq:aftermeanfield} does look like a Dirichlet form of the process, but it is not as it contains long-range energy transfer terms. The goal is then to estimate it by the Dirichlet form of the process by creating a deterministic path of transformations allowed by the dynamics of the process, that enables the transfer of an energy unit from any site $x$ to any site $y$. This is the objective of the following \emph{moving energy lemma}.

\begin{lemma}[Moving energy lemma]
    \label{lemma:movingenergy}
    There exists a universal constant $C_9=C_9(\kappa )>0$ such that for any $\eta\in\Omega$ and any $x,y\in\mathrm{Part}(\eta )$ with $x<y$, such that $\eta_x>1$ and $\eta_y<\kappa$, we have 
    \begin{equation*}
        \big[ f(\eta^{x\to y})-f(\eta )\big]^2\mu_{\rho ,\mathcal{E}}(\eta |r) \le C_9\ell \sum_{\substack {z,z'\in\{0,\hdots ,\ell -1\}\\ |z-z'|=1}} \Big( \mathfrak{D}_{z\to z'}^p(f|r)+\mathfrak{D}_{z\to z'}^e(f|r)\Big).
    \end{equation*}
\end{lemma}

\begin{figure}
    \centering
    \includegraphics[scale=0.7]{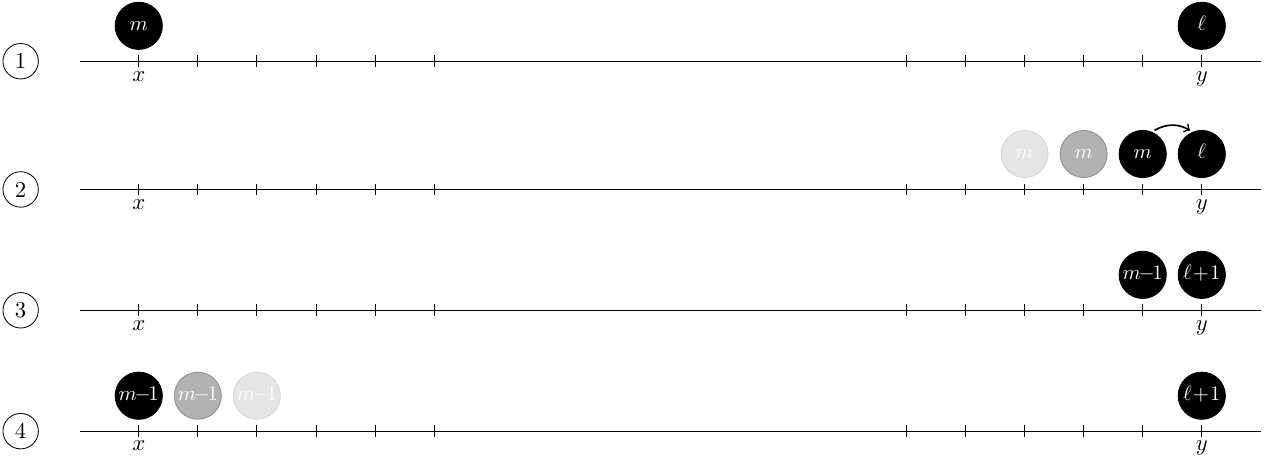}
    \label{fig:movingenergy}
    \caption{Illustation of the path of transformations used in the proof of the moving energy \cref{lemma:movingenergy}. The initial configuration $\eta$ is represented in \ding{192} where there is a particle with energy $m$ at $x$ and a particle with energy $\ell$ at $y$, with potentially some other particles in between. In step \ding{193}, we bring the particle at $x$ close to the particle at $y$ by performing nearest-neighbour exchanges. In step~\ding{194}, we transfer an energy unit from the particle at $x$ to the particle at $y$. Finally, in step \ding{195}, we bring back the particle to its initial position $y$ by performing again nearest-neighbour exchanges and we obtain the configuration $\eta^{x\to y}$.}
\end{figure}

\begin{proof}
    Take $\eta\in\Omega$ and $x,y\in\mathrm{Part}(\eta )$ as in the statement, and set $m=y-x$. As the dynamics of the process enables only energy transfer between nearest-neighbour sites, we first bring the particle at $y$ close to the particle at $x$. To do so, we define recursively a sequence of configurations by $\eta^{(0)}=\eta$ and $\eta^{(k+1)}=(\eta^{(k)})^{z_k,z_k+1}$ for any $0\le k\le m-2$, where $z_k=y-k-1$. After these nearest-neighbour exchanges, we have $(\eta_x^{(m-1)},\eta_{x+1}^{(m-1)})=(\eta_x ,\eta_y)$, \textit{i.e.}~the particle that was initially at $y$ is now at $x+1$. Then, we set $\eta^{(m)}=(\eta^{(m-1)})^{x\to x+1}$ to transfer an energy unit from the particle at $x$ to the particle initially at $y$ (which is possible under the dynamics of the process). Finally, we bring back the particle to its initial position $y$ by setting~$\eta^{(k+1)}=(\eta^{(k)})^{z_k,z_k+1}$ for any~$m\le k\le 2m-2$ where $z_k=x+k-m+1$. We illustrate this path in Figure 2. After all these exchanges, we have $\eta^{(2m-1)}=\eta^{x\to y}$. An application of Cauchy-Schwarz inequality then yields
    \begin{multline}\label{eq:CSmovingenergy}
        \big[ f(\eta^{x\to y})-f(\eta )\big]^2 \le 2m \sum_{k=0}^{m-2} \big[ f((\eta^{(k)})^{z_k,z_k+1})-f(\eta^{(k)})\big]^2  \\ + 2m \big[ f((\eta^{(m-1)})^{x\to x+1})-f(\eta^{(m-1)})\big]^2 \\ 
        + 2m \sum_{k=m}^{2m-2} \big[ f((\eta^{(k)})^{z_k,z_k+1})-f(\eta^{(k)})\big]^2.
    \end{multline}
    Note that the exchanges described here are not yet all permitted by the dynamics of the process, since they may involve swapping the positions of two consecutive particles. Even though the particle-jump dynamics of the process does not allow the exchange of two consecutive particles, the energy–transfer dynamics allows to swap their energy values, thereby effectively performing the exchange. This is the objective of the following \emph{moving particle lemma}.

    \begin{lemma}[Moving particle lemma]\label{lemma:movingparticle}
        There exists a positive constant $C_{10}=C_{10}(\kappa )$ such that for any $\eta\in\Omega$ and any $x\in\Z$, we have 
        \begin{multline*}
            \big[ f(\eta^{x,x+1})-f(\eta )\big]^2\mu_{\rho ,\mathcal{E}}(\eta |r )\\ \le C_{10}\big( \mathfrak{D}_{x\to x+1}^p (f|r) +\mathfrak{D}_{x+1\to x}^p (f|r) + \mathfrak{D}_{x\to x+1}^e (f|r) +\mathfrak{D}_{x+1\to x}^e (f|r)\big).
        \end{multline*}
    \end{lemma}
    We first complete the proof of the moving energy \cref{lemma:movingenergy} using the moving particle \cref{lemma:movingparticle}, and then we prove the latter. Multiply both sides of inequality \eqref{eq:CSmovingenergy} by $\mu_{\rho ,\mathcal{E}}(\eta |r)$. Notice that~$\mu_{\rho ,\mathcal{E}}(\eta |r) = \mu_{\rho ,\mathcal{E}}(\eta^{(k)}|r)$ for any $k\in \{0,\hdots ,m-2\}$ as these configurations are obtained from each other only performing particle exchanges, and also $\mu_{\rho ,\mathcal{E}}(\eta^{(k)}|r)=\mu_{\rho ,\mathcal{E}}(\eta^{(m-1)}|r)$ for any~$k\in \{m,\hdots ,2m-2\}$ for the same reason. However, a computation based on the reversibility relation \eqref{eq:reversibility1} shows that $\mu_{\rho ,\mathcal{E}}(\eta |r) \le (\kappa -1)^2 \mu_{\rho ,\mathcal{E}}(\eta^{(m-1)}|r)$. Thus, applying \cref{lemma:movingparticle} to each term on the right-hand side of \eqref{eq:CSmovingenergy}, one gets 
    \begin{multline*}
        \big[ f(\eta^{x\to y})-f(\eta )\big]^2\mu_{\rho ,\mathcal{E}}(\eta |r) \\ \le 2mC_{10} \sum_{k=0}^{m-2}\Big( \mathfrak{D}_{z_k\to z_k+1}^p(f|r) +\mathfrak{D}_{z_k+1\to z_k}^p(f|r)  + \mathfrak{D}_{z_k\to z_k+1}^e(f|r) +\mathfrak{D}_{z_k+1\to z_k}^e(f|r)\Big) \\
        + 2m(\kappa -1)^2 \mathfrak{D}_{x\to x+1}^e(f|r)\\ 
        + 2m (\kappa -1)^2C_{10}\sum_{k=m}^{2m-2}\Big( \mathfrak{D}_{z_k\to z_k+1}^p(f|r) +\mathfrak{D}_{z_k+1\to z_k}^p(f|r)  + \mathfrak{D}_{z_k\to z_k+1}^e(f|r) +\mathfrak{D}_{z_k+1\to z_k}^e(f|r)\Big).
    \end{multline*}
    As $m\le \ell$, we can bound the right-hand side by $C_9\ell$ times the total Dirichlet form in the box~$\{0,\hdots ,\ell -1\}$ for some constant $C_9=C_9(\kappa )>0$, which concludes the proof of \cref{lemma:movingenergy}.
\end{proof}

\begin{proof}[Proof of \cref{lemma:movingparticle}]
    Notice that $f(\eta^{x,x+1})-f(\eta ) =0$ if $\eta_x=\eta_{x+1}$ so we only need to consider the case where $\eta_x\neq\eta_{x+1}$, and without loss of generality we assume that $\eta_x>\eta_{x+1}$, the converse case being treated similarly. First, write that
    \begin{multline}\label{eq:movingparticle1}
        \big[ f(\eta^{x,x+1})-f(\eta )\big]^2 = \big[ \xi_x(1-\xi_{x+1})+(1-\xi_x)\xi_{x+1}\big]\big[ f(\eta^{x,x+1})-f(\eta )\big]^2  \\ + \xi_x\xi_{x+1}\big[ f(\eta^{x,x+1})-f(\eta )\big]^2.
    \end{multline}
    The first term on the right-hand side of this inequality is equal to 
    \begin{equation*}       
        \frac{1}{\kappa -1}\big( c_{x\to x+1}^p(\eta )+c_{x+1\to x}^p(\eta )\big)\big[ f(\eta^{x,x+1})-f(\eta )\big]^2 
    \end{equation*}
    so we focus on the second term. When both sites $x$ and $x+1$ are occupied by a particle, it is possible to exchange the values of $\eta_x$ and $\eta_{x+1}$ by using the energy exchange dynamics. More precisely, we define a sequence $(\eta^{(k)})_{0\le k\le n}$ recursively by setting $\eta^{(0)}=\eta$, and at any step,~$\eta^{(k+1)}$ is obtained from $\eta^{(k)}$ by performing the transformation $x\to x+1$ until we indeed have $\eta^{(n)}=\eta^{x,x+1}$. This is illustrated in Figure 3. To do so, we need $n = \eta_x-\eta_{x+1}\le\kappa$ steps. Using Cauchy-Schwarz inequality, and using the fact that $(\eta_x^{(k)}-1)(\kappa -\eta_{x+1}^{(k)})\ge 1$ at any step~$k$, we can bound
    \begin{equation*}
        \xi_x\xi_{x+1}\big[ f(\eta^{x,x+1})-f(\eta )\big]^2 \le \kappa\sum_{k=0}^{n-1}c_{x\to x+1}^e(\eta )\big[ f((\eta^{(k)})^{x\to x+1})-f(\eta^{(k)})\big]^2.
    \end{equation*}
    Multiply both sides of the inequality \eqref{eq:movingparticle1} by $\mu_{\rho ,\mathcal{E}}(\eta |r)$, and use the fact that the integrand of the Dirichlet form is non-negative to get that
    \begin{multline}\label{eq:movingparticle2}
        \big[ f(\eta^{x,x+1})-f(\eta )\big]^2 \mu_{\rho ,\mathcal{E}}(\eta |r) \le \frac{1}{\kappa -1}\big( \mathfrak{D}_{x\to x+1}^p(f|r)+\mathfrak{D}_{x+1\to x}^p(f|r)\big)\\
        +\kappa\sum_{k=0}^{n-1} c_{x\to x+1}^e(\eta )\big[ f((\eta^{(k)})^{x\to x+1})-f(\eta^{(k)})\big]^2\mu_{\rho ,\mathcal{E}}(\eta |r).
    \end{multline}
    Now, using the fact that for any $0\le k\le n(\le\kappa)$, we have $(\eta_x^{(k)},\eta_{x+1}^{(k)})=(\eta_x -k,\eta_{x+1}+k)$, a calculation based on the reversibility relation \eqref{eq:reversibility1} shows that 
    \begin{equation*}
    \mu_{\rho ,\mathcal{E}}(\eta |r) \le c \mu_{\rho ,\mathcal{E}}(\eta^{(k)}|r)  ,\qquad \mbox{ for } k =1,\hdots ,n
    \end{equation*}
    for some constant $c=c(\kappa )>0$. Using once again the fact that $n\le\kappa$, we can bound the second term on the right-hand side of \eqref{eq:movingparticle2} by $c\kappa^2\mathfrak{D}_{x\to x+1}^e(f|r)$ and get the result.
\end{proof}

\begin{figure}
    \centering
    \includegraphics[scale=.9]{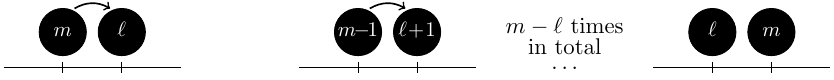}
    \label{fig:movingparticle}
    \caption{Illustration of the path of transformation used in the moving particle \cref{lemma:movingparticle} to exchange the values of $\eta_x$ and $\eta_{x+1}$ when both sites are occupied by a particle.}
\end{figure}

We are now in position to complete the proof of the spectral gap estimate in \cref{prop:spectralgap}. Recall \eqref{eq:aftermeanfield}. By \cref{lemma:movingenergy}, we get the bounds 
\begin{align*}
        \E_{\rho ,\mathcal{E}}[h^2|r] & \le \frac{C_{11}}{r}C_9\ell \left(\int_{\Omega}\sum_{x,y\in\mathrm{Part}(\eta)} (\eta_x-1)(\kappa -\eta_y)\diff\mu_{\rho ,\mathcal{E}}(\eta |r)\right) \\
        &\hspace{4cm}\times \sum_{\substack {z,z'\in\{0,\hdots ,\ell -1\}\\ |z-z'|=1}} \Big( \mathfrak{D}_{z\to z'}^p(f|r)+\mathfrak{D}_{z\to z'}^e(f|r)\Big)\\ 
        & \le \frac{C_9C_{11}\ell}{r}\kappa^2 r^2 \sum_{\substack {z,z'\in\{0,\hdots ,\ell -1\}\\ |z-z'|=1}} \Big( \mathfrak{D}_{z\to z'}^p(f|r)+\mathfrak{D}_{z\to z'}^e(f|r)\Big)
\end{align*}
which yields the result with respect to the measure $\mu_{\rho ,\mathcal{E}}(\cdot |r)$ as $r\le \ell$. Notice that the Dirichlet form with respect to $\mu_{\rho ,\mathcal{E}}(\cdot |r)$ is always zero if $r=0$. Then, putting this inequality into the decomposition of the variance of $h$, we get the desired estimate. This concludes the proof of \cref{prop:spectralgap}.\hfill\qed

\subsection{Mean-field spectral gap estimate for \texorpdfstring{SEP($\kappa$)}{kappa}}
\label{subsec:meanfield_spectral_gap}

The SEP($\kappa$) is a generalization of the simple exclusion process where each site can contain up to $\kappa$ particles instead of only one, it was first introduced in \cite{schutz_non-abelian_1994} and thoroughly studied in recent years, see for instance \cite{franceschini_non-equilibrium_2024}. During the proof of \cref{prop:spectralgap}, we needed a mean-field version of the spectral gap estimate, that we state and prove here.

\begin{lemma}[Mean-field spectral gap estimate]
    \label{lemma:meanfieldspectralgap}
    For $\alpha\in [0,1]$, define the binomial product measure $\varpi_\alpha = \mathrm{Bin}(\kappa ,\alpha )^{\otimes\Z}$. There exists a universal constant $C_{11}=C_{11}(\kappa )>0$ such that for any local function $f:\Omega\longrightarrow\R$ supported in $\{0,\hdots ,\ell -1\}$ for some $\ell\ge 1$, and such that~$\varpi_\alpha(f)=0$ for any $\alpha\in [0,1]$, we have
    \begin{equation}\label{eq:meanfieldspectralgap}
        \mathrm{Var}_{\varpi_\alpha}(f) \le \frac{C_{11}}{\ell}\sum_{x,y=0}^{\ell -1}\int_{\Omega}\eta_x(\kappa -\eta_y)\big[ f(\eta^{x\to y})-f(\eta )\big]^2\diff\varpi_\alpha(\eta ).
    \end{equation}
\end{lemma}

In \cite{kanegae_analogue_2026}, the authors computed exactly the spectral gap of the SEP($\kappa$) on a general graph in terms of the spectral gap of a random walk on the same graph. Their result implies the estimate given in \cref{lemma:meanfieldspectralgap}. Here, we provide another proof relying on the general techniques developed in \cite{caputo_spectral_2008,sasada_spectral_2013}.

\begin{proof}
    It is sufficient to show inequality \eqref{eq:meanfieldspectralgap} with respect to the conditioning $\varpi_m$ of the measure~$\varpi_\alpha$ on the event $\{ \eta_0+\hdots +\eta_{\ell -1}=m\}$ for any $m\in \{0,\hdots ,\kappa\ell\}$. Indeed, the fact that~$\varpi_\alpha(f)=0$ for any $\alpha\in [0,1]$ is equivalent to $\varpi_m(f)=0$ for any $m\in \{0,\hdots ,\kappa\ell\}$, so if we prove \eqref{eq:meanfieldspectralgap} for all the measure $\varpi_m$, we obtain the result for $\varpi_\alpha$ by simply averaging over~$m$. For $\ell\ge 1$, let also $\varpi_m^\ell$ be the restriction of the measure $\varpi_m$ to the box $\{0,\hdots ,\ell -1\}$.

    Define the infinitesimal generator of the mean-field SEP($\kappa$) on the box $\{0,\hdots ,\ell -1\}$ to be the operator that acts on local functions $f: \{0,\hdots ,\kappa\}^\ell\longrightarrow\R$ through the formula 
    \begin{equation*}
        \mathbb{L}f(\eta )=\frac1\ell\sum_{x,y=0}^{\ell -1} \eta_x(\kappa -\eta_y)\big[ f(\eta^{x\to y})-f(\eta )\big]. 
    \end{equation*}
    For $\ell\ge 1$ and $0\le m\le \kappa\ell$, define 
    \begin{equation*}
        \lambda (\ell ,m) = \inf\left\{ \frac{\varpi_m(f(-\mathbb{L})f)}{\varpi_m(f^2)} : f\in L^2(\varpi_m),\; \varpi_m(f)=0\right\}\quad\mbox{ and }\quad \lambda (\ell )=\inf_{0\le m\le \kappa\ell}\lambda (\ell ,m).
    \end{equation*}
    Then, the proof of \eqref{eq:meanfieldspectralgap} amounts to showing that $\inf_{\ell\ge 2}\lambda (\ell )>0$. To do so, we will use general techniques developed in \cite{caputo_spectral_2008,sasada_spectral_2013}. More precisely, according to \cite[Theorem 7]{sasada_spectral_2013}, it suffices to show that $\lambda (2)>0$ and $\lambda^\star (3)>\frac13$ where $\lambda^\star(\ell )$ is defined similarly as $\lambda (\ell )$ but with respect to the generator $\mathbb{L}^\star$ of an auxiliary process given by
    \begin{equation*}
        \mathbb{L}^\star f(\eta )=\frac1\ell \sum_{x,y=0}^{\ell -1} \big[\varpi_m(f|\mathcal{F}_{xy})-f\big] (\eta ).
    \end{equation*}
    In this formula, $\varpi_m(f|\mathcal{F}_{xy})$ stands for the $\varpi_m$-conditional expectation of $f$ with respect to the variables $\eta_z$ for $z\neq x,y$. 

    We first prove that $\lambda^\star(3)>\frac13$. Following \cite[Section 2]{carlen_determination_2003}, we can show that 
    \begin{equation*}
        \lambda^\star (3,m)\ge \frac13\min\{2+\mu_1,2-2\mu_2\} \quad\mbox{ any }0\le m\le 3\kappa,
    \end{equation*}
    where $\mu_1$ and $\mu_2$ are respectively the smallest and largest eigenvalues different from 1 of the operator $\mathcal{K}$ that acts on functions $\phi : \{0,1,\hdots ,\kappa\}\longrightarrow\R$ through $\mathcal{K}\phi (\xi )= \varpi_m^3(\phi(\eta_1)|\eta_2=\xi)$. If we write the matrix of the operator on the basis of functions~$(\phi_i = \ind_{\{\eta =i\}})_{0\le i\le\kappa }$, we get that
    \begin{equation*}
        \mathcal{K}_{i,j} = \varpi_m^3(\eta_1=i|\eta_2=j) = \varpi_{m-j}^2(\eta_i=i) = \frac{1}{m-j+1}\ind\{0\le i\le m-j\}\ind\{0\le j\le k\}
    \end{equation*}
    for all $0\le i,j\le\kappa$. Alternatively, in matrix form,
    \begin{equation*}
      \mathcal{K} = \left(\begin{array}{c|c}
         J_{m+1} & 0 \\ 
         \hline
         0 & 0
      \end{array}\right)\in\R^{(\kappa +1)\times (\kappa +1)}\qquad\mbox{ where }\qquad J_n = \begin{pmatrix}
         \frac1n & \frac{1}{n-1} & \hdots & 1 \\
         \vdots & \vdots & \iddots & \\ 
         \vdots & \frac{1}{n-1} & (0) & \\ 
         \frac 1n & &  & 
      \end{pmatrix}\in \R^{n\times n}.
   \end{equation*}
   It follows that 
   \begin{equation*}
    \mathrm{Sp}(\mathcal{K})=\{0\}\cup\mathrm{Sp}(J_{m+1}) = \{0\}\cup\left\{ \tfrac{(-1)^{i+1}}{i} : 1\le i\le m+1\right\},
   \end{equation*}
   hence $\mu_1 = -\frac12$ and $\mu_2=\frac13$ and it proves that $\lambda^*(3)>\frac13$.

   To conclude, we show that $\lambda (2)>0$ by computing $\lambda (2,m)$ for any $m\in \{0,\hdots ,2\kappa\}$. We are reduced to the constrained state-space of configurations of the form $(\eta_1,\eta_2)=(i,m-i)$ for~$i\in \{0,\hdots ,m\}$. It suffices to look at the first coordinate, which evolves through a birth-death process with the following dynamics : it goes from $i$ to $i+1$ at rate $(m-i)(\kappa -i)$, and from $i$ to $i-1$ at rate $i(\kappa -(m-i))$. With an abuse of notation, it means that the generator $\mathbb{L}$ acts on functions $g:\{0,\hdots ,m\}\longrightarrow\R$ as follows:
   \begin{equation*}
    \forall i\in \{0,\hdots ,m\},\qquad\mathbb{L}g(i)=(m-i)(\kappa -i)\big[g(i+1)-g(i)]+i(\kappa-m+i)\big[g(i-1)-g(i)\big].
   \end{equation*}
   Consider the basis of functions $(g_n)_{0\le n\le m}$ defined for any $i\in \{0,\hdots m\}$ by $g_0(i)=1$ and for~$n\in \{1,\hdots ,m\}$ by $g_n(i)=i(i-1)\hdots (i-n+1)$ for $n\in\{1,\hdots m\}$. A calculation shows that 
   \begin{equation*}
    \mathbb{L}g_n(i)=-n(2\kappa +1-n)g_n(i)+\mbox{linear combination of }g_{n-1}(i),\hdots ,g_0(i).
   \end{equation*}
   In this basis, the matrix of $\mathbb{L}$ is triangular, and its eigenvalues are given by~$n(2\kappa +1-n)$ for any~{$0\le n\le m$}. The spectral gap is then given by $2\kappa$, which does not depend on $m$, so~{$\lambda (2)=2\kappa >0$} and it concludes the proof of \cref{lemma:meanfieldspectralgap}.
\end{proof}

\section{The second-order Boltzmann-Gibbs principle}
\label{sec:2BG-Principle}

In this section, we state and prove the main tool to prove result of the article, namely the \emph{second-order Boltzmann-Gibbs principle}, that allows to replace local functions of the process by an expression of the local averages of the conserved quantities. But before that, we need to introduce some notations.

For $x\in\Z$, define the \emph{translation operator} $\tau_x:\Omega\longrightarrow\Omega$ through $(\tau_x\eta)_y=\eta_{y+x}$ for every~$\eta\in\Omega$ and $y\in\Z$. If $f:\Omega\longrightarrow\R$ is a function, we denote by $\tau_xf:\Omega\longrightarrow\R$ the function defined by~$\tau_xf(\eta )=f(\tau_x\eta)$. If $v:\Z\longrightarrow\R$ is a discrete function, we denote simply by $\| v\|$  its $\ell^2(\Z)$-norm:
\begin{equation*}
    \| v\| = \sqrt{\sum_{x\in\Z}v(x)^2}.
\end{equation*}
For simplicity, if $f:\Omega\longrightarrow\R$ is a function, we write $f(t)=f(\eta(t))$ for any $t\ge 0$.

The result is the following.

\begin{theorem}[Second-order Boltzmann-Gibbs Principle]\label{thm:2BG-Principle}
    Let $f:\Omega\longrightarrow\R$ be a local function such that $\psi_f(\rho ,\mathcal{E})=0$ and $\nabla\psi_f(\rho ,\mathcal{E})=0$ for fixed $\rho\in [0,1]$ and $\mathcal{E}\in [\rho ,\kappa\rho]$. Then, there exists a constant $C_{12}=C_{12}(\kappa ,\rho ,\mathcal{E} ,f)>0$ such that for any $\ell\ge 1$, any $t\ge 0$ and any measurable function~$v:[0,T]\longrightarrow\ell^2(\Z )$, we have
    \begin{multline}
        \E_{\rho ,\mathcal{E}}\left[\left( \int_0^t\sum_{x\in\Z}\tau_x\left\{ f(s)-\frac12 \bar{\mathbf{u}}^\ell (s)^\top\mathrm{Hess}(\psi_f)(\rho ,\mathcal{E})\bar{\mathbf{u}}^\ell (s)\right.\right.\right. \\
        \left.\left.\left. + \frac{1}{2\ell}\mathrm{Tr}(\chi\mathrm{Hess}(\psi_f))(\rho ,\mathcal{E})\right\}v_s(x)\diff s\right)^2\right] 
        \le C_{12}\frac{\ell}{N^a}\int_0^t \|v_s\|^2\diff s.
    \end{multline}
    where we recall that $\bar{\mathbf{u}}^\ell$ is the column vector whose coordinates are the averages $\xi_0^\ell -\rho$ and $\eta_0^\ell -\mathcal{E}$ defined in \eqref{eq:averages}, and the compressibility matrix $\chi$ has been given in \eqref{eq:compressibilitymatrix}.
\end{theorem}

We prove this result by a several step process, by following the strategy of \cite{goncalves_nonlinear_2014}. Recall the notation $\E_\ell [f] = \E_{\rho ,\mathcal{E}}[f|\xi_0^\ell ,\eta_0^\ell ]$. The first step is to replace the local function $f$ by $\E_{\ell _0}[f]$ where~$\ell_0$ is the size of the support of $f$, this is the aim of the one-block estimate of \cref{lemma:1BE}. Then, we make a renormalization step (\cref{lemma:renormalization}) showing that $\E_\ell [f]$ can by replaced by $\E_{2\ell}[f]$ for any $\ell\ge\ell_0$. Using repeatedly this renormalization step, we can prove the two-blocks estimate of \cref{lemma:2BE}, stating that $\E_{\ell_0}[f]$ can be replaced by $\E_{\ell }[f]$ for any $\ell\ge \ell_0$. Finally, we replace $\E_\ell [f]$ by the expected function of the local conserved quantities averages in \cref{lemma:laststep}.

\begin{lemma}[One-block estimate]
    \label{lemma:1BE}
    Consider a local function $f:\Omega\longrightarrow\R$ whose support is included in~$\{0,\hdots ,\ell_0-1\}$ for some $\ell_0\ge 1$. There exists a constant $C_{13}=C_{13}(\kappa ,\rho ,\mathcal{E} ,f)>0$ such that for any $t\in [0,T]$ and any measurable function $v:[0,T]\longrightarrow\ell^2(\Z )$, we have 
    \begin{equation}
        \E_{\rho ,\mathcal{E}}\left[\left(\int_0^t \sum_{x\in\Z}\tau_x \big\{ f(s) - \E_{\ell_0}[f](s)\big\} v_s(x)\diff s\right)^2\right] \le C_{13}\frac{\ell_0^3}{N^a}\int_0^t \| v_s\|^2\diff s.
    \end{equation}
\end{lemma}

\begin{proof}
    Note that for any $x\in\Z$, the function $\tau_x(f-\E_{\ell_0}[f])$ is supported in $\{x,\hdots ,x+\ell_0-1\}$. Rewrite the sum over $\Z$ as a sum over blocks of size $\ell_0$, and apply Cauchy-Schwarz inequality to get that
    \begin{multline*}
        \E_{\rho ,\mathcal{E}}\left[\left(\int_0^t \sum_{x\in\Z}\tau_x \big\{ f(s) - \E_{\ell_0}[f](s)\big\} v_s(x)\diff s\right)^2\right] \\ \le \ell_0\sum_{i=0}^{\ell_0-1} \E_{\rho ,\mathcal{E}}\left[ \left(\int_0^t\sum_{x\in\Z } \tau_{\ell_0x+i}\big\{ f(s) - \E_{\ell_0}[f](s)\big\} v_s(\ell_0x+i)\diff s\right)^2\right].
    \end{multline*}
    Now, all the terms in the right-hand side of this inequality satisfy the hypotheses of \cref{cor:spectral}, which gives that this is bounded from above by
    \begin{equation*}
        \frac{C_7\ell_0^3}{N^a}\sum_{i=0}^{\ell_0-1}\sum_{x\in\Z} \int_0^t v_s(\ell_0x+i)^2 \mathrm{Var}_{\rho ,\mathcal{E}}(f-\E_{\ell_0} [f])\diff s \le C_{13}\frac{\ell_0^3}{N^a}\int_0^t \| v_s\|^2\diff s
    \end{equation*}
    where $C_{13}$ is a positive constant that depends on $\kappa$, $\rho$, $\mathcal{E}$ and $f$.
\end{proof}

\begin{lemma}[Renormalization step]
    \label{lemma:renormalization}
    Let $f:\Omega\longrightarrow\R$ be a local function whose support is included in $\{0,\hdots ,\ell_0 -1\}$ for some $\ell_0\ge 1$. There exists a constant $C_{14}=C_{14}(\kappa ,\rho ,\mathcal{E} ,f)>0$ such that for any $t\in [0,T]$, and $\ell\ge \ell_0$ and any measurable function $v:[0,T]\longrightarrow\ell^2(\Z)$, we have 
    \begin{equation}
        \label{eq:renormalization}
        \E_{\rho ,\mathcal{E}}\left[\left(\int_0^t\sum_{x\in\Z}\tau_x\big\{ \E_{\ell }[f]-\E_{2\ell}[f]\big\} (s)v_s(x)\diff s\right)^2\right] \le C_{14} \frac{\ell^\beta}{N^a}\int_0^t\| v_s\|^2\diff s,
    \end{equation}
    where
    \begin{equation*}
        \beta = \begin{cases}
            2 & \mbox{ if }\nabla\psi_f(\rho ,\mathcal{E})\neq 0,\\
            1 & \mbox{ if } \nabla\psi_f(\rho ,\mathcal{E} )=0\mbox{ and }\mathrm{Hess}(\psi_f)(\rho ,\mathcal{E})\neq 0,\\
            0 & \mbox{ if } \nabla\psi_f(\rho ,\mathcal{E})=0\mbox{ and }\mathrm{Hess}(\psi_f)(\rho ,\mathcal{E})=0.
        \end{cases}
    \end{equation*}
\end{lemma}

\begin{proof}
    We use similar ideas. For each $x\in\Z$, the function $\tau_x (\E_\ell [f]-\E_{2\ell}[f])$ is supported in the box $\{ x,\hdots ,x+2\ell -1\}$ so now we split the sum over $\Z$ into a sum over blocks of size $2\ell$ and use Cauchy-Schwarz inequality to get that the left-hand side of \eqref{eq:renormalization} is bounded above by
    \begin{equation*}
        2\ell \sum_{i=0}^{2\ell -1} \E_{\rho ,\mathcal{E}}\left[\left( \int_0^t \sum_{x\in\Z}\tau_{2\ell x+i}\big\{ \E_\ell [f]-\E_{2\ell}[f]\big\} (s)v_s(2\ell x+i)\diff s\right)^2\right].
    \end{equation*}
    Each term inside the sum satisfies the hypotheses of \cref{cor:spectral}, which thus gives that this is bounded above by 
    \begin{equation*}
        C_7\frac{(2\ell)^3}{N^a}\sum_{i=0}^{2\ell -1} \sum_{x\in\Z} \int_0^t v_s(2\ell x+i)^2\mathrm{Var}_{\rho ,\mathcal{E}}(\E_\ell [f]-\E_{2\ell }[f])\diff s
    \end{equation*}
    According to \cref{cor:equivalenceofensembles}, we have that $\mathrm{Var}_{\rho ,\mathcal{E}}(\E_\ell [f]-\E_{2\ell}[f])\le C_3\ell^{\beta -3}$, so we easily deduce the result.
\end{proof}

\noindent We can use repeatedly \cref{lemma:renormalization} to obtain the following result.

\begin{lemma}[Two-blocks estimate]
    \label{lemma:2BE}
    Let $f:\Omega\longrightarrow\R$ be a local function whose support is included in $\{0,\hdots ,\ell_0-1\}$ for some $\ell_0\ge 1$. Then, there exists a constant $C_{15}=C_{15}(\kappa ,\rho ,\mathcal{E} ,f)>0$ such that for any $t\in [0,T]$, any $\ell\ge\ell_0$ and any measurable function $v:[0,T]\longrightarrow\ell^2(\Z )$, we have 
    \begin{equation}\label{eq:2BE}
        \E_{\rho ,\mathcal{E}}\left[\left( \int_0^t \sum_{x\in\Z}\tau_x\big\{ \E_{\ell_0}[f]-\E_\ell [f]\big\}(s)v_s(x)\diff s\right)^2\right] \le C_{15}\frac{c_\ell}{N^a}\int_0^t \| v_s\|^2\diff s,
    \end{equation}
    where
        \begin{equation*}
        c_\ell = \begin{cases}
            \ell^2 & \mbox{ if }\nabla\psi_f(\rho ,\mathcal{E})\neq 0,\\
            \ell & \mbox{ if } \nabla\psi_f(\rho ,\mathcal{E} )=0\mbox{ and }\mathrm{Hess}(\psi_f)(\rho ,\mathcal{E})\neq 0,\\
            (\log\ell)^2 & \mbox{ if } \nabla\psi_f(\rho ,\mathcal{E})=0\mbox{ and }\mathrm{Hess}(\psi_f)(\rho ,\mathcal{E})=0.
        \end{cases}
    \end{equation*}
\end{lemma}

\begin{proof}
    Assume first that $\ell$ is of the form $\ell = 2^m\ell_0$ for some $m\in\N$. Then, using Minkowski inequality, we have that the left-hand side of \eqref{eq:2BE} is bounded above by
    \begin{equation*}
        \left( \sum_{i=0}^{m-1} \sqrt{\E_{\rho ,\mathcal{E}}\left[\left( \int_0^t \sum_{x\in\Z}\tau_x\big\{ \E_{2^i\ell_0}[f]-\E_{2^{i+1}\ell_0}[f]\big\} (s)v_s(x)\diff s\right)^2\right]}\right)^2
    \end{equation*}
    The renormalization step of \cref{lemma:renormalization} allows to bound each term of this sum, so that this is bounded above by
    \begin{equation*}
        \left(\sum_{i=0}^{m-1} \sqrt{C_{14}\frac{(2^{i+1}\ell_0)^\beta }{N^a}\int_0^t \|v_s\|^2\diff s}  \right)^2 \le C_{14} m^2\frac{(2^m\ell_0)^\beta}{N^a}\int_0^t \| v_s\|^2\diff s
    \end{equation*}
    and this proves the result in the case where $\ell =2^m\ell_0$ for some $m\in\N$. To generalize it for any~$\ell\ge\ell_0$, we choose an integer $m\in\N$ such that $2^m\ell_0<\ell <2^{m+1}\ell_0$ and then compare $\E_\ell [f]$ with~$\E_{2^m\ell_0}[f]$ as in the renormalization \cref{lemma:renormalization}.
\end{proof}

The final step consists in replacing $\E_\ell [f]$ by a function of the local averages of the density of particles and energy.

\begin{lemma}
    \label{lemma:laststep}
    Let $f:\Omega\longrightarrow\R$ be a local function supported in $\{0,\hdots ,\ell_0-1\}$ for some $\ell_0\ge 1$, and assume that $\psi_f(\rho ,\mathcal{E})=0$ and $\nabla\psi_f(\rho ,\mathcal{E})=0$ for fixed $\rho\in [0,1]$ and $\mathcal{E}\in [\rho ,\kappa\rho]$. Then, there exists a constant~$C_{16}=C_{16}(\kappa ,\rho ,\mathcal{E} ,f)>0$ such that for any $t\in [0,T]$, any $\ell\ge\ell_0$ and any measurable function $v:[0,T]\longrightarrow\ell^2(\Z )$, we have 
    \begin{multline}\label{eq:laststep}
        \E_{\rho ,\mathcal{E}}\left[\left( \int_0^t \sum_{x\in\Z} \tau_x\left\{ \E_\ell [f](s)-\frac12 \bar{\mathbf{u}}^\ell (s)^\top\mathrm{Hess}(\psi_f)(\rho ,\mathcal{E})\bar{\mathbf{u}}^\ell (s)\right.\right.\right. \\
        \left.\left.\left. +\frac{1}{2\ell}\mathrm{Tr}(\chi\mathrm{Hess}(\psi_f))(\rho ,\mathcal{E})\right\} v_s(x)\diff s\right)^2\right] \le \frac{C_{16}}{N^a}\int_0^t\| v_s\|^2\diff s.
    \end{multline}
\end{lemma}

\begin{proof}
    Again, we split the sum over $\Z$ into a sum over blocks of size $\ell$ and use Cauchy-Schwarz inequality to bound the expectation in \eqref{eq:laststep} by
    \begin{multline*}
        \ell\sum_{i=0}^{\ell -1} \E_{\rho ,\mathcal{E}}\left[ \left( \int_0^t \sum_{x\in\Z} \tau_{\ell x +i}\left\{ \E_\ell [f]-\frac12 \bar{\mathbf{u}}^\ell (s)\mathrm{Hess}(\psi_f)(\rho ,\mathcal{E})\bar{\mathbf{u}}^\ell (s)\right.\right.\right.\\
        \left.\left.\left. +\frac{1}{2\ell }\mathrm{Tr}(\chi\mathrm{Hess}(\psi_f))(\rho ,\mathcal{E})\right\} v_s(\ell x+i)\diff s\right)^2\right].
    \end{multline*}
    The functions have now disjoint support and satisfy the hypotheses of \cref{cor:spectral}, which gives that this is bounded above by
    \begin{multline*}
        C_7 \frac{\ell^3}{N^a}\sum_{i=0}^{\ell -1}\sum_{x\in\Z}  \int_0^t v_s(\ell x+i)^2 \mathrm{Var}_{\rho ,\mathcal{E}}\left( \E_\ell [f] -\frac12 \bar{\mathbf{u}}^\ell \mathrm{Hess}(\psi_f)(\rho ,\mathcal{E})\bar{\mathbf{u}}^\ell \right. \\\left.+ \frac{1}{2\ell }\mathrm{Tr}(\chi\mathrm{Hess}(\psi_f))(\rho ,\mathcal{E})\right)\diff s
    \end{multline*}
    and according to \cref{cor:equivalenceofensembles}, this variance is bounded above by $C_3\ell^{-3}$ so we immediately deduce the result.
\end{proof}

\section{Martingale decomposition}
\label{sec:martingaledecomposition}

Recall the definition of the fluctuations fields $\mathcal{Y}^{p,N}$ and $\mathcal{Y}^{e,N}$ defined in \eqref{def:fluctuationfields} that are càdlàg stochastic processes with values in $\schwartzprime$. In order to show the convergence to the solution to the SBE, we need to introduce \emph{Dynkin's martingale} associated to the fluctuation fields. Take~$v\in\R$ and $\mathfrak{c}\in\R$, and consider a linear combination of these fluctuations fields of the form
\begin{equation}\label{def:Z}
    \mathcal{Z}_t^N(\varphi )= \mathfrak{c}\mathcal{Y}_t^{p,N}(T_{vN^{a-1}t}\varphi) + \mathcal{Y}_t^{e,N}(T_{vN^{a-1}t}\varphi)
\end{equation}
for any $t\in [0,T]$ and any~$\varphi\in\schwartzprime$. It is well-known (\textit{cf}.~\cite[Appendix 1.5]{kipnis_scaling_1999}) that for any~$\varphi\in\schwartz$, the process defined by 
\begin{equation}\label{def:martingale}
    \mathcal{M}_t^N(\varphi) = \mathcal{Z}_t^N(\varphi )-\mathcal{Z}_0^N(\varphi )-\int_0^t (\partial_t +N^a\mathcal{L})\mathcal{Z}_s^N(\varphi )\diff s ,\qquad t\in [0,T],
\end{equation}
is a mean-zero martingale with respect to the natural filtration of $(\eta (t))_{0\le t\le T}$, and its quadratic variation is given by
\begin{equation}\label{eq:dynkinquadraticvariation}
    \langle\mathcal{M}^N(\varphi )\rangle_t = N^a \int_0^t \big( \mathcal{L}\mathcal{Z}_s^N(\varphi )^2 - 2\mathcal{Z}_s^N(\varphi )\mathcal{L}\mathcal{Z}_s^N(\varphi )\big)\diff s.
\end{equation}
We start computing the integral term appearing in the definition of Dynkin's martingale \eqref{def:martingale}. First of all, it is immediate that 
\begin{equation}\label{eq:timederiv_martingaledecomp}
    \partial_t\mathcal{Z}_s^{N}(\varphi )= - vN^{a-1}\mathcal{Z}_s^{N}(\nabla\varphi ),
\end{equation}
so let us compute the term coming from the action of the generator $\mathcal{L}$. Using the gradient conditions \eqref{eq:gradientjump} and \eqref{eq:gradientenergy}, it is not hard to see that the action of the symmetric part on the fluctuation fields is given by 
\begin{equation}
    N^a\mathcal{L}_S\mathcal{Z}_s^N(\varphi ) = (\kappa -1)N^{a-2}\mathcal{Z}_s^N(\Delta_N\varphi )
\end{equation}
where $\Delta_N$ is the discrete Laplacian defined by
\begin{equation*}
        \Delta_N\varphi\left(\frac xN\right) = N^2\left(\varphi\left(\frac{x+1}{N}\right)-2\varphi\left(\frac xN\right)+\varphi\left(\frac{x-1}{N}\right)\right).
\end{equation*}
Let us now focus on the action of the antisymmetric part. After some computations, we see that for any $x\in\Z$,
\begin{equation*}
        \mathcal{L}_A\bar\xi_x = h_{x-1}^p(\eta )-h_x^p(\eta )\qquad\mbox{ and }\qquad \mathcal{L}_A\bar\eta_x = h_{x-1}^e(\eta )-h_x^e(\eta )
\end{equation*}
where the local function $h_x^p$ and $h_x^e$ are respectively given by
\begin{equation*}
    h_x^p(\eta )= \frac{\alpha_p+\alpha_e\eta_x}{2}c_{x\to x+1}^p(\eta )+ \frac{\alpha_p+\alpha_e\eta_{x+1}}{2}c_{x+1\to x}^p(\eta )
\end{equation*}
and
\begin{equation*}
    h_x^e(\eta )= \frac{\alpha_e+\alpha_e\eta_x }{2}\eta_xc_{x\to x+1}^p(\eta )+\frac{\alpha_p+\alpha_e\eta_{x+1}}{2}\eta_{x+1}c_{x+1\to x}^p (\eta )+ \frac{\alpha_e}{2} \big( c_{x\to x+1}^e(\eta )+c_{x+1\to x}^e(\eta )\big).
\end{equation*}
Therefore, after a summation by parts, we obtain that
\begin{equation*}
    N^{a-\gamma}\mathcal{L}_A\mathcal{Z}_s^N(\varphi ) = N^{a-\gamma -\frac32} \sum_{x\in\Z} \big( \mathfrak{c}h_x^p(s)+h_x^e(s)\big)\nabla_NT_{vN^{a-1}s}\varphi\Big( \frac xN\Big)
\end{equation*}
where the discrete gradient is defined by
\begin{equation*}
    \nabla_N\varphi\left(\frac xN\right) = N\left(\varphi\left(\frac{x+1}{N}\right)-\varphi\left(\frac xN\right)\right).
\end{equation*}
Recall the definition of the macroscopic current vector $\mathbf{j}(\rho ,\mathcal{E})$ given in \eqref{def:macroscopiccurrent}, and its Jacobian matrix $\mathbf{J}(\rho ,\mathcal{E})$, both are explicitly computed in \cref{appendix:couplingmatrices}. Notice that we have $\psi_{h^p}=\mathbf{j}_1$ and~$\psi_{h^e}=\mathbf{j}_2$. Then, redefining the local functions
\begin{equation*}
    \tilde{h}_x^p(\eta )= h_x^p(\eta )-\psi_{h^p}(\rho ,\mathcal{E}) - \nabla\psi_{h^p}(\rho ,\mathcal{E})\cdot\begin{pmatrix}
            \bar\xi_x \\
            \bar\eta_x
    \end{pmatrix},
\end{equation*}
\begin{equation*}
    \tilde{h}_x^e(\eta )= h_x^e(\eta )-\psi_{h^e}(\rho ,\mathcal{E}) - \nabla\psi_{h^e}(\rho ,\mathcal{E})\cdot\begin{pmatrix}
            \bar\xi_x \\
            \bar\eta_x
    \end{pmatrix},
\end{equation*}
the term coming from the antisymmetric part can be rewritten as
\begin{multline*}
    N^{a-\gamma -\frac32} (\mathfrak{c}\mathbf{J}_{11}+\mathbf{J}_{21})\sum_{x\in\Z}\bar\xi_x(s)\nabla_NT_{vN^{a-1}s}\varphi\Big( \frac xN\Big) \\ + N^{a-\gamma -\frac32} (\mathfrak{c}\mathbf{J}_{12}+\mathbf{J}_{22})\sum_{x\in\Z}\bar\eta_x(s)\nabla_NT_{vN^{a-1}s}\varphi\Big( \frac xN\Big) \\ + N^{a-\gamma -\frac32}\sum_{x\in\Z} \big( \mathfrak{c}\tilde{h}_x^p(s)+\tilde{h}_x^e(s)\big) \nabla_NT_{vN^{a-1}s}\varphi\Big( \frac xN\Big)
\end{multline*}
because the terms involving the averages $\psi_{h^p}(\rho ,\mathcal{E})$ and $\psi_{h^e}(\rho ,\mathcal{E})$ have no contribution as they are summed against a discrete gradient. If we replace the discrete gradient $\nabla_N$ by its continuous counterpart~$\nabla$ in the first two terms above, we can cancel them out with the time derivative term \eqref{eq:timederiv_martingaledecomp} by choosing $v$ and $\mathfrak{c}$ such that
\begin{equation}\label{eq:system_vc}
    \begin{cases}
        \frac{1}{N^\gamma}(\mathfrak{c}\mathbf{J}_{11}+\mathbf{J}_{21}) = v\mathfrak{c},\\
        \frac{1}{N^\gamma}(\mathfrak{c}\mathbf{J}_{12}+\mathbf{J}_{22}) = v.
    \end{cases}
\end{equation}
According to \cref{thm:diagonalizability}, under the conditions $\rho\in (0,1)$, $\mathcal{E}\in (\rho ,\kappa\rho )$ and $\alpha_e\neq 0$, this system admits two solutions because the matrix $\mathbf{J}(\rho ,\mathcal{E})$ is diagonalizable with two distinct real eigenvalues $\lambda_1>\lambda_2$. Moreover, since $\mathbf{J}_{12}\neq 0$ in this case (\textit{cf}.~\cref{appendix:couplingmatrices}), we can normalize the associated left eigenvectors under the form
\begin{equation*}
    \begin{pmatrix}
        \mathfrak{c}_i & 1
    \end{pmatrix} \mathbf{J}= \lambda_i \begin{pmatrix}
         \mathfrak{c}_i & 1
    \end{pmatrix}\qquad\mbox{ for }i=1,2.
\end{equation*}
More precisely, the two solutions of the system \eqref{eq:system_vc} are given by $v_i=\frac{\lambda_i}{N^\gamma}$ and $\mathfrak{c}_i$ for $i=1,2$.

Define $\mathcal{Z}^{N,i}$ for $i\in \{1,2\}$ to be the linear combination of the fluctuation fields by choosing the values~$v=v_i$ and $\mathfrak{c}=\mathfrak{c}_i$ in \eqref{def:Z}, and also $\mathcal{M}^{N,i}$ the associated martingale defined in \eqref{def:martingale}. Then, we obtained the following decomposition for these martingales.
\begin{proposition}[Martingale decomposition]
    For any $i\in \{1,2\}$, any test function $\varphi\in \schwartz$ and any $t\in [0,T]$, the martingale $\mathcal{M}^{N,i}$ decomposes as
    \begin{equation}\label{eq:martdecomp}
        \mathcal{M}_t^{N,i}(\varphi ) = \mathcal{Z}_t^{N,i}(\varphi )-\mathcal{Z}_0^{N,i}(\varphi ) - \mathcal{I}_t^{N,i}(\varphi ) - \Lambda_t^{N,i} (\varphi ) + E^{N,i}_t(\varphi)
    \end{equation}
    where
    \begin{align}
        & \mathcal{I}_t^{N,i}(\varphi ) = (\kappa -1)N^{a-2}\int_0^t \mathcal{Z}_s^{N,i}(\Delta_N\varphi )\diff s, \label{eq:It}\\
        & \Lambda_t^{N,i} (\varphi ) = N^{a-\gamma -\frac32} \int_0^t \sum_{x\in\Z} \big( \mathfrak{c}_i\tilde{h}_x^p(s)+\tilde{h}_x^e(s)\big) \nabla_NT_{v_iN^{a-1}s}\varphi\Big( \frac xN\Big)\diff s, \label{eq:Lambda}
    \end{align}
    and $E_t^{N,i}(\varphi )$ is the remainder term coming from the replacement of the discrete gradient by its continuous counterpart, whose $L^2$-norm is of order $O(N^{2a-2\gamma -4})$.
\end{proposition}

Define also the matrix $\mathbf{H}^i=\mathbf{H}^i(\rho ,\mathcal{E})$ to be the Hessian matrix of the $i$-th coordinate of the macroscopic current: $\mathbf{H}^i\coloneq \mathrm{Hess}(\mathbf{j}_i)$ for $i=1,2$. Notice that we can apply the second-order Boltzmann-Gibbs principle of \cref{thm:2BG-Principle} to the term $\Lambda^{N,i}_t(\varphi )$ since the local function ~$\mathfrak{c}_i\tilde{h}^p+\tilde{h}^e$ satisfies the required hypotheses. This yields that for any $t\in [0,T]$, any $\ell\ge 1$ and any $\varphi\in\schwartz$ we can write
\begin{equation}\label{eq:decompLambda}
    \Lambda_t^{N,i}(\varphi ) = \mathscr{B}_t^{N,i,\ell} (\varphi )+ \mathcal{R}_t^{N,i,\ell} (\varphi )
\end{equation}
where
\begin{equation}\label{def:Bt}
    \mathscr{B}_t^{N,i,\ell} (\varphi ) = \frac12 N^{a-\gamma -\frac 32} \int_0^t \sum_{x\in\Z} \mathscr{Q}_x^{i,\ell} (s) \nabla_NT_{v_iN^{a-1}s}\varphi \Big( \frac xN\Big)\diff s
\end{equation}
with
\begin{equation}\label{def:Q}
    \mathscr{Q}_x^{i,\ell} (\eta ) \coloneq \tau_x \left\{ (\bar{\mathbf{u}}^\ell)^\top (\mathfrak{c}_i\mathbf{H}^1+\mathbf{H}^2)\bar{\mathbf{u}}^\ell - \frac1\ell \mathrm{Tr}\big( \chi(\mathfrak{c}_i\mathbf{H^1}+\mathbf{H}^2)\big) \right\},
\end{equation}
and $\mathcal{R}_t^{N,i,\ell}(\varphi )$ is an error term whose variance is bounded as 
\begin{equation}\label{eq:boundvarianceR}
    \E_{\rho ,\mathcal{E}}\big[ \mathcal{R}_t^{N,i,\ell} (\varphi )^2\big] \le C_{17} tN^{a-2\gamma -2}\ell \|\nabla\varphi\|_{L^2}^2
\end{equation}
for some constant $C_{17}=C_{17}(\kappa ,\rho ,\mathcal{E} ,\varphi ) >0$.

Next, we take $\ell =\varepsilon N$ for some $\varepsilon >0$ and we give an alternative expression for the quadratic term $\mathscr{B}_t^{N,i,\varepsilon}(\varphi )\coloneq \mathscr{B}_t^{N,i,\varepsilon N}(\varphi )$ in terms of the two fields $\mathcal{Z}_t^{N,1}$ and $\mathcal{Z}_t^{N,2}$, which will be useful in the sequel to identify the limit points.

\begin{proposition}[Alternative expression of the quadratic term]
    Define the transformation matrix 
    \begin{equation*}
        R = \begin{pmatrix}
             \mathfrak{c}_1 & 1 \\
             \mathfrak{c}_2 & 1
        \end{pmatrix}\qquad \Longleftrightarrow\qquad R^{-1} = \frac{1}{\mathfrak{c}_1 - \mathfrak{c}_2}\begin{pmatrix}
            1 & -1 \\
            -\mathfrak{c}_2 & \mathfrak{c}_1
        \end{pmatrix}.
    \end{equation*}
    so that $R\mathbf{J}R^{-1}=\mathrm{diag}(\lambda_1,\lambda_2)$. Define also the coupling matrices 
    \begin{equation}\label{def:couplingmatrices}
        \mathbf{G}^i = \frac12 (R^{-1})^\top (\mathfrak{c}_i\mathbf{H}^1+\mathbf{H}^2)R^{-1} 
    \end{equation}
    for $i\in\{1,2\}$. Then, for any $\varphi\in\schwartz$, any $t\in [0,T]$ and any $\varepsilon >0$, the term $\mathscr{B}_t^{N,i,\varepsilon}(\varphi )$ rewrites as 
    \begin{align}
        \mathscr{B}_t^{N,i,\varepsilon}(\varphi ) & = N^{a-\gamma -\frac32} \mathbf{G}_{11}^i\int_0^t\frac1N \sum_{x\in\Z} \mathcal{Z}_s^{N,1}\left( \iota_\varepsilon \big( \cdot\, ;\tfrac{x-(v_1-v_i)N^as}{N} \big)\right)^2 \nabla_N\varphi\Big( \frac xN\Big)\diff s \label{eq:Bt_alternative1}\\
        & \quad +N^{a-\gamma-\frac32}(\mathbf{G}_{12}^i +\mathbf{G}_{21}^i)\int_0^t\frac1N\sum_{x\in\Z} \mathcal{Z}_s^{N,1}\left(\iota_\varepsilon \big( \cdot\, ; \tfrac{x-(v_1-v_i)N^as}{N}\big)\right) \notag \\
        & \hspace{6cm}\times \mathcal{Z}_s^{N,2}\left(\iota_\varepsilon \big(\cdot\, ; \tfrac{x-(v_2-v_i)N^as}{N}\big)\right) \nabla_N\varphi\Big( \frac xN\Big)\diff s \label{eq:Bt_alternative2}\\ 
        & \quad + N^{a-\gamma -\frac32} \mathbf{G}_{22}^i\int_0^t \frac1N\sum_{x\in\Z} \mathcal{Z}_s^{N,2}\left(\iota_\varepsilon \big(\cdot\, ; \tfrac{x-(v_2-v_i)N^as}{N} \big)\right)^2 \nabla_N\varphi\Big( \frac xN\Big)\diff s \label{eq:Bt_alternative3}
    \end{align}
\end{proposition}

We don't prove this proposition in details, as it is only a matter of rewriting based on a change of basis, and on the identity $\sqrt{N}\big(\mathfrak{c}_i\bar\xi_x^{\varepsilon N}(s) + \bar\eta_x^{\varepsilon N}(s)\big) = \mathcal{Z}_s^{N,i}\big(\iota_\varepsilon \big(\cdot\, ;\tfrac xN\big)\big)$ valid for any~$x\in\Z$ and any $s\in [0,T]$, together with the fact that the term $\frac{1}{\varepsilon N}\mathrm{Tr}\big( \chi (\mathfrak{c}_i\mathbf{H}^1+\mathbf{H}^2)\big)$ vanishes when summed against a discrete gradient.

Finally, we give an expression for the quadratic variation of the martingales $\mathcal{M}^{N,i}$, and detail about its asymptotic behaviour as $N$ goes to infinity.

\begin{proposition}[Quadratic variation]\label{prop:quadraticvariation}
    For $i\in\{1,2\}$, any $\varphi\in\schwartz$ and any $t\in [0,T]$, the quadratic variation writes under the form
    \begin{equation}\label{eq:quadraticvariation}
        \langle\mathcal{M}^{N,i}(\varphi )\rangle_t = N^{a-3}\int_0^t \sum_{x\in\Z} \begin{pmatrix}
         \mathfrak{c}_i & 1
        \end{pmatrix}\Upsilon_x(s)\begin{pmatrix}
            \mathfrak{c}_i \\
            1
        \end{pmatrix} \left( \nabla_NT_{v_iN^{a-1}s}\varphi \Big(\frac xN\Big)\right)^2\diff s
    \end{equation}
    where for any $x\in\Z$, $\Upsilon_x=\Upsilon_x(\eta )$ is the $2\times 2$ matrix given by
    \begin{multline*}
        \Upsilon_x(\eta )\coloneq \left[\left(1+\frac{\alpha_p+\alpha_e\eta_x}{2N^\gamma}\right)c_{x\to x+1}^p (\eta )+ \left(1-\frac{\alpha_p+\alpha_e\eta_{x+1}}{2N^\gamma}\right)c_{x+1\to x}^p(\eta )\right]\\
        \hspace{2cm}\times\begin{pmatrix}
            (\xi_x-\xi_{x+1})^2 & (\xi_x-\xi_{x+1})(\eta_x-\eta_{x+1}) \\
            (\xi_x-\xi_{x+1})(\eta_x-\eta_{x+1}) & (\eta_x-\eta_{x+1})^2
        \end{pmatrix}\\
            +\left[ \left( 1+\frac{\alpha_e}{2}\right)c_{x\to x+1}^e(\eta )+ \left( 1-\frac{\alpha_e}{2}\right) c_{x+1\to x}^e(\eta )\right] \begin{pmatrix}
                0 & 0 \\
                0 & 1
            \end{pmatrix}.
    \end{multline*}
    Moreover,  for any $t\in [0,T]$ and any $\varphi\in\schwartz$, we have the following $L^2$-convergence
    \begin{equation}\label{eq:limitquadraticvariation}
        N^{2-a}\langle\mathcal{M}^{N,i}(\varphi )\rangle_t  \xrightarrow[N\to \infty]{L^2} \sigma_i( t )
    \end{equation}
    where $\sigma_i : [0,T]\longrightarrow\R_+$ is the function given by
    \begin{equation}\label{def:sigmai}
       \sigma_i(t) =2(\kappa -1)t \begin{pmatrix}
        \mathfrak{c}_i & 1
       \end{pmatrix} \chi(\rho ,\mathcal{E})\begin{pmatrix}
        \mathfrak{c}_i\\
        1
       \end{pmatrix} \|\nabla\varphi\|_{L^2}^2
    \end{equation}
    and $\chi$ is the compressibility matrix given in \eqref{eq:compressibilitymatrix}.
\end{proposition}

\begin{proof}
    The expression \eqref{eq:quadraticvariation} is obtained by a simple, though tedious, computation using the definition of the generator $\mathcal{L}$ and the formula \eqref{eq:dynkinquadraticvariation} of the quadratic variation. By a direct computation we get that $\E_{\rho, \mathcal{E}}\big[\Upsilon_x (\eta )\big] = 2(\kappa -1)\chi (\rho ,\mathcal{E})$, the function $\sigma_i(t)$ is thus exactly obtained as the limit 
    \begin{equation}\label{eq:limitexpectationquadraticvariation}
        \lim_{N\to +\infty} N^{2-a}\E_{\rho ,\mathcal{E}}\big[\langle\mathcal{M}^{N,i}(\varphi )\rangle_t \big] = \sigma_i(t).
    \end{equation}
    Therefore, in order to prove the convergence \eqref{eq:limitquadraticvariation}, it suffices to prove that the variance of the quadratic variation vanishes as $N$ goes to infinity. Using \eqref{eq:quadraticvariation}, Cauchy-Schwarz inequality and stationarity yield that
    \begin{multline*}
        \mathrm{Var}_{\rho ,\mathcal{E}}\big( N^{2-a}\langle\mathcal{M}^{N,i}(\varphi )\rangle_t\big)  \le \frac{t^2}{N^2} \sum_{x,y\in\Z} \nabla_NT_{v_iN^{a-1}s}\varphi\Big(\frac xN\Big)^2 \nabla_NT_{v_iN^{a-1}s}\varphi\Big(\frac yN\Big)^2 \times\\ 
        \E_{\rho ,\mathcal{E}} \left[ \begin{pmatrix}
             \mathfrak{c}_i & 1
        \end{pmatrix} \big(\Upsilon_x(\eta )-2(\kappa -1)\chi (\rho ,\mathcal{E})\big)\begin{pmatrix}
            \mathfrak{c}_i \\
            1
        \end{pmatrix} \begin{pmatrix}
            \mathfrak{c}_i & 1
        \end{pmatrix} \big(\Upsilon_y(\eta )-2(\kappa -1)\chi (\rho ,\mathcal{E})\big)\begin{pmatrix}
            \mathfrak{c}_i \\
            1
        \end{pmatrix} \right].
    \end{multline*}
    Note that the matrix $\Upsilon_x(\eta )$ is $(\eta_x,\eta_{x+1})$-measurable, so that under $\mu_{\rho ,\mathcal{E}}$, $\Upsilon_x (\eta )$ and $\Upsilon_y(\eta )$ are independent as soon as $|y-x|>1$. Therefore, the expectation above vanishes as soon as~$|y-x|>1$, and it is bounded by a constant otherwise. Hence, we get that the variance is of order $\frac{1}{N}$ as $N$ goes to infinity, which concludes the proof of the convergence \eqref{eq:limitquadraticvariation}.
\end{proof}

\section{Tightness}
\label{sec:tightness}

In this section, we prove tightness of the sequence of tempered distribution-valued processes~$\mathcal{Z}^{N,i}$. As the proof does not differ depending on the value of $i$, we omit the superscript/subscript for simplicity in the remaining of the section. To do so, the first result we rely on is \emph{Mitoma's criterion} \cite{mitoma_tightness_1983} that states that proving tightness of the sequence of càdlàg $\schwartzprime$-valued processes $\mathcal{Z}^N$ is equivalent to proving tightness of the sequence of càdlàg real-valued processes $\mathcal{Z}^N(\varphi )$ for any test function $\varphi\in\schwartz$. According to the martingale decomposition \eqref{eq:martdecomp}, $\mathcal{Z}^N(\varphi )$ can be decomposed as the sum of terms $\mathcal{I}^N(\varphi )$,~$\Lambda^{N}(\varphi )$,~$\mathcal{M}^N(\varphi )$,~$\mathcal{Z}_0^N(\varphi )$ and~$E^{N}(\varphi )$, so it suffices to prove tightness of each of these terms. Tightness for the latter two terms is rather immediate. Indeed, as the initial distribution is the product measure~$\mu_{\rho ,\mathcal{E}}$, the Central Limit Theorem yields that the initial field~$\mathcal{Z}_0^N(\varphi )$ is tight since it converges in distribution to a centered Gaussian variable with bounded variance. Meanwhile, the bound on the $L^2$-norm of the remainder term $E^N(\varphi )$ implies that this term vanishes in $L^2$ as $N\to \infty$ for a suitable choice of the parameters, so that it is also tight. Let us focus on the three other terms.

An equivalent of Prokhorov's theorem in Skorokhod spaces is well-known (\textit{cf.}~\cite[Section 4.1]{kipnis_scaling_1999}), so that to prove tightness of a sequence of a càdlàg real-valued process $(X_t^N)_{t\in [0,T]}$, it is enough to prove the following two conditions:
\begin{enumerate}[label=(\roman*)]
    \item\label{condition1tightness} For any $t\in [0,T]$ and any $\delta >0$, there exists a compact set $K(t,\delta )\subset\R$ such that for all $N\ge 1$, we have $\p (X_t^N\not\in K(t,\delta ))\le \delta$.
    \item\label{condition2tightness} For any $\delta >0$, we have 
    \begin{equation}\label{eq:conditiontightness}
        \lim_{\theta\to 0}\limsup_{N\to \infty} \p\left( \sup_{s,t\in [0,T]\colon |t-s|<\theta} |X_t^N - X_s^N| > \delta \right) = 0.
    \end{equation}
\end{enumerate}
In particular, showing \ref{condition2tightness} implies that any limiting distribution is supported on the space of continuous paths.
Condition \ref{condition1tightness} is rather immediate in our context so we focus on proving condition \ref{condition2tightness} for the three remaining terms in the martingale decomposition. We start by proving some estimates on their increment, starting with the diffusive term.

\begin{proposition}\label{prop:incrementIt}
    For any $\varphi\in\schwartz$, any $0\le s<t\le T$, the diffusive term defined in \eqref{eq:It} satisfies
    \begin{equation}
        \E_{\rho ,\mathcal{E}}\left[ \big( \mathcal{I}_t^N(\varphi ) - \mathcal{I}_s^N(\varphi )\big)^2\right] \le C_{18}N^{2a-4} (t-s)^2 \|\Delta\varphi\|_{L^2}^2
    \end{equation}
    for some constant $C_{18}=C_{18}(\kappa ,\rho ,\mathcal{E})>0$.
\end{proposition}

\begin{proof}
    Applying Cauchy-Schwarz inequality to the time integral, we get that the $L^2$-norm of the increment satisfies
    \begin{align*}
        \E_{\rho ,\mathcal{E}}\big[ \big( \mathcal{I}_t^N(\varphi )&- \mathcal{I}_s^N(\varphi )\big)^2\big]  \le (\kappa -1)^2N^{2a-4}(t-s)\int_s^t \E_{\rho ,\mathcal{E}}\left[ \mathcal{Z}_\tau^N(\Delta_N\varphi )^2\right]\diff \tau \\
        & \le (\kappa -1)^2N^{2a-4} (t-s) \int_s^t \frac1N\sum_{x\in\Z}\left[ \Delta_NT_{vN^{a-1}\tau}\varphi\Big(\frac xN\Big)\right]^2\E_{\rho ,\mathcal{E}}\big[ (\mathfrak{c}\bar\xi_x + \bar\eta_x)^2\big]\diff \tau
    \end{align*}
    where we used the fact that the measure $\mu_{\rho ,\mathcal{E}}$ is product, and invariant. As the expectation on the right-hand side is bounded by a constant, we easily deduce the result.
\end{proof}

\noindent We have a similar result for the quadratic term.

\begin{proposition}\label{prop:incrementLambda}
    For any $\varphi\in\schwartz$, any $0\le s<t\le T$ and any $\varepsilon >0$, the quadratic term defined in \eqref{eq:Lambda} satisfies
    \begin{equation}
        \E_{\rho ,\mathcal{E}}\left[ \big( \Lambda_t^N(\varphi )-\Lambda_s^N(\varphi )\big)^2\right] \le C_{19} N^{\frac32 a-2\gamma -2}(t-s)^{3/2} \|\nabla\varphi\|_{L^2}^2
    \end{equation}
    for some constant $C_{19}=C_{19}(\kappa ,\rho ,\mathcal{E})>0$. 
\end{proposition}

\begin{proof}
    Recall the expression of the quantity $\mathscr{B}_t^{N,\ell}(\varphi )$ defined in \eqref{def:Bt}. On the one hand, the second-order Boltzmann-Gibbs principle stated in \cref{thm:2BG-Principle} implies that
    \begin{equation*}
        \E_{\rho ,\mathcal{E}}\left[ \big( \Lambda_t^N(\varphi )-\Lambda_s^N(\varphi ) - (\mathscr{B}_t^{N,\ell}(\varphi ) - \mathscr{B}_s^{N,\ell}(\varphi ) )\big)^2\right] \le C_{12} (t-s) N^{a-2\gamma -2}\ell \|\nabla\varphi\|_{L^2}^2.
    \end{equation*}
    On the other hand, a direct computation using Cauchy-Schwarz inequality and stationarity yields that 
    \begin{multline*}
        \E_{\rho ,\mathcal{E}} \left[ \big( \mathscr{B}_t^{N,\ell}(\varphi ) - \mathscr{B}_s^{N,\ell}(\varphi )\big)^2\right] \\ \le N^{2a-2\gamma -3}(t-s)\sum_{x,y\in\Z} \E_{\rho ,\mathcal{E}} \left[ \mathscr{Q}_x^\ell (\eta )\mathscr{Q}_y^\ell (\eta )\right] \int_s^t\nabla_NT_{vN^{a-1}\tau}\varphi\Big(\frac xN\Big)\nabla_NT_{vN^{a-1}\tau}\varphi\Big(\frac yN\Big) \diff\tau.
    \end{multline*}
    Note that for $|x-y|>\ell$, the expectation above vanishes since the terms $\mathscr{Q}_x^\ell$ and $\mathscr{Q}_y^\ell$ are independent and centered under the stationary measure $\mu_{\rho ,\mathcal{E}}$. Therefore, we can restrict the sum to $|x-y|\le \ell$, and an application of the inequality $2ab \le a^2+b^2$ then yields that this is bounded above by 
    \begin{equation*}
        N^{2a-2\gamma -3}(t-s)^2 \ell N\|\nabla\varphi\|_{L^2}^2 \E_{\rho ,\mathcal{E}} \big[ \mathscr{Q}_0^\ell (\eta )^2\big].
    \end{equation*}
    The term $\E_{\rho ,\mathcal{E}} \big[ \mathscr{Q}_0^\ell (\eta )^2\big]$ is the fourth moment of an empirical average of centered variables, so a basic calculation shows that it is of order $O(\ell^{-2})$. Choosing $\ell = \sqrt{t-s}N^{a/2}$ in these two estimates concludes the proof.
\end{proof}

\noindent Finally, we turn to the martingale term.

\begin{proposition}\label{prop:incrementMn}
    For any $\varphi\in\schwartz$ and any $0\le s<t\le T$, the martingale term satisfies 
    \begin{equation}
        \E_{\rho ,\mathcal{E}}\left[ \big( \mathcal{M}_t^N(\varphi )-\mathcal{M}_s^N(\varphi )\big)^2\right] \le C_{20} N^{a-2}(t-s) \|\nabla\varphi\|_{L^2}^2
    \end{equation}
    for some constant $C_{20}=C_{20}(\kappa ,\rho ,\mathcal{E})>0$.
\end{proposition}

\begin{proof}
    Recall the expression of the quadratic variation of the martingale given in \eqref{eq:quadraticvariation}. Then, the $L^2$-norm of the increment of the martingale writes
    \begin{equation*}
        \E_{\rho ,\mathcal{E}}\left[ \big( \mathcal{M}_t^N(\varphi )-\mathcal{M}_s^N(\varphi )\big)^2\right] = N^{a-3}\E_{\rho ,\mathcal{E}} \left[ \int_s^t \sum_{x\in\Z} \begin{pmatrix} 
            \mathfrak{c} & 1
        \end{pmatrix}\Upsilon_x(\tau )\begin{pmatrix}
            \mathfrak{c} \\
            1
        \end{pmatrix}
        \left(\nabla_NT_{vN^{a-1}\tau}\varphi\Big(\frac xN\Big)\right)^2\diff\tau\right].
    \end{equation*}
    The expectation of the term involving the matrix $\Upsilon_x$ is bounded by a constant depending on $\kappa$, $\rho$, and $\mathcal{E}$ so the result follows easily.
\end{proof}

Thanks to the previous estimates, we easily deduce by Markov's inequality that condition \ref{condition2tightness} holds for those three terms when $a=2$ and $\gamma =\frac12$. Therefore, we have proved the following result.

\begin{corollary}[Tightness]
    The sequences of processes $(\mathcal{I}^N)_{N\ge 1}$,~$(\Lambda^N)_{N\ge 1}$,~$(\mathcal{M}^N)_{N\ge 1}$ and hence~$(\mathcal{Z}^N)_{N\ge 1}$ are tight in the uniform topology of the Skorokhod space~$\mathcal{D}([0,T],\schwartzprime)$ provided that $a=2$ and $\gamma =\frac12$.
\end{corollary}

\section{Identification of limit points}
\label{sec:identification}

Set $a=2$ and $\gamma =\frac12$. We have just proved that the sequence $(\mathcal{Z}^{N,i},\mathcal{M}^{N,i},\mathcal{I}^{N,i},\Lambda^{N,i})_{N\ge 1}$ is tight with respect to the uniform topology of $\mathcal{D}([0,T],\schwartzprime)^4$ for $i=1,2$. Therefore, we can extract a converging subsequence (which we do not relabel for simplicity) whose limit in distribution is denoted $(\mathcal{Z}^i,\mathcal{M}^i,\mathcal{I}^i,\Lambda^i)$. Moreover, these limiting processes almost-surely have continuous trajectories thanks to the uniform topology, and to the fact that the size of the jumps of $\mathcal{Z}^{N,i}$ vanish as $N$ goes to infinity (see \eqref{eq:jumpmartingale} below). In particular, the couple~$(\mathcal{Z}^{N,1},\mathcal{Z}^{N,2})_{N\ge 1}$ is tight in $\mathcal{D}([0,T],\schwartzprime)^2$, and we can extract a subsequence converging in distribution to~$(\mathcal{Z}^1,\mathcal{Z}^2)$ in this space. Our goal is to identify the limit point $(\mathcal{Z}^1,\mathcal{Z}^2)$ as the unique (up to indistinguishability) energy solution to the two-component uncoupled SBEs given by \eqref{eq:SBE_main} in \cref{thm:main}. This solution being unique (\textit{cf.}~\cref{remark:uncoupledSBEs}), we deduce that the whole sequence $(\mathcal{Z}^{N,1},\mathcal{Z}^{N,2})_{N\ge 1}$ converges in distribution to this unique limit point. To this end, it suffices to check that each coordinate $\mathcal{Z}^i$ is an energy solution to the SBE given in \eqref{eq:SBE_main} in the sense of \cref{defin:energysolution}, that they start from independent initial conditions, that they are driven by independent $\schwartzprime$-valued Brownian motions, and that the initial condition of one is independent of the driving noise of the other. By \cref{remark:uncoupledSBEs} again, this will automatically imply that both processes $\mathcal{Z}^1$ and $\mathcal{Z}^2$ are independent.

\medskip

First, as the measure $\mu_{\rho ,\mathcal{E}}$ is product and stationary, computing the characteristic function of~$\mathcal{Z}_t^{N,i}$ at any time $t\in [0,T]$ and letting $N$ go to infinity, one can deduce that $\mathcal{Z}^i$ is a white noise with variance given by
\begin{equation*}
    \sigma_i^2 = \mathrm{Var}_{\rho ,\mathcal{E}} (\mathfrak{c}_i\xi_x+\eta_x) = \begin{pmatrix}
        \mathfrak{c}_i & 1
    \end{pmatrix} \chi (\rho ,\mathcal{E})\begin{pmatrix}
        \mathfrak{c}_i \\
        1
    \end{pmatrix}
\end{equation*}
so condition \ref{energysol_item1} in \cref{defin:energysolution} is satisfied. Similarly, consider the characteristic function of the couple~$(\mathcal{Z}_0^{1},\mathcal{Z}_0^{2})$ which is defined for any $\varphi,\psi\in\schwartz$ by
\begin{equation*}
    \Phi_N(\varphi ,\psi) = \E_{\rho ,\mathcal{E}}\left[ \exp\left( \mathbf{i} \left(  \mathcal{Z}_0^{N,1}(\varphi ) +   \mathcal{Z}_0^{N,2}(\psi )\right)\right)\right] .
\end{equation*}
By performing a similar proof as the one of the Central Limit Theorem, and relying on the identity $\mathrm{Cov}_{\rho ,\mathcal{E}} (\mathfrak{c}_1\xi_x+\eta_x,\mathfrak{c}_2\xi_x+\eta_x) = 0$ for any $x\in\Z$ (which is a corollary of the orthogonality relation \eqref{eq:orthogonalityrelation}), one can check that
\begin{equation*}
\Phi_N(u,v) \xrightarrow[N\to \infty]{} \exp\left( -\frac12 \sigma_1^2\|\varphi\|_{L^2}^2 - \frac12 \sigma_2^2\|\psi\|_{L^2}^2\right) ,
\end{equation*}
for any $\varphi ,\psi\in\schwartz$. Therefore, the initial conditions $\mathcal{Z}_0^1$ and $\mathcal{Z}_0^2$ are independent spatial white noises with respective variances~$\sigma_1^2$ and $\sigma_2^2$.

\medskip

Let us now check condition \ref{energysol_item3} of \cref{defin:energysolution} for each process. First of all, it is not hard to see that the limit~$\mathcal{I}^i$ of the diffusive term writes as
\begin{equation*}
    \mathcal{I}_t^i(\varphi ) = (\kappa -1)\int_0^t \mathcal{Z}_s^i(\Delta\varphi )\diff s
\end{equation*}
for any $t\in [0,T]$ and any $\varphi\in\schwartz$. Second, to treat the limit $\mathcal{M}^i$ of the martingale term, we use the following general result about convergence of martingales in $\mathcal{D}([0,T],\R )$. It can be found in \cite{jacod_limit_2003} but we state it here for convenience and completeness.

\begin{theorem}[About convergence of martingales, \cite{jacod_limit_2003}]
    Let $(M^N)_{N\ge 1}$ be a sequence of càdlàg real-valued martingales, and let $\langle M^N\rangle$ be the quadratic variation of $M^N$ for any $N\ge 1$. Assume that
    \begin{enumerate}
        \item For any $N\ge 1$, the process $t\longmapsto \langle M^N\rangle_t$ has almost-surely continuous trajectories ;
        \item We have
        \begin{equation*}
            \lim_{N\to \infty}\E \left[\sup_{t\in [0,T]} \big| M_t^N - M_{t^-}^N\big|\right] = 0 \; ;
        \end{equation*}
        \item For any $t\in [0,T]$, the sequence of random variables $\big( \langle M^N\rangle_t\big)_{N\ge 1}$ converges in probability to $c(t)$, where $c:[0,T]\longrightarrow\R_+$ is a deterministic continuous function.
    \end{enumerate}
    Then, $(M^N)_{N\ge 1}$ converges in distribution in the Skorokhod space $\mathcal{D}([0,T],\R )$ to a mean-zero Gaussian process $M$ which is a martingale with continuous trajectories, and whose quadratic variation is given by the function $c$.
\end{theorem}

The first point of this theorem is clearly satisfied in our context thanks to the expression of the quadratic variation given in \eqref{eq:quadraticvariation}. The second point is also satisfied since the size of the jumps of $\mathcal{M}^{N,i}(\varphi )$ is of order $N^{-3/2}$. Indeed, since when a particle jump or an energy transfer occurs, say at time $t$ through a bond $\{x,x+1\}$, $\xi_x$ or $\eta_x$ change only by plus or minus one, we have
\begin{align}
    \big| \mathcal{M}_t^{N,i}(\varphi ) - \mathcal{M}_{t^-}^{N,i}(\varphi )\big| = \big| \mathcal{Z}_t^{N,i}(\varphi ) - \mathcal{Z}_{t^-}^{N,i}(\varphi )\big| & \le \frac{C}{\sqrt{N}}\left|\varphi\Big( \frac{x+1}{N}\Big)-\varphi\Big( \frac xN\Big)\right| \notag \\
    & \le \frac{C}{N^{3/2}}\|\nabla\varphi\|_{L^\infty} \label{eq:jumpmartingale}
\end{align}
for some constant $C>0$. The third point of the theorem is directly given by Markov's inequality together with the convergence \eqref{eq:limitquadraticvariation} proved in \cref{prop:quadraticvariation}. Therefore, we deduce that for each $i\in\{1,2\}$, the limit $\mathcal{M}^i(\varphi )$ of the martingale term is a mean-zero Gaussian process with continuous trajectories, and also a martingale whose quadratic variation is given by the deterministic function $\sigma_i$ defined in \eqref{def:sigmai}. The martingale property of $\mathcal{M}^i$ holds with respect to the filtration generated by $(\mathcal{Z}^1,\mathcal{Z}^2)$. Indeed, as $\mathcal{M}^{N,i}$ is a martingale with respect to the natural filtration of $(\eta (t))_{t\in [0,T]}$, it is also a martingale with respect to the filtration generated by $(\mathcal{Z}^{N,1},\mathcal{Z}^{N,2})$. Therefore, for any $0\le t_1<\hdots <t_n <s$, if $H$ denotes a bounded continuous function of $\{\mathcal{Z}_{t_j}^{N,i}\colon 1\le j\le n, \; i=1,2\}$, we have
\begin{equation*}
    \E_{\rho ,\mathcal{E}}\big[ \mathcal{M}_t^{N,i}(\varphi )H\big] = \E_{\rho ,\mathcal{E}}\big[ \mathcal{M}_s^{N,i}(\varphi )H\big]
\end{equation*}
for any $s<t$ and any $\varphi\in\schwartz$. Letting $N$ go to infinity, noting that the limiting distribution is concentrated on continuous paths, we deduce the desired martingale property for $\mathcal{M}^i(\varphi )$.

We can check at this point that the processes $\mathcal{M}^1$ and~$\mathcal{M}^2$ are independent Brownian motions. By Lévy's characterization of Brownian motion, it suffices to check that the covariation of martingales $\mathcal{M}^1(\varphi )$ and $\mathcal{M}^2(\psi )$ vanishes for any test functions $\varphi ,\psi\in\schwartz$. By similar computations as those made in \cref{prop:quadraticvariation} and thanks to the polarization formula, for two test functions $\varphi ,\psi\in\schwartz$, we have the following expression for the covariation of the martingales (again in matrix form)
\begin{multline*}
     \big\langle \mathcal{M}^{N,1}(\varphi ),\mathcal{M}^{N,2}(\psi )\big\rangle_t \\ = N^{a-3}\int_0^t \sum_{x\in\Z} \begin{pmatrix}
        \mathfrak{c}_1 & 1
    \end{pmatrix}\Upsilon_x(s)\begin{pmatrix}
        \mathfrak{c}_2 \\
        1
    \end{pmatrix} T_{v_1N^{a-1}s}\nabla_N\varphi\Big( \frac xN\Big)T_{v_2N^{a-1}s}\nabla_N\psi\Big( \frac xN\Big)\diff s
\end{multline*}
where $\Upsilon_x$ is the matrix obtained in \cref{prop:quadraticvariation}. Since $\E_{\rho ,\mathcal{E}}\big[ \Upsilon_x(\eta )\big]= 2(\kappa -1)\chi(\rho ,\mathcal{E})$ and thanks to the orthogonality relation in \cref{thm:diagonalizability}, we have that 
\begin{equation*}
    \E_{\rho ,\mathcal{E}} \Big[ \big\langle \mathcal{M}^{N,1}(\varphi ),\mathcal{M}^{N,2}(\psi )\big\rangle_t\Big]=0
\end{equation*}
for any $N\ge 1$. Using the same techniques as those used in \cref{prop:quadraticvariation}, one can show that the variance of this covariation vanishes as $N$ goes to infinity. Therefore, we deduce that the covariation itself vanishes in the limit, which concludes the proof of the independence of the martingales. We also need to check that the initial condition $\mathcal{Z}_0^1$ is independent of the martingale~$\mathcal{M}^2$ (and likewise by swapping the indices). This is an immediate consequence of the martingale property of $\mathcal{M}^2$ with respect to the filtration $\mathcal{F}_t = \sigma (\mathcal{Z}_s^1,\mathcal{Z}_s^2 ; s\le t)$. Indeed, as~$\mathcal{M}^2$ is a martingale with deterministic linear quadratic variation, and it starts from zero, it is independent of any element of $\mathcal{F}_0$.

Let us now turn to the quadratic term $\Lambda^{N,i}$, and the first objective is to obtain an expression for its limit $\Lambda^i$ by using the decomposition \eqref{eq:decompLambda} with $\ell =\varepsilon N$. To do so, we make use of the following general result about the convergence of the crossed field terms. This theorem is adapted from the one in \cite[Section 6]{cannizzaro_equilibrium_2026}, we do not repeat the proof here but we state it for reader's convenience.

\begin{theorem}[About crossed fields, \cite{cannizzaro_equilibrium_2026}]\label{thm:crossedfields}
    Let $(\mathscr{X}^{N,1})_{N\ge 1}$ and $(\mathscr{X}^{N,2})_{N\ge 1}$ be two sequences of random fields taking valued in the Skorokhod space $\mathcal{D}([0,T],\schwartzprime)$, and defined on a same probability space. For $i=1,2$, consider also sequences $(v_i^N)_{N\ge 1}$ of real numbers. Assume that
    \begin{enumerate}[label=(\alph*)]
        \item \label{itema} For almost-every $(k_1,k_2)\in\R^2$, we have
            \begin{equation*}
                \lim_{N\to\infty} \frac{|k_1v_1^N+k_2v_2^N|}{N} = \infty .
            \end{equation*}
        \item \label{itemb} There exist $\alpha\in (0,1)$ and a constant $C>0$ such that for any $f_1,f_2\in\schwartz$, the following bounds hold
        \begin{equation*}
            \sup_{t\in [0,T]} \E\left[ \big| \mathscr{X}_t^{N,1}(f_1)\mathscr{X}_t^{N,2}(f_2)\big|\right] \le C \prod_{i=1}^2 \|\Delta f_i\|_{L^2}\vee\|\nabla f_i\|_{L^2} ,
        \end{equation*}
        \begin{equation*}
            \E \left[\big| \mathscr{X}_t^{N,1}(f_1)\mathscr{X}_t^{N,2}(f_2) - \mathscr{X}_s^{N,1}(f_1)\mathscr{X}_s^{N,2}(f_2)\big|\right] \le C (t-s)^\alpha \prod_{i=1}^2 \|\Delta f_i\|_{L^2}\vee\|\nabla f_i\|_{L^2}
        \end{equation*}
        uniformly in $N\ge 1$, for any $0\le s\le t\le T$.
    \end{enumerate}
    Then, for any $t\in [0,T]$, any $\varphi\in\schwartz$ and any~$f_1,f_2\in C_c^\infty (\R )$, we have
    \begin{equation*}
        \lim_{N\to\infty} \E\left[\left| \int_0^t \frac1N \sum_{x\in\Z} \nabla_N\varphi\left(\frac xN\right) \prod_{i=1}^2 \mathscr{X}_s^{N,i}\left( f_i\left(\cdot\, ; \frac{x+v_i^Ns}{N}\right)\right)\diff s\right|\right] =0,
    \end{equation*}
    where we recall that for $u\in\R$, we denote by $f(\cdot\, ; u)$ the function $v\longmapsto f(v-u)$.
\end{theorem}

Consider two distinct indices $i,j\in\{1,2\}$. We will apply this theorem to the terms appearing in the alternative expression of~$\mathscr{B}_t^{N,i,\varepsilon}(\varphi )$ given in \eqref{eq:Bt_alternative1}-\eqref{eq:Bt_alternative3} in order to determine an expression for its limit as $N$ is sent to infinity. Assume for now that condition \ref{itemb} of \cref{thm:crossedfields} holds for the sequences of fields $(\mathcal{Z}^{N,i})_{N\ge 1}$ and $(\mathcal{Z}^{N,j})_{N\ge 1}$, but we will verify it later. The function~$\iota_\varepsilon$ does not belong to $\schwartz$ so we need first to approximate it by a Schwartz function $\iota_{\varepsilon ,\delta}$ such that $\|\iota_{\varepsilon ,\delta} -\iota_\varepsilon\|_{L^2}\le\delta$ for some arbitrary $\delta >0$. Then, we can decompose the term $\mathscr{B}_t^{N,i,\varepsilon}(\varphi )$ as the sum of the three terms \eqref{eq:Bt_alternative1}-\eqref{eq:Bt_alternative3} with $\iota_\varepsilon$ replaced by $\iota_{\varepsilon ,\delta}$ up to a remainder term which vanishes in $L^2$ as $\delta$ goes to zero uniformly in $N\ge 1$. Then, we proceed as follows. First, the crossed term \eqref{eq:Bt_alternative2} writes as
\begin{equation*}
    (\mathbf{G}_{12}^i+\mathbf{G}_{21}^i)\int_0^t \frac1N\sum_{x\in \Z} \mathcal{Z}_s^{N,i}\big( \iota_{\varepsilon ,\delta}\big( \cdot\, ; \tfrac xN\big)\big)\mathcal{Z}_s^{N,j} \Big( \iota_{\varepsilon ,\delta}\big( \cdot\, ; \tfrac{x-(\lambda_j-\lambda_i)N^{3/2}s}{N}\big)\Big)\nabla_N\varphi\Big(\frac xN\Big)\diff s.
\end{equation*}
Therefore, since $\lambda_i\neq\lambda_j$, condition \ref{itema} of \cref{thm:crossedfields} is satisfied so this term vanishes as $N$ goes to infinity. Then, it remains to treat the two other terms \eqref{eq:Bt_alternative1} and \eqref{eq:Bt_alternative3}, whose sum writes as follows
\begin{equation*}
    \mathbf{G}_{ii}^i \mathcal{A}_t^{N,i,\varepsilon ,\delta}(\varphi ) +     \mathbf{G}_{jj}^i\int_0^t \frac1N  \sum_{x\in\Z}\mathcal{Z}_s^{N,j}\left( \iota_{\varepsilon ,\delta}\big( \cdot\, ;\tfrac{x-(\lambda_j-\lambda_i)N^{3/2}s}{N}\big)\right)^2\nabla_N\varphi\Big( \frac xN\Big)\diff s.
\end{equation*}
where
\begin{equation*}
    \mathcal{A}_t^{N,i,\varepsilon ,\delta}(\varphi )\coloneq \int_0^t \frac1N  \sum_{x\in\Z}\mathcal{Z}_s^{N,i}\left( \iota_{\varepsilon ,\delta}\big( \cdot\, ;\tfrac xN\big)\right)^2\nabla_N\varphi\Big( \frac xN\Big)\diff s.
\end{equation*}
Condition \ref{itema} is satisfied for the second term above, so it vanishes as $N$ goes to infinity by \cref{thm:crossedfields}. Therefore, sending $N\to\infty$ and $\delta\to 0$, we get that the limit $\mathscr{B}_t^{i ,\varepsilon}(\varphi )$ of the quadratic term $\mathscr{B}_t^{N,i,\varepsilon}(\varphi )$ is given by $\mathbf{G}_{ii}^i\mathcal{A}_t^{i,\varepsilon}(\varphi )$ where we recognize the term defined in \eqref{eq:A_energysol}
\begin{equation*}
    \mathcal{A}_t^{i,\varepsilon}(\varphi ) = \lim_{N\to\infty }\mathcal{A}_t^{N,i,\varepsilon}(\varphi )=\int_0^t \int_\R\mathcal{Z}_s^{i}\left( \iota_{\varepsilon }( \cdot\, ; u)\right)^2\nabla\varphi (u)\diff u\diff s.
\end{equation*}
Condition \ref{energysol_item3} of \cref{defin:energysolution} readily follows by letting $\varepsilon$ go to zero once we have shown that the sequence of fields $(\mathcal{Z}^{N,i})_{N\ge 1}$ satisfies condition \ref{energysol_item2} of \cref{defin:energysolution}, namely the energy estimate \eqref{eq:energyestimate} which ensures the convergence of $\mathcal{A}_t^{i,\varepsilon}(\varphi )$ as $\varepsilon$ goes to zero.

\medskip

Let us then check the energy estimate. Take $0<\delta <\varepsilon <1$. The decomposition \eqref{eq:decompLambda} together with the bound \eqref{eq:boundvarianceR} ensures that for any $0\le s\le t\le T$, we have
\begin{equation*}
    \E_{\rho ,\mathcal{E}}\left[\big( \mathscr{B}_t^{N,i,\varepsilon}(\varphi ) - \mathscr{B}_s^{N,i,\varepsilon}(\varphi ) - (\mathscr{B}_t^{N,i,\delta}(\varphi )-\mathscr{B}_s^{N,i,\delta}(\varphi))\big)^2\right] \le C(t-s)\varepsilon\|\nabla\varphi\|_{L^2}^2
\end{equation*}
for some constant $C>0$, so standard manipulations allow to obtain the same inequality for~$\mathcal{A}_t^{N,i,\varepsilon}(\varphi )$ instead of $\mathscr{B}_t^{N,i,\varepsilon}(\varphi )$. Then, letting $N$ go to infinity yields the desired energy estimate. This concludes the proof of both conditions \ref{energysol_item2} and \ref{energysol_item3} of \cref{defin:energysolution}.

\medskip

We can repeat these exact same arguments with the reversed processes $(\mathcal{Z}_{T-t}^{N,i})_{0\le t\le T}$ whose dynamics is driven by the adjoint $\mathcal{L}^*$ of the generator $\mathcal{L}$ is $L^2(\mu_{\rho ,\mathcal{E}})$ in order to prove item \ref{energysol_item4} of \cref{defin:energysolution}. We omit it for the sake of brevity.

\medskip

The only missing point is to verify that the sequences of fields $(\mathcal{Z}^{N,1})_{N\ge 1}$ and $(\mathcal{Z}^{N,2})_{N\ge 1}$ satisfy condition \ref{itemb} of \cref{thm:crossedfields}. This is a direct consequence of the following result together with Cauchy-Schwarz inequality.

\begin{proposition}
    There exists a constant $C_{21}=C_{21}(\kappa ,\rho ,\mathcal{E})>0$ such that for any $i\in \{1,2\}$, any $\varphi \in \schwartz$ and any $0\le s<t\le T$, we have 
    \begin{align}
        & \E_{\rho ,\mathcal{E}} \left[\sup_{0\le t\le T} \big|\mathcal{Z}_t^{N,i}(\varphi )\big|^2\right] \le C_{21}  \|\Delta\varphi\|_{L^2}^2 \vee \|\nabla\varphi\|_{L^2}^2 , \label{eq:boundZ}\\ 
        & \E_{\rho ,\mathcal{E}}\left[ \big| \mathcal{Z}_t^{N,i}(\varphi ) - \mathcal{Z}_s^{N,i}(\varphi )\big|^2\right] \le C_{21} (t-s) \|\Delta\varphi\|_{L^2}^2\vee \|\nabla\varphi\|_{L^2}^2.\label{eq:timeincrementZ}
    \end{align}
\end{proposition}

\begin{proof}
    Recall the decomposition \eqref{eq:martdecomp} of the martingale $\mathcal{M}^{N,i}$. First of all, in view of \cref{prop:incrementIt,prop:incrementLambda,prop:incrementMn}, the bound \eqref{eq:timeincrementZ} readily follows. We now prove the bound \eqref{eq:boundZ}. It suffices to prove it for each term in the decomposition \eqref{eq:martdecomp}. These bounds for the diffusive term~$\mathcal{I}_t^{N,i}(\varphi )$ and the quadratic term $\Lambda_t^{N,i}(\varphi )$ are consequences of \cref{prop:incrementIt,prop:incrementLambda} together with Kolmogorov-Chentsov theorem. In order to prove \eqref{eq:boundZ} for the martingale term $\mathcal{M}_t^{N,i}(\varphi )$, we use the Burkholder-Davis-Gundy inequality for càdlàg martingales to get
    \begin{equation*}
        \E_{\rho ,\mathcal{E}}\left[\sup_{0\le t\le T}\big|\mathcal{M}_t^{N,i}(\varphi )\big|^2\right]  \le C\left( \E_{\rho ,\mathcal{E}}\Big[ \big\langle \mathcal{M}^{N,i}(\varphi) \big\rangle_T\Big] + \E_{\rho ,\mathcal{E}} \left[ \sup_{0\le t\le T} \big|\mathcal{M}_t^{N,i}(\varphi ) - \mathcal{M}_{t^-}^{N,i}(\varphi )\big|^2\right]\right).
    \end{equation*}
    Then, applying the inequality $(a+b)^2\le 2(a^2+b^2)$ together with Doob's $L^2$ inequality yield
    \begin{align*}
        \E_{\rho ,\mathcal{E}} \left[ \sup_{0\le t\le T} \big|\mathcal{M}_t^{N,i}(\varphi ) - \mathcal{M}_{t^-}^{N,i}(\varphi )\big|^2\right] \le 2\E_{\rho ,\mathcal{E}}\left[ \sup_{0\le t\le T}\big|\mathcal{M}_t^{N,i}(\varphi) \big|^2\right] \le 2 \E_{\rho ,\mathcal{E}}\Big[ \big\langle \mathcal{M}^{N,i}(\varphi )\big\rangle_T\Big].
    \end{align*}
    Using \eqref{eq:limitexpectationquadraticvariation}, we conclude that this is bounded above by $\|\nabla\varphi\|_{L^2}^2$ up to a multiplicative constant that depends on the parameters $\kappa ,\rho ,\mathcal{E}$.
\end{proof}

\noindent This concludes the proof of \cref{thm:main}.

\appendix
\crefalias{section}{appendix}
\section{Coupling matrices and emerging universality classes}
\label{appendix:couplingmatrices}

In this appendix, we give the explicit expression of the Jacobian, Hessian and coupling matrices that are considered in this paper. A direct computation shows that the macroscopic current vector defined in \eqref{def:macroscopiccurrent} writes under the form
\begin{equation*}
    \mathbf{j}(\rho ,\mathcal{E}) = \begin{pmatrix}
        (\kappa -1)(\alpha_p\rho +\alpha_e\mathcal{E})(1-\rho )\\
        (\kappa -1)\big( \alpha_p + \alpha_e\frac{\mathcal{E}}{\rho }\big)\mathcal{E}(1-\rho ) + \alpha_e\tfrac{(\mathcal{E}-\rho )(\kappa\rho -\mathcal{E})}{\rho}
    \end{pmatrix}
\end{equation*}
for any $\rho\in (0,1)$ and $\mathcal{E}\in (\rho ,\kappa\rho )$. Hence we have the following expression of the Jacobian matrix
\begin{multline*}
    \mathbf{J}(\rho ,\mathcal{E}) =\\ \begin{pmatrix}
        (\kappa -1)\big( \alpha_p(1-2\rho )-\alpha_e \mathcal{E}\big) & \alpha_e(\kappa -1)(1-\rho )\\
        -\alpha_p (\kappa -1)\mathcal{E} -\alpha_e\big( \kappa +(\kappa -2)\frac{\mathcal{E}^2}{\rho^2}\big ) & \alpha_p(\kappa -1)(1-\rho )+\alpha_e\big( \kappa +1-2(\kappa -1)\mathcal{E}+2(\kappa -2)\frac{\mathcal{E}}{\rho }\big)
    \end{pmatrix}.
\end{multline*}
The Hessian matrices, defined by $\mathbf{H}^i = \mathrm{Hess}(\mathbf{j}_i)$ for $i=1,2$, are respectively given by
\begin{equation*}
    \mathbf{H^1}= \alpha_p(\kappa -1)\begin{pmatrix}
        -2 & 0 \\
        0 & 0
    \end{pmatrix} + \alpha_e(\kappa -1)\begin{pmatrix}
        0 & -1\\
        -1 & 0
    \end{pmatrix},
\end{equation*}
and
\begin{equation*}
    \mathbf{H}^2 = \alpha_p(\kappa -1)\begin{pmatrix}
        0 & -1\\
        -1 & 0
    \end{pmatrix} + \alpha_e(\kappa -1)\begin{pmatrix}
        0 & 0 \\
        0 & -2
    \end{pmatrix} + 2\alpha_e (\kappa -2)\begin{pmatrix}
        \frac{\mathcal{E}^2}{\rho ^3} & -\frac{\mathcal{E}}{\rho ^2}\\
        -\frac{\mathcal{E}}{\rho^2} & \frac{1}{\rho}
    \end{pmatrix}.
\end{equation*}
Thanks to \cref{thm:diagonalizability}, we know that $\mathbf{J}$ is diagonalizable with distinct real eigenvalues, so denote by $R$ the transformation matrix so that $R\mathbf{J}R^{-1}=\mathrm{diag}(\lambda_1,\lambda_2)$ where, without loss of generality, we can assume $\lambda_1>\lambda_2$. Define the coupling matrices by 
\begin{equation}\label{def:couplingmatrices_appendix}
    \mathbf{G}^i = \frac12 \sum_{j=1}^2 R_{ij}(R^{-1})^\top \mathbf{H}^j R^{-1}.
\end{equation}
We now prove some properties about the diagonal elements of the coupling matrices, which are the ones that determine the universality class of the fluctuations according to NLFH theory. We start with the case where $\alpha_e=0$.

\begin{proposition}\label{prop:couplingmatricesalpha_e0}
    Assume that $\rho\in (0,1)$, $\mathcal{E}\in (\rho ,\kappa\rho )$ and $\alpha_e=0$. Then, we have 
    \begin{equation*}
        \mathbf{G}_{11}^1 = \mathbf{G}_{22}^1=\mathbf{G}_{11}^2=0 \qquad\mbox{ while }\qquad \mathbf{G}_{22}^2\neq 0
    \end{equation*}
    Therefore, according to NLFH theory the first mode is expected to behave diffusively, while the second mode is expected to display KPZ behaviour (\textit{cf.}~the orange cell in \cref{table:classification}).
\end{proposition}

\begin{proof}
    Assume $\alpha_e=0$, hence $\alpha_p\neq 0$ as we exclude the trivial case. Then, the Jacobian matrix reduces to
    \begin{equation*}
        \mathbf{J}(\rho ,\mathcal{E}) = \alpha_p(\kappa -1)\begin{pmatrix}
            1-2\rho & 0 \\
            -\mathcal{E} & 1-\rho 
        \end{pmatrix},
    \end{equation*}
    The eigenvalues of this matrix are $\lambda_1 = \alpha_p(\kappa -1)(1-\rho )$ and $\lambda_2 = \alpha_p(\kappa -1)(1-2\rho )$, with respective associated eigenvectors $\begin{pmatrix} 1 & -\frac{\rho}{\mathcal{E}}\end{pmatrix}$ and $\begin{pmatrix} 1 & 0\end{pmatrix}$. After computation, the coupling matrices write
    \begin{equation*}
        \mathbf{G}^1= \frac{\alpha_p(\kappa-1)}{2}\begin{pmatrix}
        0 & -1\\
        -1 & 0
        \end{pmatrix} \qquad\mbox{ and }\qquad \mathbf{G}^2 = \frac{\alpha_p(\kappa-1)}{2}\begin{pmatrix}
            0 & 0 \\
            0 & -2
        \end{pmatrix}
    \end{equation*}
    so this proves the result.
\end{proof}

Let us now turn to the case where $\alpha_e\neq 0$. In this case, the coefficient $\mathbf{J}_{12}$ of the Jacobian matrix is non-zero, so $\begin{pmatrix} 1 & 0\end{pmatrix}$ is not an eigenvector. As a consequence, the eigenvectors of~$\mathbf{J}$ can be normalized under the form
\begin{equation*}
    \begin{pmatrix}
        \mathfrak{c}_i & 1
    \end{pmatrix} \mathbf{J}(\rho ,\mathcal{E}) = \lambda_i \begin{pmatrix}
        \mathfrak{c}_i & 1
    \end{pmatrix}
\end{equation*}
for $i=1,2$. Even though we do not make it explicit in the notation for simplicity, the coefficients~$\mathfrak{c}_i$ strongly depend on the parameters. Moreover, as the eigenvalues are distinct, we also have $\mathfrak{c}_1\neq\mathfrak{c}_2$. In this case, the transformation matrix write
\begin{equation*}
    R = \begin{pmatrix}
        \mathfrak{c}_1 & 1 \\
        \mathfrak{c}_2 & 1
    \end{pmatrix}\qquad\Longleftrightarrow\qquad R^{-1} = \frac{1}{\mathfrak{c}_1-\mathfrak{c}_2}\begin{pmatrix}
        1 & -1 \\
        -\mathfrak{c}_2 & \mathfrak{c}_1
    \end{pmatrix}.
\end{equation*}
After tedious computations, one can check that the coupling matrices write
\begin{multline*}
    \mathbf{G}^1 = \frac{\alpha_p(\kappa -1)}{2(\mathfrak{c}_1-\mathfrak{c}_2)}\begin{pmatrix}
        -2 & 1 \\
        1 & 0
    \end{pmatrix} + \frac{\alpha_e(\kappa -1)}{2(\mathfrak{c}_1-\mathfrak{c}_2)}\begin{pmatrix}
        2\mathfrak{c}_2 & - \mathfrak{c}_1 \\
        -\mathfrak{c}_1 & 0
    \end{pmatrix} \\ 
    + \frac{\alpha_e(\kappa -2)}{\rho (\mathfrak{c}_1-\mathfrak{c}_2)^2}\begin{pmatrix}
        \big( \frac{\mathcal{E}}{\rho} + \mathfrak{c}_2\big)^2 & -\big( \frac{\mathcal{E}}{\rho} +\mathfrak{c}_1\big)\big( \frac{\mathcal{E}}{\rho} +\mathfrak{c}_2\big) \\
        -\big( \frac{\mathcal{E}}{\rho} +\mathfrak{c}_1\big)\big( \frac{\mathcal{E}}{\rho} +\mathfrak{c}_2\big) & \big( \frac{\mathcal{E}}{\rho} +\mathfrak{c}_1\big)^2
    \end{pmatrix},
\end{multline*}

\begin{multline*}
    \mathbf{G}^2 = \frac{\alpha_p(\kappa -1)}{2(\mathfrak{c}_1-\mathfrak{c}_2)}\begin{pmatrix}
        0 & -1 \\
        -1 & 2
    \end{pmatrix} + \frac{\alpha_e(\kappa -1)}{2(\mathfrak{c}_1-\mathfrak{c}_2)}\begin{pmatrix}
        0 & \mathfrak{c}_2 \\
        \mathfrak{c}_2 & -2\mathfrak{c}_1
    \end{pmatrix} \\ 
    + \frac{\alpha_e(\kappa -2)}{\rho (\mathfrak{c}_1-\mathfrak{c}_2)^2}\begin{pmatrix}
        \big( \frac{\mathcal{E}}{\rho} + \mathfrak{c}_2\big)^2 & -\big( \frac{\mathcal{E}}{\rho} +\mathfrak{c}_1\big)\big( \frac{\mathcal{E}}{\rho} +\mathfrak{c}_2\big) \\
        -\big( \frac{\mathcal{E}}{\rho} +\mathfrak{c}_1\big)\big( \frac{\mathcal{E}}{\rho} +\mathfrak{c}_2\big) & \big( \frac{\mathcal{E}}{\rho} +\mathfrak{c}_1\big)^2
    \end{pmatrix}.
\end{multline*}

We immediately notice that some cancellations occur when $\kappa =2$, so we have to treat this case separately. In this case, we always have~$\mathbf{G}_{22}^1=\mathbf{G}_{11}^2=0$. Therefore, the universality class depends only on the self-coupling terms~$\mathbf{G}_{11}^1$ and $\mathbf{G}_{22}^2$ according to the green cells of \cref{table:classification}, as it was already predicted in \cite{cannizzaro_abc_2025}. Finally, we turn to the general case $\kappa >2$.

\begin{proposition}\label{prop:noncancellation}
    Assume that $\rho\in (0,1)$, $\mathcal{E}\in (\rho ,\kappa\rho )$, $\kappa >2$ and $\alpha_e\neq 0$. Then, the coefficients~$\mathbf{G}_{22}^1$ and $\mathbf{G}_{11}^2$ are always non-zero. Moreover, there exists some choice of parameters that makes all the diagonal entries of the coupling matrices non-zero, leading to the emergence of KPZ behaviour for both modes according to NLFH theory (\textit{cf.}~the red cell in \cref{table:classification}).
\end{proposition}

\begin{proof}
    In view of the expression of the coupling matrices, in order to show that $\mathbf{G}_{22}^1$ and $\mathbf{G}_{11}^2$ are non-zero, it is enough to show that the vector 
    \begin{equation*}
    \begin{pmatrix} -\frac{\mathcal{E}}{\rho} & 1\end{pmatrix}
    \end{equation*}
    cannot be an eigenvector of the Jacobian. If it was, then computing the product between the vector and matrix would give on the first hand that the corresponding eigenvalue should be equal to $\lambda = \mathbf{J}_{22}-\frac{\mathcal{E}}{\rho}\mathbf{J}_{12}$, and on the other hand that we should have $\mathbf{J}_{21}-\frac{\mathcal{E}}{\rho}\mathbf{J}_{11} = -\frac{\mathcal{E}}{\rho}\lambda$. By plugging the values of the coefficients of the Jacobian in this last identity, and after some simplifications, one gets that we should have
    \begin{equation*}
        \alpha_e\left( \kappa  \frac{\rho}{\mathcal{E}}-1\right)\left(\frac{\rho}{\mathcal{E}}-1\right) =0
    \end{equation*}
    and this is excluded by the choice of parameters. 
    
    Finally, the existence of some choice of parameters that makes the diagonal entries of the coupling matrices non-zero can be checked by direct computations. Indeed, by choosing for instance~$(\alpha_p,\alpha_e,\kappa ,\rho ,\mathcal{E}) = (0,1,3,\frac12 ,1)$, one finds
    \begin{equation*}
        \mathbf{G}_{11}^1 = -\frac14-\sqrt{2} \qquad\mbox{ and }\qquad \mathbf{G}_{22}^2 = -\frac14 +\sqrt{2}.
    \end{equation*}
\end{proof}

The following proposition states that, nonetheless, one can cancel one of the coefficients~$\mathbf{G}_{ii}^i$, which according to NFLH theory, gives rise to KPZ behaviour for one mode, and $\frac53$-Lévy behaviour for the second one (\textit{cf.}~purplish-blue cells in \cref{table:classification}). Without loss of generality, we focus on the coefficient $\mathbf{G}_{11}^1$.

\begin{proposition}\label{prop:vanishingG11}
    Assume that $(\alpha_p,\alpha_e)=(0,1)$ and $\kappa >2$. Then, we have $\mathbf{G}_{22}^2(\kappa ,\rho ,\mathcal{E} )> 0$ for any $\rho \in (0,1)$ and $\mathcal{E} \in (\rho,\kappa\rho)$. Moreover, if $0<\rho <1-\frac{2}{\kappa}$, there exists an energy value~$\mathcal{E}^\star = \mathcal{E}^\star (\kappa ,\rho )\in (\rho ,\kappa\rho )$ for which $\mathbf{G}_{11}^1(\kappa ,\rho ,\mathcal{E}^\star )=0$. 
\end{proposition}

\begin{proof}
    We first justify that, for fixed values of $\rho$ and $\kappa$, the function defined by
    \begin{equation*}
        \mathcal{E}\longmapsto \mathbf{G}_{11}^1(\kappa ,\rho ,\mathcal{E}) = (\kappa -1)\frac{\mathfrak{c}_2}{\mathfrak{c}_1-\mathfrak{c}_2} + \frac{\kappa -2}{\rho} \left(\frac{\frac{\mathcal{E}}{\rho}+\mathfrak{c}_2}{\mathfrak{c}_1-\mathfrak{c}_2}\right)^2
    \end{equation*}
    is a continuous function of $\mathcal{E}$, that changes sign on the interval $(\rho ,\kappa\rho )$, so that the intermediate value theorem ensures the existence of a root $\mathcal{E}^\star$ in this interval. Notice that the eigenvalues~$\lambda_1$ and $\lambda_2$ are continuous functions of $\mathcal{E}$. Since for any $i\in \{1,2\}$, we have $\mathfrak{c}_i = \frac{\lambda_i-\mathbf{J}_{22}}{\mathbf{J}_{12}}$ and the coefficient $\mathbf{J}_{12}$ does not vanish, we deduce that both $\mathbf{c}_1$ and $\mathbf{c}_2$ are continuous functions of $\mathcal{E}$. The continuity of $\mathbf{G}_{11}^1$ then follows.
    On the one hand, when $\mathcal{E}=\rho$, a direct computation shows that the left eigenvectors of the Jacobian matrix are given by 
    \begin{equation*}
        \begin{pmatrix}
        -1 & 1
        \end{pmatrix}\qquad\mbox{ and }\qquad \begin{pmatrix}
            \frac{2}{\rho -1} & 1
        \end{pmatrix},
    \end{equation*}
    hence  $\mathfrak{c}_1 =-1$ and $\mathfrak{c}_2 = \frac{2}{\rho -1}$. Therefore, we have that
    \begin{equation*}
        \mathbf{G}_{11}^1(\kappa ,\rho ,\rho )= \frac{\kappa -2-\kappa\rho}{\rho (1+\rho )}
    \end{equation*}
    which is positive as we chose $\rho <1-\frac{2}{\kappa}$. On the other hand, when $\mathcal{E}=\kappa\rho$, the left eigenvectors of the Jacobian are given by
    \begin{equation*}
        \begin{pmatrix}
           \frac{\kappa -1}{\rho -1} & 1
        \end{pmatrix}\qquad\mbox{ and }\qquad\begin{pmatrix}
           -\kappa & 1
        \end{pmatrix}.
    \end{equation*}
    Putting these values in the expression of $\mathbf{G}_{11}^1$, we find that $\mathbf{G}_{11}^1(\kappa ,\rho ,\kappa\rho )$ is always negative. 
    
    On the other hand, $\mathcal{E}\longmapsto \mathbf{G}_{22}^2(\kappa ,\rho ,\mathcal{E})$ is also continuous and if $\mathfrak{c_1}<0$, then $\mathbf{G}_{22}^2(\kappa ,\rho ,\mathcal{E}) >0$ by the explicit form of $\mathbf{G}_{22}^2$ with $(\alpha_p,\alpha_e)=(0,1)$. Notice that $(\alpha_p,\alpha_e)=(0,1)$ also implies that~$\mathbf{J}_{12}>0$. Now, since $\mathbf{J}_{12}\mathbf{J}_{21} \neq 0$, we have $\lambda_i \neq \mathbf{J}_{22}$ and so $\mathfrak{c_i} \neq 0$ for any $i\in\{1,2\}$, $\rho \in (0,1)$ and~${\mathcal{E} \in (\rho,\kappa\rho)}$. Therefore, as $\mathfrak{c_1}<0$ when $\mathcal{E}=\rho$ and $\mathcal{E}=\kappa\rho$, by the above computation, we conclude that $\mathfrak{c_1} <0$ for any $\mathcal{E} \in (\rho,\kappa\rho)$. This concludes the proof. 
\end{proof}

It is not clear whether some choice of parameters would allow to cancel $\mathbf{G}_{11}^1$ and $\mathbf{G}_{22}^2$ simultaneously, which would lead to a Lévy behaviour whose exponent is the golden ration for both modes (\textit{cf.}~\cref{table:classification}). We leave this question open for future research.

\section*{Declarations}
\subsection*{Competing interests}
The authors have no competing interests to declare that are relevant to the content of this article.
\subsection*{Data availability}
No datasets were generated or analysed during the current study.
\section*{Acknowledgments}
H.D.C. gratefully acknowledges The University of Tokyo for its hospitality, and Japan Society for the Promotion of Science for its support. Both authors would like to warmly thank Kohei Hayashi for many stimulating discussions, and Tadahisa Funaki for his helpful comment about the uniqueness of energy solution for uncoupled SBEs.

\bibliographystyle{plain}
\bibliography{biblio}
\end{document}